\documentclass{amsart}

\usepackage{amsmath,amsthm, latexsym, amssymb,mathrsfs,  tikz-cd}
\usepackage{stackengine,scalerel}
\usepackage[colorlinks]{hyperref}
\usepackage[capitalize]{cleveref}
\definecolor{dark-red}{rgb}{0.6,0,0}
\definecolor{dark-green}{rgb}{0,0.4,0}
\definecolor{medium-blue}{rgb}{0,0,0.5}
\hypersetup{
    colorlinks, linkcolor={dark-red},
    citecolor={dark-green}, urlcolor={medium-blue}
}

\newcommand{\Aut}{\mr{Aut}}

\newcommand{\Id}{\mr{Id}}

\newcommand{\GL}{\mathrm{GL}}

\newcommand{\mc}[1]{\mathcal{#1}}
\newcommand{\mbb}[1]{\mathbb{#1}}
\newcommand{\mbf}[1]{\mathbf{#1}}
\newcommand{\mr}[1]{\mathrm{#1}}
\newcommand{\mf}[1]{\mathfrak{#1}}

\newcommand{\et}{\mathrm{\acute{e}t}}

\newcommand{\Sym}{\mathrm{Sym}}

\DeclareMathOperator{\Spec}{Spec}

\DeclareMathOperator{\Gal}{Gal}

\DeclareMathOperator{\Hom}{Hom}

\newcommand{\ul}[1]{\underline{#1}}

\newcommand{\Log}{\mathrm{Log}}

\numberwithin{equation}{subsection}
\numberwithin{equation}{subsubsection}

\newcommand{\Rep}{\mathrm{Rep}}
\newcommand{\Sp}{\mathrm{Sp}}
\newcommand{\fqbar}{\overline{\mathbb{F}}_q}
\newcommand{\fq}{\mbb{F}_q}
\newcommand{\mult}{\mr{mult}}

\theoremstyle{plain}

\newtheorem{maintheorem}{Theorem}

\newtheorem*{theorem*}{Theorem}

\newtheorem{theorem}[subsubsection]{Theorem}
\newtheorem{corollary}[subsubsection]{Corollary}

\newtheorem{proposition}[subsubsection]{Proposition}
\newtheorem{lemma}[subsubsection]{Lemma}

\theoremstyle{definition}

\newtheorem{example}[subsubsection]{Example}

\newtheorem{definition}[subsubsection]{Definition}
\newtheorem{remark}[subsubsection]{Remark}

\newcommand{\Exp}{\mathrm{Exp}}

\newcommand{\lcm}{\mathrm{lcm}}
\newcommand{\bdd}{\mathrm{bdd}}

\title[Probability in $\lambda$-rings]{Random matrix statistics and zeroes of $L$-functions via probability in $\lambda$-rings}

\author{Sean Howe}

\begin{document}

\begin{abstract}
We introduce a theory of probability in $\lambda$-rings designed to efficiently describe random variables valued in multisets of complex numbers, varieties over a field, or other similar enriched settings. A key role is played by the $\sigma$-moment generating function based on the plethystic exponential, which allows us to describe distributions and argue with independence in a way that is as simple as classical probability theory. As a first application, we use this theory to obtain a concise description of the asymptotic $\sigma$-moment generating functions describing distributions of eigenvalues of Haar random matrices in compact classical groups. Beyond the theory of probability in $\lambda$-rings, the proof uses only classical invariant theory. Using our description we reprove the results of Diaconis and Shahshahani on the joint distributions of traces of powers of matrices, and we also treat symmetric groups. Next, we use Poonen's sieve to establish equidistribution results for the zeroes of $L$-functions in some natural families: simple Dirichlet characters for $\mathbb{F}_q(x)$ and the vanishing cohomology of smooth hypersurface sections. We give concise descriptions of the asymptotic $\sigma$-moment generating functions in these families, then compare them to the associated random matrix distributions. These equidistribution results are sideways in that we fix $q$ and take the degree $d$ to infinity, as opposed to Deligne equidistribution for fixed $d$ as $q \rightarrow \infty$, and the large $d$-limits are related to explicit descriptions of stable homology with twisted coefficients.
\end{abstract}

\maketitle
\tableofcontents

\section{Introduction}

In this work we first give a concise new description of the eigenvalue distributions of orthogonal, symplectic, unitary, and symmetric groups (\cref{main.random-matrix}), then use this description to compare random matrix statistics with the distributions of zeroes of function field $L$-functions in natural families parameterized by polynomials with coefficients in a finite field (Theorems \ref{main.dirichlet-characters} and \ref{main.hypersurface-L-functions}). Along the way, we reprove the results of Diaconis and Shahshahani \cite{DiaconisShahshahani.OnTheEigenvaluesOfRandomMatrices} on the joint moments of traces of powers of matrices in classical groups and give a probabilistic interpretation of generating functions that have previously arisen in stable cohomology computations. 

There is a rich literature comparing zeroes of $L$-functions to random matrix statistics, starting with the Montgomery-Odlyzko law for the spacings of critical zeroes of the Riemann zeta function. In a seminal work \cite{KatzSarnak.RandomMatricesFrobeniusEigenvaluesAndMonodromy}, Katz and Sarnak observed that Deligne's equidistribution theorem could be applied to prove precise theorems when zeroes in intervals on the critical line for the Riemann zeta function are replaced with zeroes of function field $L$-functions in natural families.  The comparisons between zeroes of $L$-functions and random matrix statistics that we give here  are ``sideways" compared to  \cite{KatzSarnak.RandomMatricesFrobeniusEigenvaluesAndMonodromy}, in that we first fix a finite field $\mathbb{F}_q$ of order $q$ and compute a limiting distribution as the polynomial degree $d$ goes infinity, then afterwards study how this limit in $d$ depends on $q$ as $q \rightarrow \infty$ to compare with random matrix statistics. By contrast, in the approach of \cite{KatzSarnak.RandomMatricesFrobeniusEigenvaluesAndMonodromy}, one would first fix $d$ and apply Deligne equidistribution to take the limit as $q \rightarrow \infty$, and then take $d\rightarrow \infty$ purely at the level of random matrices. From the perspective of stable cohomology, our result is richer since it sees, through the Grothendieck-Lefschetz theorem, the stable cohomology of families of local systems in all degrees, whereas to apply Deligne equidistribution one only needs to know this cohomology in degree zero (to compute the monodromy group; see also \cref{remark.sidways-equidistribution}). Our method of proof, on the other hand, is completely elementary and does not use any cohomology computations at all; in fact, it would be possible to state and prove most of our results on $L$-functions without even knowing the theory of \'{e}tale cohomology exists. 

The main new tool that makes our computations possible is the theory of probability in $\lambda$-rings and the accompanying $\sigma$-moment generating function. To make precise statements, we begin by describing our results on random matrices in essentially classical terms. This will lead us naturally to the theory of $\lambda$-probability.  

\subsection{Random matrices}
Let $G \subseteq \GL_n(\mbb{C})$ be a compact subgroup equipped with Haar measure, and let $X$ be the function which sends a matrix $M \in G$ to its multiset of eigenvalues. From this function, we can extract many complex-valued random variables by evaluating symmetric functions of this multiset. Concretely, let $\Lambda_n=\mathbb{Z}[t_1,\ldots, t_n]^{\Sigma_n}$ be the ring of symmetric polynomials in $n$-variables. Then, for any $f \in \Lambda_n$, we can evaluate $f$ on the multiset of $n$ eigenvalues of a matrix $M \in G$ to obtain a random variable $X_f$ on $G$.

\begin{example}\label{example.random-variables-matrices-simple}Notation as above, let $M\in G$. 
\begin{itemize}
\item For the power sum symmetric polynomials $p_i = \sum_j t_j^i$ 
\[ X_{p_i}(M)=\mr{Tr}(M^i). \]
\item For the complete symmetric polynomials $h_i$, obtained by summing all monomials of degree $i$,
    \[ \frac{1}{\det(\Id-xM)} = \sum_{i\geq 0} X_{h_i}(M)x^i. \]
\item For the elementary symmetric polynomials $e_i$, obtained by summing all monomials of degree $i$ where each variable appears at most once,
    \[ \det(\Id+xM) = \sum_{i \geq 0} X_{e_i}(M)x^i.\]
\end{itemize}
Note that we have $p_1=h_1=e_1$. In the above and below we also adopt the convention that $p_0=h_0=e_0=1$. 
\end{example}

In order to compare the distributions of these random variables across different $n$, we will use the ring of symmetric functions $\Lambda=\lim_n \Lambda_n$ (where $\Lambda_{n+1} \rightarrow \Lambda_n$ sends $t_i$ to $t_i$ for $ 1\leq i \leq n$ and sends $t_{n+1}$ to $0$). For each $i$, the symmetric polynomials $p_i$, $h_i$, and $e_i$ defined in \cref{example.random-variables-matrices-simple} are compatible under the transition maps as $n$ varies and thus define symmetric functions $p_i$, $h_i$, and $e_i$ in $\Lambda$. 

Symmetric functions can be evaluated on a multiset of $n$ complex numbers by projection to the $n$th component, so that given $f \in \Lambda$ and $X$ as above, we can still make sense of $X_f$. Then, we can encode the joint moments of all the random variables $X_f, f \in \Lambda$ simultaneously by considering the \emph{$\Lambda$-distribution $\mu_X$ of $X$}, 
\[ \mu_X: \Lambda \rightarrow \mbb{C}, f \mapsto \mathbb{E}[X_f]. \]

In classical probability, the distribution of a (real-valued\footnote{For a complex-valued random variable one should instead take joint moments with the complex conjugate. Most, but not all, of the classical random variables we encounter below are actually real-valued; for the ones that are not we will use joint moments with the complex conjugate.}) random variable $Y$ is, in many cases, determined by its moments. One can thus encode the distribution of such a random variable using the moment generating function:
\[ \mbb{E}[\exp(Yt)]=\sum_{i=0}^\infty \mbb{E}\left[\frac{Y^i}{i!} \right]t^i. \]
This is often a very concise description: for example, a Gaussian distribution of mean $0$ and variance $1$ is determined by its moment generating function $\exp(t^2)$. 

To define the analog in $\lambda$-probability, we will need to name some elements of $\Lambda$ indexed by partitions. By a partition, we mean a non-increasing sequence of non-negative integers $(\tau_1, \tau_2, \ldots)$ that is eventually zero. Given a partition $\tau$, let $m_\tau\in \Lambda$ denote the monomial symmetric function obtained by formally summing the monomials that can be obtained from $\prod t_j^{\tau_j}$ by permutating indices. Let 
\[ p_\tau=\prod p_{\tau_j}, h_\tau=\prod h_{\tau_j}, \textrm{ and } e_\tau=\prod e_{\tau_j}.\]
The $m_\tau$, $h_\tau$, and $e_\tau$ each form $\mbb{Z}$-bases for $\Lambda$ as $\tau$ varies over all partitions, and the $p_\tau$ form a $\mbb{Q}$-basis for $\Lambda_{\mbb{Q}}.$

We define the $\sigma$-moment generating function for any random variable $X$ valued in multisets of complex numbers as above by
\begin{equation}\label{eq.intro-moment-gen-expansion} \mbb{E}[\Exp_{\sigma}(Xh_1)] =\sum_{\tau \textrm{ a partition}} \mbb{E}[X_{h_\tau}] m_\tau \in \Lambda_{\mbb{C}}^\wedge. \end{equation}
Here $\Lambda_{\mbb{C}}^\wedge$ is the ring of symmetric power series obtained by completing $\Lambda_{\mbb{C}}$ for the filtration by polynomial degree. 
Since the $h_\tau$ and $m_\tau$ are each $\mbb{Z}$-bases for $\Lambda$, giving the $\sigma$-moment generating function of $X$ is equivalent to giving its $\Lambda$-distribution. 

\begin{remark}\label{remark.intro-sigma-exp}Here $\Exp_\sigma$ is the plethystic exponential which, for $R$ a $\lambda$-ring, can be evaluated on any symmetric power series $x \in \Lambda_R^{\wedge}$ with constant term zero by 
\[ \Exp_{\sigma}(x)=\sum_{i=0}^\infty h_i \circ x, \]
where $h_i \circ x$ denotes a plethystic action of $h_i$ on the element $x$. In the theory of $\lambda$-rings, $h_i \circ x$ is often written $\sigma_i(x)$. It is frequently a decategorified $i$th symmetric power, and thus a natural analog of $\frac{x^i}{i!}$ in the usual exponential $\exp(x)=\sum_{i=0}^\infty \frac{x^i}{i!}$. 
\end{remark}

\begin{remark}\label{remark.hall-inner-product}
The equivalence between $\Lambda$-distributions and $\sigma$-moment generating functions is the isomorphism from $\Lambda_\mbb{C}^\wedge$ to $\Hom_{\mbb{Z}}(\Lambda, \mbb{C})$ given by the Hall inner product $\langle \cdot, \cdot \rangle$ (this follows because $\langle m_{\tau}, h_{\gamma}\rangle = \delta_{\tau\gamma}$). 
\end{remark}

For a partition $\tau$, we write $||\tau||$ for the number of entries in $\tau$, i.e. the number of distinct variables appearing in each monomial of $m_\tau$, and $|\tau|$ for the sum of $\tau$, i.e. the degree of each monomial in $m_\tau$. For example, $||(2,2,1)||=3$ and $|(2,2,1)|=5$.

\begin{maintheorem}[see \cref{theorem.body-random-matrices}]\label{main.random-matrix}
For $n\geq 1$ and $G_n$ a compact subgroup of $\GL_n(\mbb{C})$, we view $G_n$ as a probability space using unit volume Haar measure, and we write $X_n$ for the random variable sending $M \in G_n$ to the multiset of eigenvalues of $M$. With this notation:
\begin{enumerate}
\item Let $G_n=O(n)$, the subgroup of matrices which are both real and unitary. Modulo the $m_\tau$ with $||\tau||>n$,
\[ \mbb{E}[\Exp_\sigma(X_nh_1)] \equiv \Exp_\sigma(h_2)=\prod_{1\leq i \leq j}\frac{1}{1-t_i t_j}. \]

\item If $n$ is even, let $G_n=\Sp(n)$, the subgroup of matrices which are both symplectic and unitary. Modulo the $m_\tau$ with $||\tau||>n$,
\[ \mbb{E}[\Exp_\sigma(X_nh_1)] \equiv \Exp_\sigma(e_2)=\prod_{1 \leq i < j}\frac{1}{1-t_it_j}. \]

\item Let $G_n=\Sigma_n$, the subgroup of permutation matrices.  Modulo the $m_\tau$ with $|\tau|>n$,
\begin{align*} \mbb{E}[\Exp_\sigma(X_n h_1)] \equiv \Exp_{\sigma}(\Exp_{\sigma}(h_1)-1) &= \Exp_\sigma(h_1 + h_2 + \ldots) \\&=\prod_{\substack{\textrm{non-constant}\\ \textrm{monomials $m$}\\\textrm{in $t_1, t_2, \ldots$}}}\frac{1}{1-m}.\end{align*}

\item Let $G_n=U(n)$, the subgroup of unitary matrices. Modulo the $m_{\tau_1} \overline{m}_{\tau_2}$ with $||\tau_1||>n$ or $||\tau_2||>n$,
\[ \mbb{E}[\Exp_\sigma(X_nh_1+ \overline{X}_n\overline{h}_1)] \equiv \Exp_\sigma(h_1 \overline{h}_1)=\prod_{1 \leq i, j}\frac{1}{1-t_i \overline{t}_j}. \]
Here $\overline{X}_n$ is the complex conjugate while $\overline{t}_i$, etc., is meant just as another variable with no relation to $t_i$. The left-hand side is the joint $\sigma$-moment generating function describing the joint $\Lambda$-distribution of $X_n$ and $\overline{X}_n$, 
\[ \mbb{E}[\Exp_\sigma(X_nh_1+\overline{X}_n\overline{h}_1)] = \sum_{\tau_1, \tau_2} \mbb{E}[X_{n,h_{\tau_1}}\overline{X}_{n,h_{\tau_2}}] m_{\tau_1}\overline{m}_{\tau_2}. \]
\end{enumerate}
In particular, in each of (1)-(3) (resp. (4)), as $n \rightarrow \infty$, the $\Lambda$-distributions (resp. joint $\Lambda$-distributions) converge to the unique $\Lambda$-distribution (resp. joint $\Lambda$-distribution) with the given $\sigma$-moments (resp. joint $\sigma$-moments).
\end{maintheorem}

\begin{remark}
    In case (1)-(3) of \cref{main.random-matrix}, the random variables $X_{n,f}$ for $f \in \Lambda$ are all real valued functions on $G_n$, so it is only in case (4) that we need to consider not just the $\Lambda$-distribution but the joint $\Lambda$-distribution with the complex conjugate to encode the distributions of the associated classical random variables. 
\end{remark}

\begin{remark}
By \cref{lemma.lambda-involution-even}, the standard involution $\omega$ on the ring of symmetric function swaps the asymptotic generating functions in the orthogonal and symplectic cases of \cref{main.random-matrix}. We thank Michael Magee for suggesting this. 
\end{remark}

To prove \cref{main.random-matrix}, we interpret the relevant coefficients of the moment generating functions as invariants in symmetric powers. For example, in (1)-(3), using the expansion \cref{eq.intro-moment-gen-expansion} and orthogonality of characters, one finds the coefficient of $m_\tau$ is the multiplicity of the trivial representation of $G_n$ in $\Sym^\tau \mbb{C}^n=\bigotimes_i \Sym^{\tau_i}\mbb{C}^n$. When $||m_\tau||\leq n$, for (1) and (2) the first fundamental theorem of invariant theory gives a description of these coefficients that matches the counting problem for the coefficient of $m_\tau$ in the power series we have written down, while for (3) the matching is a combinatorial exercise. The argument in (4) is similar to (1)-(2).

  In \cref{corollary.trace-distributions}, we deduce from \cref{main.random-matrix}-(1),(2), and (4) the results of Diaconis and Shahshahani \cite[Theorems 2, 4, and 6]{DiaconisShahshahani.OnTheEigenvaluesOfRandomMatrices} on the joint moments of traces of powers of matrices in classical groups. By \cref{example.random-variables-matrices-simple}, to accomplish this we need to extract the values of the $\Lambda$-distribution on power sum monomials from the $\sigma$-moment generating function. By \cref{remark.hall-inner-product}, this can be done by taking the Hall inner product of the $\sigma$-moment generating function with power sum monomials. This inner product is straightforward to compute after using the alternative expansion of the $\sigma$-exponential via power sum plethysms described in \cref{lemma.exponential-power-sum}. We note that our proof makes no use of classical, orthogonal, or symplectic Schur polynomials; in particular, this eliminates the need for Ram's theory for symplectic and orthogonal Schur polynomials used in the proof of \cite{DiaconisShahshahani.OnTheEigenvaluesOfRandomMatrices}.  
  
  For the symmetric group, \cite[Theorem 7]{DiaconisShahshahani.OnTheEigenvaluesOfRandomMatrices} computes the joint moments of the cycle counting functions. These can also be derived from the computation of the moment generating function in \cref{main.random-matrix}-(3), as we explain in \cref{ss.cycle-counting}.

\begin{remark} Since each of $h_2$, $e_2$, and $h_1 \overline{h}_1$ is a degree $2$ homogeneous polynomial, the $\Lambda$-distributions appearing in \cref{main.random-matrix}-(1), (2), and (4) should be thought of as different incarnations of Gaussians in $\lambda$-probability (recalling that a classical Gaussian of mean zero and variance $1$ has moment generating function $e^{t^2}$). This is related to the fact that, when considering distributions of traces in \cref{corollary.trace-distributions}, the traces of powers of matrices behave asymptotically as independent Gaussians. Similarly, the $\Lambda$-distribution appearing in  \cref{main.random-matrix}-(3) should be though of as a $\lambda$-incarnation of a Poisson distribution (noting that a classical Poisson distribution of mean 1 has moment generating function $e^{e^t -1}$), and this is related to the fact that the cycle counting functions behave asymptotically as independent Poisson random variables (see \cref{ss.cycle-counting}). It would be interesting to formulate a more general principle governing this relation in all of these cases. 
\end{remark}

\begin{remark}It is perhaps interesting to note that we discovered the formulas in \cref{main.random-matrix} by computing the arithmetic limits related to zeroes of $L$-functions discussed in \cref{s.intro-L-functions} below. We knew in advance that these limits should match  random matrix statistics, but it was actually the final form of these arithmetic limits that led us to the simple description of the random matrix statistics. \end{remark}

\subsection{$\lambda$-probability}
\cref{main.random-matrix} is most naturally interpreted using a theory of probability in $\lambda$-rings. Here a $\lambda$-ring is a ring $R$ equipped with a plethystic action of $\Lambda$, denoted $f\circ r$ for $f \in \Lambda$ and $r \in R$, satisfying certain natural compatibilities. A key example is the monoid ring $\mathbb{Z}[\mbb{C}]$, where $\mathbb{C}$ is a monoid under multiplication so that $[z_1][z_2]=[z_1z_2]$ and the plethystic action is determined by $h_i \circ [z]=p_i \circ [z]=[z^i].$ A function on a probability space valued in multisets of complex numbers can be viewed as a function valued in $\mathbb{Z}[\mbb{C}]$ to obtain a random variable valued in a $\lambda$-ring. One can formulate \cref{main.random-matrix} in this setting, and such a formulation is useful for making basic probabilistic arguments. 

\begin{example}\label{example.standard-rep}
In \cref{main.random-matrix}-(3), note that the permutation representation of $\Sigma_n$ can be decomposed as the direct sum of the standard irreducible representation $V_n$ of dimension $n-1$ and the trivial representation. Let $Z_n$ be the random variable on $\Sigma_n$ sending $g \in \Sigma_n$ to the eigenvalues of $g$ acting on $V_n$. If we view these multisets as elements of $\mbb{Z}[\mbb{C}]$, we have $Z_n=X_n-[1]$. Since any random variable is independent to the constant random variable $[1]$, our formalism gives
\begin{align*} \mbb{E}[\Exp_\sigma(Z_nh_1)]&=\mbb{E}[\Exp_\sigma(X_n h_1)]\mbb{E}[\Exp_\sigma(-[1]h_1)]\\
&=\mbb{E}[\Exp_\sigma(X_nh_1)] \Exp_\sigma(-h_1).\end{align*}
From \cref{main.random-matrix}-(3), we deduce that, modulo the $m_\tau$ with $||\tau||>n$,
\[ \mbb{E}[\Exp_\sigma(Z_n h_1)]\equiv\Exp_\sigma{(h_2+h_3+\ldots)}= \prod_{{\substack{\textrm{monomials $m$ in $t_1, t_2, \ldots$}\\ \textrm{with $\deg m \geq 2$ }}}}\frac{1}{1-m}. \]
\end{example}

In fact there is an even more natural formulation of the general theory and \cref{main.random-matrix}: instead of considering random variables valued in $\lambda$-rings, one can consider  abstract $\lambda$-rings of random variables equipped with an expectation functional. In particular, one can take the Grothendieck ring of representations of a compact group as a $\lambda$-ring of random variables and equip it with the expectation functional given by the integral of the trace or, equivalently, the multiplicity of the trivial representation. This is how we will understand \cref{main.random-matrix} in the body of the text (see \cref{theorem.body-random-matrices}). In particular, this perspective explains why the moment generating functions are all power series with coefficients in $\mathbb{Z}$ even though the random variables $X_f$ discussed above were all complex-valued. 

\subsection{Zeroes of $L$-functions}\label{s.intro-L-functions}
There is a large body of work in number theory relating the distributions of zeroes of $L$-functions to the distributions of eigenvalues of random matrices. The theory of $\lambda$-probability is well-suited for this type of question because of the simple form of the $\sigma$-moment generating functions in \cref{main.random-matrix}. 

We illustrate this by studying the average behavior of zeroes of $L$-functions for two natural families: simple id\`{e}le class characters for $\mathbb{F}_q(t)$, and the vanishing cohomology of smooth hypersurface sections of smooth projective varieties over $\mbb{F}_q$. As in the theory of random matrices, there are two natural $\lambda$-probability spaces that one can use for this study, one a literal function space and one more abstract.

Given a variety $S/\mbb{F}_q$ and $\ell \nmid q$ prime, a natural abstract ring of random variables is the Grothendieck ring of constructible $\ell$-adic sheaves on $S$. Indeed, the families of $L$-functions we consider all arise from $\ell$-adic local systems on such a variety $S$. However, the methods we will use in the present work depend only on a much simpler notion: in fact, we could formulate and prove most of our results without even knowing the definition of \'{e}tale cohomology. Instead, for $W(\mbb{C})$ the big Witt ring of $\mathbb{C}$, which we can think of concretely as the ring $\prod_{i=1}^\infty \mbb{C}$ with plethystic operations determined by $p_j \circ (a_1, a_2, \ldots)=(a_j, a_{2j, \ldots})$, we consider the space $C(S(\fqbar), W(\mbb{C}))$ of $W(\mbb{C})$-valued function on $S(\fqbar)$ that are invariant under the $\Gal(\fqbar/\fq)$-action (equivalently, one could think of these as functions on the closed points of $S$). This ring has a natural pointwise $\lambda$-ring structure, and an expectation functional 
\[ \mbb{E}: C(S(\fqbar), W(\mbb{C})) \rightarrow W(\mbb{C}) \]
that, for the $i$th coordinate, averages over the $\mbb{F}_{q^i}$-points $s \in S(\mbb{F}_{q^i})\subseteq S(\fqbar)$ the values $f(s)_{i/|s|}$, where $|s|$ is the size of the Galois-orbit of $s$. To understand the motivation for this expectation, we note that the additive of group $W(\mbb{C})$ is also naturally isomorphic to $1+t\mbb{C}[[t]]$, and under this isomorphism the formula for the expectation functional comes from the Euler product expansion of L-functions. 

After choosing a prime $\ell$ coprime to $q$ and an isomorphism $\overline{\mbb{Q}}_\ell \cong \mbb{C}$, it follows from \cite[Th\'{e}or\'{e}me (1.1.2)]{Laumon.TransformationDeFourierConstantsDEquationsFonctionellesEtConjectureDeWeil} that the Grothendieck ring of constructible $\overline{\mbb{Q}}_\ell$ sheaves embeds into $C(S(\fqbar), W(\mbb{C}))$ by taking the sequence of traces of powers of Frobenius at each point (or, viewing $W(\mbb{C})$ as $1+t\mbb{C}[[t]]$, by taking the characteristic power series of Frobenius). This matches the pointwise $\lambda$-ring structure on $C(S(\fqbar), W(\mbb{C}))$ with the usual $\lambda$-ring structure on $\ell$-adic sheaves where the plethystic operations $h_k \circ$ (resp. $e_k \circ$) decategorify symmetric (resp. exterior) powers of local systems. 

\subsubsection{Some characters}\label{ss.intro-some-characters}
We fix a finite field $\mbb{F}_q$ of order $q$ and a prime $\ell$ dividing $q-1$. We write $\mu_\ell \subseteq \mbb{F}_q^{\times}$ for the $\ell$th roots of unity in $\mbb{F}_q$, a cyclic group of order $\ell$, and fix a non-trivial character $\chi:\mu_\ell \rightarrow \mbb{C}^\times$. For any $\ell$th-power-free monic polynomial $f\in \mbb{F}_q[x]$, we obtain by Kummer theory a non-trivial map $\Gal(\overline{\fq(x)}/\fq(x)) \rightarrow \mu_\ell$ and, by composition with $\chi$, a non-trivial character $\chi_f$ of $\Gal(\overline{\fq(x)}/\fq(x))$. By class field theory, we can also view $\chi_f$ as an id\'{e}le class character (which can also be described explicitly in terms of $f$). The $\chi_f$ of this form are a natural family of characters: as we vary over $\ell$th-power-free monic polynomials of positive degree, they hit every equivalence class of character of order $\ell$ up to globally unramified twist exactly once. Moreover, if we vary only over degrees $d$ not divisible by $\ell$, then we obtain exactly the characters which are ramified at $\infty$. Note that, for $\kappa/\fq$ a finite extension and $f \in \kappa[x]$ an $\ell$th-power-free monic polynomial, we may also form $\chi_f$ as a character of $\Gal(\overline{\kappa(x)}/\kappa(x))$, and this is compatible with restriction. 

\begin{example}
    If $\ell=2$ and $f\in \fq[x]$ is monic square-free of odd degree $>2$, then $\chi_f$ is ramified at $\infty$ and, as in \cite[\S 11]{BergstromDiaconuPetersenWesterland.HyperellipticCurvesTheScanningMapAndMomentsOfFamiliesOfQuadraticLFunctions}, the $L$-function associated to $\chi_f$ is given by the Dirichlet series
    \[ L(s, \chi_f)=\sum_{m \in \mbb{F}_q[x] \textrm{ monic }} \binom{f}{m}|m|^{-s} \]
    where $\binom{f}{m}$ denotes the usual Legendre symbol for $\mbb{F}_q[x]$.  The zeta function $Z_{C_f}(t)$ of the hyperelliptic curve $C_f$ with affine equation $y^2=f(x)$ is of the form 
    \[ Z_{C_f}(t)=\frac{\mc{L}(t,\chi_f)}{(1-t)(1-qt)} \]
    where $\mc{L}(q^{-s},\chi_f) =L(s,\chi_f)$. In particular, by the Riemann hypothesis for curves over finite fields, the roots of $\mc{L}(t,\chi_f)$ have absolute value $q^{-1/2}$. 
\end{example}

In general, for $f \in \fqbar[x]$ monic $\ell$th-power-free, and $\chi_f$ the associated id\'{e}le class character of $\mbb{F}_q(f)(x)$ (where $\mbb{F}_q(f)$ is the extension of $\fq$ obtained by adjoining the coefficients of $f)$, there is a polynomial $\mc{L}(t,\chi_f) \in \mbb{C}[t]$ whose roots have absolute value $(\#\fq(f))^{-\frac{1}{2}}=q^{-\frac{1}{2}[\fq(f):\fq]}$ such that
\[ L(s,\chi_f) = \mc{L}(t,\chi_f)|_{t=q^{-s}}. \]
Varying over $f$ of a fixed degree, we would like to understand the distribution of the random variable sending such an $f$ to the multiset of normalized zeroes of $\mc{L}(t,\chi_f)$.  

In what follows, we note that the plethystic exponential $\Exp_\sigma$ and logarithm $\Log_\sigma$ can be used to define powers $f^r=\Exp_{\sigma}(r \Log_\sigma(f))$ whenever $f$ is power series with constant term $1$ and coefficients in a $\lambda$-ring $R$ and $r \in R$ (see \cref{ss.powers}). We will use this to give a succint expression of the limiting $\sigma$-moment generating function. We also note that there is a natural embedding of $\lambda$-rings $\mathbb{Z}[\mbb{C}] \hookrightarrow W(\mbb{C}), [z]\mapsto (z,z^2, z^3, \ldots)$. In particular, we may view a multiset of complex numbers as an element of $W(\mbb{C})$. Finally, we write $W(\mbb{C})^\bdd$ for the sub-$\lambda$-ring of $W(\mbb{C})=\prod_{i \geq 1}\mbb{C}$ consisting of $(a_1,a_2,\ldots)$ such that there exists an $M>0$ with $|a_i|\leq M$ for all $i$. The congruences in $W(\mbb{C})^\bdd$ modulo $[q^{-1}]$ below express that two quantities described over $\mbb{F}_{q^n}$ (obtained by projecting to the $n$th component of $W(\mbb{C})$) agree up to $O(q^{-n})$ --- see \cref{s.comparison} for more details on this translation.

\begin{maintheorem}[\cref{theorem.motivic-characters} and \cref{prop.statistics-comparison-characters}; cf. also \cref{theorem.point-counting-characters}]\label{main.dirichlet-characters}
Let $q$ be a prime power, let $\ell$ be a prime dividing $q-1$, and let $\chi:\mu_\ell \rightarrow \mbb{C}^\times$ be a non-trivial character of the group $\mu_\ell$ of $\ell$th roots of unity in $\mbb{F}_q$. Let $U_{d,\ell}(\fqbar)$ denote the set of degree $d$ monic $\ell$th-power-free polynomials over $\fqbar$. Let $X_d$ be the $W(\mbb{C})$-valued random variable on $U_{d,\ell}(\fqbar)$ sending $f$ to the set of zeroes of $\mc{L}(\chi_f, (\#\mbb{F}_q(f))^{-\frac{1}{2}}t)$, where $\mbb{F}_q(f)$ is the extension generated by the coefficients of $f$ and $\chi_f$ is the id\`{e}le class character of $\mbb{F}_q(f)(x)$ associated to $f$ and $\chi$ by Kummer theory as above. 

For $\ell=2$, i.e. for $\chi$ quadratic, 
\begin{align*} \lim_{\substack{d \rightarrow \infty\\ \ell \nmid d}} \mathbb{E}[\Exp_{\sigma}(X_d h_1)] &=\left(1 + \frac{[q]}{[q]+1} \sum_{k\geq 1} [q^{-k}] e_{2 k}\right)^{[q]} \\
&\equiv \Exp_{\sigma}(e_2) \mod [q^{-1}]\Lambda_{W(\mbb{C})^\bdd}^\wedge.
\end{align*}
where we recall that, by \cref{main.random-matrix}-(2), $\Exp_{\sigma}(e_2)$ encodes the eigenvalue distribution of a random compact symplectic matrix. 

For $\ell>2$, i.e. for $\chi$ non-quadratic,
\begin{multline*} \lim_{\substack{d \rightarrow \infty\\ \ell \nmid d}} \mathbb{E}[\Exp_{\sigma}(X_d h_1 + \overline{X}_d \overline{h}_1)] =\\\left(1 + \frac{[q^{\ell-1}]}{[q^{\ell-1}]+\ldots+[q]+1} \sum_{\substack{k_1 \geq 1\textrm{ or } k_2\geq1 \\ k_1 \equiv k_2 \mod \ell}} (-1)^{k_1+k_2}[q^{-\frac{k_1+k_2}{2}}] e_{k_1}\overline{e}_{k_2}\right)^{[q]} \\
\equiv \Exp_{\sigma}(h_1\overline{h}_1) \mod [q^{-1}]\Lambda_{W(\mbb{C})^\bdd}^\wedge.
\end{multline*}
where we recall that, by \cref{main.random-matrix}-(4), $\Exp_{\sigma}(h_1\overline{h_2})$ encodes the eigenvalue distribution of a random unitary matrix. 
\end{maintheorem}

Note that we equip $W(\mbb{C})$ with the product topology, so that convergence can be treated after projection to the $i$th component. By the definition of the expectation functional, this reduces to a problem on the finite $\lambda$-probability space of $W(\mbb{C})$-valued functions on $U_{d,\ell}(\mbb{F}_{q^i})$ equipped with the $\mbb{C}$-valued expectation obtained by averaging the first component. The $W(\mbb{C})$-valued random variables on these finite probability spaces can be made explicit as infinite sums of random variables parameterized by closed points in $\mbb{A}^1_{\mbb{F}_{q^i}}$, and an application of Poonen's \cite{Poonen.BertiniTheoremsOverFiniteFields} sieve shows that they are asymptotically independent and computes their asymptotic $\sigma$-moment generating functions; the asymptotic independence implies that the moment generating function we are interested in is a product of the moment generating functions of these variables. The $i$th component of the power appearing on the right-hand side can also be made very explicit, allowing us to match it with the result.

\begin{remark}\label{remark.quadratic-case-stable-traces}In \cref{example.stable-homology-quadratic} we show that the case of \cref{main.dirichlet-characters} where $\chi$ is quadratic (i.e. $\ell=2$) implies \cite[Theorem 11.2.3]{BergstromDiaconuPetersenWesterland.HyperellipticCurvesTheScanningMapAndMomentsOfFamiliesOfQuadraticLFunctions} by projecting to the first component and applying the standard involution $\omega$ on the ring of symmetric functions. The result \cite[Theorem 11.2.3]{BergstromDiaconuPetersenWesterland.HyperellipticCurvesTheScanningMapAndMomentsOfFamiliesOfQuadraticLFunctions} is proved in two ways in loc. cit.: first, it is deduced from the Poincar\'{e} series for stable cohomology of \cite[Theorem 1.6]{BergstromDiaconuPetersenWesterland.HyperellipticCurvesTheScanningMapAndMomentsOfFamiliesOfQuadraticLFunctions}, the description of the Frobenius action in \cite[Theorem 1.10]{BergstromDiaconuPetersenWesterland.HyperellipticCurvesTheScanningMapAndMomentsOfFamiliesOfQuadraticLFunctions}, and bounds on the size of unstable cohomology, by substituting in $[q^{-1}]$ for $z$ and $[q^{k}]h_{2k}$ for $h_{2k}$ as described in \cite[Top of p.4]{BergstromDiaconuPetersenWesterland.HyperellipticCurvesTheScanningMapAndMomentsOfFamiliesOfQuadraticLFunctions}. Second, it is deduced from a more direct argument using generating functions. By contrast, our proof of \cref{main.dirichlet-characters} is by an  application of Poonen's sieve \cite{Poonen.BertiniTheoremsOverFiniteFields}. The relation between the second method of \cite{BergstromDiaconuPetersenWesterland.HyperellipticCurvesTheScanningMapAndMomentsOfFamiliesOfQuadraticLFunctions} and our proof here seems to be similar to the relation between studying square-free polynomials as a configuration space via the methods of \cite{Howe.MotivicRandomVariablesAndRepresentationStabilityIConfigurationSpaces} or as a special case of smooth hypersurface sections via the methods of \cite{Howe.MotivicRandomVariablesAndRepresentationStabilityIIHypersurfaceSections}.  We emphasize that the non-quadratic case of \cref{main.dirichlet-characters} is completely new and also suggests a new cohomological stability result (cf. \cite{Howe.TheNegativeSigmaMomentGeneratingFunction}). 
\end{remark}

\begin{example}\label{example.conrey-snaith-moments} As in \cite{BergstromDiaconuPetersenWesterland.HyperellipticCurvesTheScanningMapAndMomentsOfFamiliesOfQuadraticLFunctions}, one classical random variable we are interested in is the central value $L(\chi_f, 1/2)$. \cref{main.dirichlet-characters} is insufficient to obtain the refined results of \cite{BergstromDiaconuPetersenWesterland.HyperellipticCurvesTheScanningMapAndMomentsOfFamiliesOfQuadraticLFunctions}, but it is sufficient to study natural truncations: indeed, for $f \in U_{d,\ell}(\fqbar)$, 
\[ L(\chi_f, 1/2)= \sum_{i=0}^\infty (-1)^i (e_i \circ X_{d})_1 (f), \]
where the subscript $1$ denotes the first projection from $W(\mbb{C})$ to $\mbb{C}$,
thus for any $k$ we can take the truncation 
\[ L^{[k]}(\chi_f, 1/2)= \sum_{i=0}^k (-1)^i (e_i \circ X_{d})_1 (f).\]
Projecting onto the first component in \cref{main.dirichlet-characters}, we find that the distribution of $L^{[k]}(\chi_f, 1/2)$ as we vary over $f \in U_{d,\ell}(\fq)$ approximates, as $d \rightarrow \infty$, the distribution of the value at $1$ of the $k$-truncation of the characteristic polynomial of a random compact symplectic matrix if $\chi$ is quadratic or a random unitary matrix if $\chi$ has order $\ell>2$. Projecting onto the $n$th component gives a similar statement for $f \in U_{d,\ell}(\mbb{F}_{q^n})$, and the approximation is up to an error term controlled by $q$ in the sense that, if we vary the limit over $\mbb{F}_{q^n}$ for all $n$, the error term is $O(q^{-n})$. 
\end{example}

\subsubsection{Vanishing cohomology of smooth hypersurface sections}
By a similar method, we compute the $\Lambda$-distributions of reciprocal\footnote{Reciprocal roots are eigenvalues of geometric Frobenius so are most natural for us.} roots of $L$-functions of vanishing cohomology of smooth hypersurface sections of smooth projective varieties over $\mbb{F}_q$. Recall (e.g. from \cite[\S4]{Howe.MotivicRandomVariablesAndRepresentationStabilityIIHypersurfaceSections}) that the vanishing cohomology of a smooth hypersurface section $H/\mbb{F}_q$ of a smooth projective variety $Y/\mbb{F}_q$ is the part of the cohomology of $H$ that is not visible in the cohomology of $Y$ (by weak Lefschetz and Poincar\'{e} duality, it lies in the middle degree cohomology of $H$). As in the case of characters considered above, we write $\mc{L}_{H}(t)\in \mbb{C}[t]$ for the polynomial version of the $L$-function of the vanishing cohomology --- its roots have absolute value $q^{-\frac{\dim H}{2}}$ and, for 
\[ L_{H}(s)=\mc{L}_{H}(q^{-s}),\]
$L_H(s)^{(-1)^{\dim Y}}$ is a factor of $\zeta_H(s)=Z_H(q^{-s})$ (which is a rational function in $q^{-s}$).  

We write $[Y/\mbb{F}_q]=(\#Y(\mbb{F}_q), \#Y(\mbb{F}_{q^2}),\ldots) \in W(\mbb{C})$ and $[H^i(Y)] \in W(\mbb{C})$ for the sequence of traces of powers of Frobenius on the $i$th \'{e}tale cohomology of $Y_{\fqbar}$.

\begin{maintheorem}[\cref{theorem.vanishing-cohom-L-functions} and \cref{prop.statistics-comparison-hypersurface}] \label{main.hypersurface-L-functions}
Let $Y \subseteq \mbb{P}^m_{\mbb{F}_q}$ be an $(n+1)$-dimensional smooth irreducible closed sub-variety, and let $U_d(\fqbar)$ be the set of nonzero homogeneous degree $d$ polynomials $F$ in $m+1$ variables over $\fqbar$ whose vanishing locus $V(F)$ intersects $Y$ transversely. Let $X_d$ be the $W(\mbb{C})$-valued random variable on $U_d(\fqbar)$ sending $F$ to the set of reciprocal roots of $\mathcal{L}_{V(F) \cap Y}(t (\#\mbb{F}_q(F))^{-\frac{n}{2}})$, where $\mbb{F}_q(F)$ is the extension generated by the coefficients of $F$ and $V(F)\cap Y$ is viewed as a variety over $\mbb{F}_q(F)$. Then, 
\[ \lim_{d \rightarrow \infty}\mbb{E}[\Exp_{\sigma}(X_dh_1)]= (1 + p\sum_{i \geq 1} [q^{-in/2}] \epsilon^i f_i)^{[Y/\mbb{F}_q]} \cdot \Exp_\sigma(\mu h_1)  \]
where $f_i=e_i$ and $\epsilon=-1$ if $n$ is odd and $f_i=h_i$ and $\epsilon=1$ if $n$ is even, 
\[ p = \frac{[q]^{n+1}-1}{[q]^{n+2}-1} \textrm{, and } \mu=-\epsilon \left(\sum_{i=0}^{n-1} (-1)^i\left([q^{\frac{-n}{2}}]+[q^{\frac{n-2i}{2}}]\right)[H^i(Y)]\right) - [q^{-n/2}][H^{n}(Y)]. \]
In particular, modulo $[q^{-\frac{1}{2}}]\Lambda_{W(\mbb{C})^\bdd}^\wedge$ this agrees with:
    \begin{enumerate}
        \item For $n>0$ even, $\Exp_{\sigma}(h_2)$ (as in \cref{main.random-matrix}-(1), the orthogonal case).
        \item For $n$ odd, $\Exp_{\sigma}(e_2)$ (as in \cref{main.random-matrix}-(2), the symplectic case).
        \item For $n=0$, $\Exp_{\sigma}(h_2+h_3+\ldots)$ (the distribution associated to the standard irreducible representation of the symmetric group as in \cref{example.standard-rep}).  
    \end{enumerate}
Thus the asymptotic $\Lambda$-distribution of $X_d$ approximates the asymptotic eigenvalue $\Lambda$-distribution of the associated random matrices modulo $[q^{-1/2}]\Lambda_{W(\mbb{C})^\bdd}$. 
\end{maintheorem}

\begin{remark}\label{remark.sidways-equidistribution}
    In each case $(1), (2), (3)$ enumerated at the end of \cref{main.hypersurface-L-functions}, one can show that for a smooth hypersurface section one obtains a normalized Frobenius conjugacy class in the associated group whose eigenvalues are the zeroes of the normalized $L$-function (an orthogonal group in case (1), a symplectic group in case (2), and a symmetric group $\Sigma_m$ acting through its irreducible $m-1$ dimensional representation in case (3)). This is essentially because in case (1) the cup product on vanishing cohomology is a non-degenerate symmetric pairing, in case (2) it is a non-degenerate antisymmetric pairing, and, in the zero-dimensional case (3), when we pass to cohomology, taking the vanishing part corresponds exactly to splitting of the standard irreducible from the standard permutation representation. However, the fact that one has such a conjugacy class plays \emph{no role whatsoever} in the proof of \cref{main.hypersurface-L-functions}. It is interesting to contrast this to the situation when one takes first a limit as $q \rightarrow \infty$ and then afterwards as $d \rightarrow \infty$ as in \cite{KatzSarnak.RandomMatricesFrobeniusEigenvaluesAndMonodromy}, where the computation of the monodromy group is crucial in order to apply Deligne equidistribution; we note further that, with that order of limits, only the degree 0 cohomology of local systems matter, whereas our results detect cohomology in all degrees. In \cite{BiluHowe.MotivicRandomVariables} we will obtain an analog of \cref{main.hypersurface-L-functions} using cohomological realizations in place of the point-counting realization that can be used, in particular, to determine the degree zero cohomology of local systems and thus also the monodromy groups. 
\end{remark}

\begin{remark}\label{remark.vanishing-cohomology-relation-to-other-work}
    \cref{main.hypersurface-L-functions} and its proof are closely related to \cite[Theorems B and C]{Howe.MotivicRandomVariablesAndRepresentationStabilityIIHypersurfaceSections}. The main improvement here is the use of the $\sigma$-moment generating function in place of what amounts, in hindsight, to a $\lambda$-variant of a falling moment generating function. In particular, this improvement is what makes possible the explicit comparison with random matrix statistics. Note also that, in \cite{Howe.MotivicRandomVariablesAndRepresentationStabilityIIHypersurfaceSections} we attached to each partition $\tau$ a sequence of local systems $\mc{V}_\tau$ over the moduli of smooth hypersurface sections by taking the irreducible representation of the symplectic, orthogonal, or symmetric group associated to $\tau$ and composing it with the monodromy representation on vanishing cohomology, and our main goal in \cite{Howe.MotivicRandomVariablesAndRepresentationStabilityIIHypersurfaceSections} was to show that homological stability held for these families of local systems at the level of enriched Euler characteristics. From \cref{main.hypersurface-L-functions} we obtain better explicit formulas than those described in \cite{Howe.MotivicRandomVariablesAndRepresentationStabilityIIHypersurfaceSections}: to extract the stable point count for a family of local system associated to a partition $\tau$ from \cref{main.hypersurface-L-functions}, one needs only compute the Hall inner product between the moment generating function and the associated orthogonal, symplectic, or symmetric Schur polynomial. This is, both conceptually and practically, simpler than the method described in \cite[Appendix A]{Howe.MotivicRandomVariablesAndRepresentationStabilityIIHypersurfaceSections}. We note that there has also been recent progress towards establishing homological stability in this setting \cite{AumonierDas.HomologicalStabilityForTheSpaceOfHypersurfacesWithMarkedPoints}, but since there is no known control over the unstable cohomology, the homological stability results do not imply point counting results.  
\end{remark}

\subsubsection{Generalizations}
In \cite{BertucciHowe.EquidistributionAndArithmeticLambdaDistributions}, we generalize the method used to compute asymptotic $\Lambda$-distributions for families of $L$-functions in Theorems \ref{main.dirichlet-characters} and \ref{main.hypersurface-L-functions}. The main result of \cite{BertucciHowe.EquidistributionAndArithmeticLambdaDistributions} is that, whenever Poonen's sieve applies, certain infinite sums of random variables parameterized by the $\fqbar$-points of a variety ($\mbb{A}^1$ in the case of \cref{main.dirichlet-characters} and $Y$ in the case of \cref{main.hypersurface-L-functions}) have asymptotic $\sigma$-moment generating functions described as motivic Euler products of local $\sigma$-moment generating functions at each point. This gives a precise interpretation of the heuristic that the random variables parameterized by distinct points are asymptotically independent, but depends crucially on having a good notion of a motivic Euler product. We are able to get by with less in the present work essentially because the moment generating functions in Theorems \ref{main.dirichlet-characters} and \ref{main.hypersurface-L-functions} are powers rather than more general motivic Euler products.

\subsection{Organization}
In \cref{s.pre-lambda}, we recall the theory of pre-$\lambda$ and $\lambda$-rings. The main novelty in our treatment is that we emphasize monoid rings as the primitive example, and this facilitates many of the explicit computations we make with plethysm later in the paper. In \cref{s.lambda-prob}, we introduce the basic notions of probability in $\lambda$-rings. In fact, we work in the generality of pre-$\lambda$ rings --- none of the foundational results depend on the added conditions that cut out $\lambda$-rings among  pre-$\lambda$ rings, and there are interesting examples of pre-$\lambda$ rings that may not be $\lambda$-rings, such as Grothendieck rings of varieties, where one hopes to apply the theory. In addition to the $\sigma$-moment generating function, we also discuss more general moment generating functions and give examples of some important $\Lambda$-distributions. In \cref{s.random-matrices}, we prove \cref{main.random-matrix} and deduce from it the main results of \cite{DiaconisShahshahani.OnTheEigenvaluesOfRandomMatrices}. 

In \cref{s.abstract-pc}, we set up an abstract point counting formalism and establish several basic results and lemmas that go into the proofs of \cref{main.dirichlet-characters} and \cref{main.hypersurface-L-functions}. In \cref{s.application-of-poonen}, we apply Poonen's sieve to give a general computation of $\sigma$-moment generating function for certain $W(\mbb{C})$-valued random variables on finite probability spaces cut out by Taylor conditions (\cref{prop.point-counting-asymptotics}); this is used in the proofs of \cref{main.dirichlet-characters} and \cref{main.hypersurface-L-functions}. In \cref{s.L-characters}, we compute the limits in \cref{main.dirichlet-characters} and discuss complements, and, in \cref{s.L-vanishing}, we compute the limits in \cref{main.hypersurface-L-functions} and discuss complements. Finally, in \cref{s.comparison}, we complete the proofs of \cref{main.dirichlet-characters} and \cref{main.hypersurface-L-functions} by comparing the limiting $\sigma$-moment generating functions computed therein with the limiting $\sigma$-moment generating functions for the associated random matrix distributions as computed in \cref{main.random-matrix} (and \cref{example.standard-rep}). 

\subsection{Acknowledgements} This work owes an enormous debt to our collaborations with Margaret Bilu and Ronno Das, whom we thank both for these collaborations and for other helpful conversations. We thank Matthew Bertucci for helpful comments on an earlier version. We thank Guillaume Dubach and Michael Magee for separate helpful conversations about random matrices and the Weingarten calculus (which we do not use anywhere, though perhaps we should!). We thank Dan Petersen for a helpful conversation about the results of \cite{BergstromDiaconuPetersenWesterland.HyperellipticCurvesTheScanningMapAndMomentsOfFamiliesOfQuadraticLFunctions}. We thank Ravi Vakil for encouraging us several years ago to write up some of our ideas on this subject, and Helmut H\"ofer for inviting us to give a mathematical conversation at the Institute for Advanced Study, which provided the perfect opportunity to revisit that project. 

During the preparation of this work, Sean Howe was supported by the National Science Foundation through grant DMS-2201112 and as a member at the Institute for Advanced Study during the academic year 2023-24 by the Friends of the Institute for Advanced Study Membership.

\section{(Pre-)$\lambda$-rings}\label{s.pre-lambda}

In this section we recall the theory of pre-$\lambda$ and $\lambda$-rings. One novelty in our treatment is that we emphasize monoid rings as the atomic examples of pre-$\lambda$ rings. In particular, we use this to construct the pre-$\lambda$ ring structure on the ring of symmetric functions $\Lambda$ by embedding it in an infinite power series ring. This perspective is extremely helpful for making computations with plethysm, since the plethystic action on the power series ring (and other monoid rings) is very simple. In particular, we use this in \cref{s.lambda-prob} to give explicit computations of $\sigma$-moment generating functions (including the crucial expansion \cref{eq.intro-moment-gen-expansion} of the $\sigma$-moment generating function). Some of the results on powers in \cref{ss.powers} are also of independent interest; for example, \cref{example.configuration-symmetric-functions} and the surrounding discussion explain a general construction of symmetric functions that, in particular, simplifies the treatment of configuration symmetric functions given in \cite{Howe.MotivicRandomVariablesAndRepresentationStabilityIConfigurationSpaces}.  

\subsection{Symmetric functions}
We recall some standard definitions and results for symmetric functions (see, e.g., \cite{Macdonald.SymmetricFunctionsAndHallPolynomials}).
Let $\Lambda_n:= \mbb{Z}[t_1,\ldots,t_n]^{\Sigma_n}$, the ring of symmetric polynomials in $n$-variables, and let $\Lambda := \lim_n \Lambda_n$, the ring of symmetric functions, where the transition map from $\Lambda_{n+1}$ to $\Lambda_n$ sends $t_i$ to $t_i$ for $1 \leq i \leq n$ and sends $t_{n+1}$ to $0$. We will use the following elements of $\Lambda$:
\begin{itemize}
\item The power sum symmetric functions $p_i = \sum_j t_j^i$. 
\item The complete symmetric functions $h_i = \sum_{\sum a_j=i,\; a_j \in \mbb{Z}_{\geq 0}} \prod_j t_j^{a_j}$
\item The elementary symmetric functions $e_i = \sum_{\sum a_j=i, a_j \in\{0,1\}} \prod_j t_j^{a_j}.$
\end{itemize}
These give polynomial bases:
\begin{equation}\label{eq.polynomial-bases} \Lambda=\mbb{Z}[e_1, e_2, \ldots] = \mbb{Z}[h_1, h_2, \ldots] \textrm{ and } \Lambda_{\mbb{Q}}=\mbb{Q}[p_1, p_2, \ldots].\end{equation}

By a partition, we mean a non-increasing sequence of non-negative integers $(\tau_1, \tau_2, \ldots)$ that is eventually zero. For $i \geq 0$, we write $\tau(i)$ for the multiplicity of $i$ in $\tau$, i.e. number of occurrences of $i$ in $\tau$. Thus we can also view a partition as a function from $\mbb{Z}_{\geq 1}$ to $\mbb{Z}_{\geq 0}$ with finite support. For $\tau$ a partition, we write
\[ p_\tau:=\prod_i p_{\tau_i}=\prod_{i} p_i^{\tau(i)}, h_\tau:=\prod_i h_{\tau_i}=\prod_{i} h_i^{\tau(i)}, e_\tau:=\prod_i e_{\tau_i}=\prod_{i} e_i^{\tau(i)}. \]
By \cref{eq.polynomial-bases}, $\{e_\tau\}_\tau$ and $\{h_\tau\}_\tau$ each give a $\mbb{Z}$-basis for $\Lambda$, while $\{p_\tau\}_\tau$ is a $\mathbb{Q}$-basis for $\Lambda_{\mbb{Q}}$. There is also a $\mbb{Z}$-basis $\{m_\tau\}_\tau$ for $\Lambda$, where $m_\tau$ is the monomial symmetric function obtained by summing the monomials that can be obtained from permutations of $t^\tau=\prod_{i \geq 1}t_i^{\tau_i}$, e.g. 
\[ m_{(2,1,0,0,\ldots)}=(t_1^2 t_2 + t_1^2t_3 +\ldots) + (t_2^2 t_1 + t_2^2 t_3 + \ldots) + \ldots .\]
In particular, note that 
\[ p_1=p_{(1,0,0,\ldots)}=e_1=e_{(1,0,0,\ldots)}=h_1=h_{(1,0,0,\ldots)}=m_{(1,0,0,\ldots)}=\sum_i t_i. \]

We recall some well-known relations between complete, elementary, and power sum symmetric functions: first, there are the obvious identities in $\Lambda[[z]]$,
\begin{equation}\label{eq.complete-and-elementary-generating-functions} \sum_{k \geq 0} h_k z^k = \prod_{j \geq 1}\frac{1}{1-t_j z} \textrm{ and } \sum_{k \geq 0} e_k z^k = \prod_{j \geq 1} (1+t_j z). \end{equation}
It follows, by substituting $-z$ for $z$ in the second identity of \cref{eq.complete-and-elementary-generating-functions}, that 
\begin{equation}\label{eq.complete-elementary-inversion} \left(\sum_{k \geq 0} h_k z^k\right)\left(\sum_{k \geq 0} (-1)^k e_k z^k\right) = 1. \end{equation}
By taking the logarithm in $\Lambda_{\mbb{Q}}[[z]]$ and applying the first identity of \cref{eq.complete-and-elementary-generating-functions}, 
\begin{equation}\label{eq.complete-powersum-relation} \prod_{i\geq 1} \exp\left(\frac{p_i z}{i}\right)=\prod_{j \geq 1}\frac{1}{1-t_j z}=\sum_{k \geq 0} h_k z^k, \end{equation}
where here $\exp$ is defined by the usual power series $\exp(x)=\sum_{i \geq 0} \frac{x^i}{i!}.$ 
\subsection{Pre-$\lambda$ rings}

We now recall the theory of pre-$\lambda$ rings. Some of this material is standard and can be found in, e.g., \cite{Howe.MotivicRandomVariablesAndRepresentationStabilityIConfigurationSpaces, GuseinZadeLuengoEtAl.OnThePowerStructureOverTheGrothendieckRingOfVarietiesAndItsApplications} and elsewhere, but we are not aware of another source for the specific material on monoid rings developed below. 

\begin{definition}\label{def.pre-lambda}A pre-$\lambda$ ring is a commutative ring $R$ equipped with a \emph{plethystic action} map 
\[ R \rightarrow \Hom_{\mr{Ring}}(\Lambda, R), r \mapsto (f \mapsto f \circ r)\]
such that 
\begin{enumerate}
    \item $h_k \circ (r + s)=(h_k \circ r) + (h_{k-1} \circ r)(h_1 \circ s) + \ldots + (h_k \circ s).$
    \item $h_1 \circ r=r$.
\end{enumerate}
\end{definition}

This matches with, e.g., \cite[Definition 2.2]{Howe.MotivicRandomVariablesAndRepresentationStabilityIConfigurationSpaces}, by setting $h_k \circ = \sigma_k$. 
We note that part (2) is not always included in the definition of a pre-$\lambda$ ring, but, as in \cite{{Howe.MotivicRandomVariablesAndRepresentationStabilityIConfigurationSpaces}}, it is a natural condition for our purposes. 

It can be deduced from \cref{eq.complete-elementary-inversion} and \cref{eq.complete-powersum-relation} that, in any pre-$\lambda$ ring $R$,
\begin{align*}
e_k \circ (r + s)&=(e_k \circ r) + (e_{k-1} \circ r)(e_1 \circ s) + \ldots + (e_k \circ s)  \\
p_k \circ (r + s)&=(p_k \circ r) + (p_k \circ s).
\end{align*}
Verifying the first of these is equivalent to verifying the analogous condition for $h_k \circ$ in \cref{def.pre-lambda}, and the second is equivalent if $R$ is torsion-free as a $\mbb{Z}$-module. 

It follows from these identities that the operations $p_k \circ$ are endomorphisms of the underlying abelian group $(R,+)$. One often writes $\psi_k = p_k \circ$ and refers to it as the $k$th Adams operation, but we will not use this notation here. 

\begin{example}\label{example.integers-pre-lambda-ring}
    $\mathbb{Z}$ admits a natural pre-$\lambda$ ring structure such that the $\sigma$-operations $\sigma_k=h_k \circ$ decategorify symmetric powers of vector spaces and the $\lambda$-operations $e_k \circ$ decategorify exterior powers of vector spaces:
    \[ \textrm{ for $n \geq 1$, } h_k \circ n = \dim \Sym^k \mathbb{C}^n=\binom{k+n-1}{n-1}, e_k \circ n= \dim \bigwedge^k \mathbb{C}^n=\binom{n}{k}.\] 
    One can show for this structure that $p_k \circ$ is the identity morphism for all $k$. 
\end{example}

The ring of symmetric functions $\Lambda$ itself has a natural pre-$\lambda$ ring structure, and using this we will define the notion of a $\lambda$-ring. For making computations, it will be useful to understand this pre-$\lambda$ ring structure on $\Lambda$ as the restriction of a very simple pre-$\lambda$ ring structure on a ring of power series. More generally, monoid rings over pre-$\lambda$ rings have a natural pre-$\lambda$ structure, as we explain now. 

We use the following notation: given a set $S$ and $k \geq 0$, we write $\Sym^k S= S^k/\Sigma_k$, where $\Sigma_k$ is the symmetric group on $k$ elements and it acts by permuting coordinates in the $k$-fold product.  A commutative monoid is a set $M$ equipped with a commutative and associative multiplication law $M \times M \rightarrow M, (m_1,m_2)\mapsto m_1m_2$ admitting an identity element $1_M$. It is locally finite if the fibers of multiplication $M \times M \rightarrow M$ are finite. 
\begin{definition}\label{def.monoid-ring-and-pre-lambda}
    Let $M$ be a commutative monoid and let $R$ be a ring.
    \begin{enumerate}
        \item We write $R[M]$ for the associated monoid ring. Explicitly, $R[M]$ is the convolution algebra of finitely supported $R$-valued functions on $M$, where addition is pointwise and multiplication is via the convolution product: 
        \[ (f+g)(m)=f(m)+g(m) \textrm{ and } (fg)(m)=\sum_{\substack{(m_1, m_2) \in M^2 \\ \textrm{ s.t. } m_1m_2=m}}f(m_1)g(m_2). \]       
        For $m \in M$, we write $[m]$ for the function that send $m$ to $1_R$ and all other elements to zero. Thus any element of $R[M]$ is a finite sum $\sum a_m [m]$, $a_m \in R$, and the multiplication is determined by $[m_1][m_2]=[m_1m_2]$ for $m_1, m_2 \in M$. It is a commutative ring with unit $[1_M]$. 
        \item If $M$ is locally finite, then we write $R[[M]]$ for convolution algebra of \emph{all} $R$-valued functions on $M$ (the locally finite condition ensures that the sum defining the convolution product above is a finite sum). It is a commutative ring containing $R[M]$. 
    \end{enumerate}
    If $R$ is a pre-$\lambda$ ring, we equip $R[M]$ and, if $M$ is locally finite, $R[[M]]$, with the plethystic action of $\Lambda$ determined by 
    \[ (h_k \circ f)(m)=\sum_{\substack{\sum a_{n} n \in \Sym^k M \\ \textrm{ s.t. } \prod n^{a_n}=m}} \prod_{n \in M} h_{a_n}\circ f(n)  \]
    Note that here in the subscript for the sum, the $a_n$ are non-negative integers indexed by elements $n\in M$ and adding up to $k$, thus $a_n$ is zero for all but finitely many $n$ so that the product appearing is in fact a finite product, and moreover it is equal to zero for all but finitely many choices so that the sum appearing is a finite sum. 
\end{definition}

When $M$ is locally finite, we topologize $R[[M]]$ with the weak topology (where $R$ is equipped with the discrete topology), so that a net $(f_i)_{i \in I}$ in $R[[M]]$ converges to $f \in R[[M]]$ if and only if $f_i(m)$ is  eventually constant equal to $f(m)$. Then $R[M]$ is dense in $R[[M]]$, and the ring structure and plethystic actions of symmetric functions $f \in \Lambda$ are evidently continuous for this topology. 

\begin{lemma}\label{lemma.monoid-pleth-is-pre-lambda-and-formulas}
    If $R$ is a pre-$\lambda$ ring and $M$ is a commutative monoid, then $R[M]$ is a pre-$\lambda$ ring for the plethystic action as in \cref{def.monoid-ring-and-pre-lambda}. If $M$ is furthermore locally finite, then $R[[M]]$ is a pre-$\lambda$ ring for the plethystic action as in \cref{def.monoid-ring-and-pre-lambda}. 
    In particular, in these pre-$\lambda$ ring structures we have
    \[  h_k \circ (a[m])=(h_k \circ a)[m^k] \textrm{ and }
    e_k \circ (a[m])= (e_k \circ a)[m^k]. \]
    If $R$ is torsion free as a $\mathbb{Z}$-module, then we have moreover
    \[ (p_k \circ f) (m)= \sum_{\substack{ n \in M \\ \textrm{s.t. } n^k=m}} p_k \circ f(n) \textrm{ and, in particular } p_k \circ (a[m])= (p_k \circ a)[m^k].\]
\end{lemma}
\begin{proof}
We first show that the plethystic action defined in \cref{def.monoid-ring-and-pre-lambda} is a {pre-$\lambda$} ring structure. The verification is essentially the same as the verification that convolution distributes over addition, but we spell it out for completeness: 
\begin{align*} (h_k \circ (f+g))(m)&=\sum_{\substack{\sum c_{n} n \in \Sym^k M \\ \textrm{ s.t. } \prod n^{c_n}=m}} \prod_{n \in M} h_{c_n}\circ (f(n) + g(n))\\
&=\sum_{\substack{\sum c_{n} n \in \Sym^k M \\ \textrm{ s.t. } \prod n^{c_n}=m}} \prod_{n \in M} \sum_{i+j=c_n} (h_i \circ f(n)) (h_j \circ g(n)) \\
&= \sum_{\substack{\sum a_{n} n \in \Sym^{k_1} M, \sum b_n n \in \Sym^{k_2} M\\ \textrm{ s.t. } k_1+k_2=k, \prod n^{a_n+b_n} = m}} \prod_{n \in M} (h_{a_n} \circ f(n)) (h_{b_n} \circ g(n)) \\
&= \sum_{\substack{m_1 m_2=m \\ k_1+k_2=k}} \left(\sum_{\substack{\sum a_{n} n \in \Sym^{k_1} M \\ \textrm{ s.t. } \prod n^{a_n}=m_1}} \prod_{n \in M} h_{a_n} \circ f(n)\right) \left(\sum_{\substack{\sum b_{n} n \in \Sym^{k_2} M \\ \textrm{ s.t. } \prod n^{b_n}=m_2}} \prod_{n \in M} h_{b_n} \circ g(n)\right) \\
&=
\sum_{k_1+k_2=k} \left((h_{k_1} \circ f)(h_{k_2} \circ g)\right)(m). 
\end{align*}

The formula for $h_k \circ (a[m])$ is immediate from the definition. It follows that, in $R[M][[z]]$, 
\[ \sum_{k \geq 0} \left(h_k \circ(a[m])\right)z^k= \sum_{k \geq 0} (h_k \circ a)[m^k] z^k. \]
Thus, applying \cref{eq.complete-elementary-inversion} to the element $a$ in $R$ and substituting $[m]z$ for $z$, we find the multiplicative inverse of this series is $\sum_{k \geq 0} (-1)^k (e_k \circ a)[m^k] z^k$. But since $R[M]$ is a pre-$\lambda$ ring, by the same logic its multiplicative inverse is also $\sum_{k \geq 0} (-1)^k (e_k \circ (a[m]))z^k$. Thus we obtain the claimed formula for $e_k \circ (a[m])$. 

Now assume $R$ is $\mathbb{Z}$-torsion free. To obtain the description of the plethystic action of $p_k$, note that the given formula is evidently an additive homomorphism on $R[M]$, so to obtain it on $R[M]$ it suffices to verify the formula $p_k \circ a([m])=(p_k \circ a)[m^k]$. This also gives the formula in the locally finite case by continuity, since $R[M]$ is dense in $R[[M]]$. To verify $p_k \circ a([m])=(p_k \circ a)[m^k]$, we argue similarly to the above but invoking \cref{eq.complete-powersum-relation} instead of \cref{eq.complete-elementary-inversion}. 
\end{proof}

\begin{example}\label{example.polynomial-pre-lambda-rings}\hfill
\begin{enumerate}
\item Taking the monoid $M_n=\mbb{Z}_{\geq 0}^n$ and equipping $\mathbb{Z}$ with the pre-$\lambda$ structure of \cref{example.integers-pre-lambda-ring}, from \cref{lemma.monoid-pleth-is-pre-lambda-and-formulas} we obtain natural pre-$\lambda$ structures on $\mbb{Z}[M_n]=\mbb{Z}[t_1, \ldots, t_n]]$ and $\mbb{Z}[[M_n]]=\mbb{Z}[[t_1,\ldots,t_n]$. Note that, since $e_k\circ 1=0$ for all $k \geq 2$, for any monomial $\ul{t}^{\ul{j}}$ we find
\[ e_1 \circ \ul{t}^{\ul{j}}=\ul{t}^{\ul{j}} \textrm{ and } e_k \circ \ul{t}^{\ul{j}}=0\, \forall k\geq2. \]
\item The invariants $\Lambda_n$ and $\Lambda_n^\wedge$ for the action of $\Sigma_n$ on $\mathbb{Z}[t_1,\ldots,t_n]$ and $\mathbb{Z}[[t_1,\ldots,t_n]]$ are preserved by the pre-$\lambda$ structure, so have natural pre-$\lambda$ structures. 
\end{enumerate}
\end{example}

We also have the following evident functoriality for the construction.
\begin{lemma}\label{lemma.monoid-ring-functoriality}
    Suppose $\varphi:R \rightarrow S$ is a map of pre-$\lambda$ rings and $f: M \rightarrow N$ is a map of monoids. Then there is a unique map of pre-$\lambda$ rings $R[M] \rightarrow S[N]$ such that $r[m]\mapsto \varphi(r)[f(m)].$ If $M$ and $N$ are both locally finite and $f$ has finite fibers, then it extends uniquely to a continuous map of pre-$\lambda$ rings 
$R[[M]] \rightarrow S[[N]].$
\end{lemma}

\subsection{Symmetric power series}\label{ss.symmetric-power-series}
Building on \cref{example.polynomial-pre-lambda-rings}, we may view \[ R[t_1,t_2, \ldots]=R[t_i, i\in \mbb{N}] \]
as $R[M_{\mbb{N}}]$ for $M_{\mbb{N}}=\bigoplus_{\mbb{N}}\mbb{Z}_{\geq 0}$ the free commutative monoid on $\mbb{N}$, and we write 
\[ R[[\ul{t}_{\mbb{N}}]]:=R[[M_\mbb{N}]] \cong \lim_n R[[t_1,\ldots, t_n]]. \]
We write $\Lambda_{R}^\wedge$ for the subset of functions in $R[[\ul{t}_{\mbb{N}}]]$ that are invariant under all automorphisms induced by automorphism of $\mbb{N}$, and $\Lambda_R$ for the subset of $\Lambda_R^\wedge$ consisting of those functions that are zero on $\mbb{Z}_{\geq n}^\mbb{N}$ for some $n\gg 0$. In particular, for $R=\mathbb{Z}$ we recover the symmetric functions $\Lambda=\Lambda_{\mbb{Z}}$ and its natural completion $\Lambda_{\mbb{Z}}^\wedge$ for the filtration by degree. More generally, $\Lambda_R$ is naturally identified with $\Lambda \otimes_{\mbb{Z}} R$, and $\Lambda_R^\wedge$ is identified with the completion of $\Lambda_R$ for the filtration by degree. 

If $R$ is a pre-$\lambda$ ring, then $\Lambda_R$ and $\Lambda_R^\wedge$ are evidently preserved by the pre-$\lambda$ operations of \cref{def.monoid-ring-and-pre-lambda}. Taking $R=\mathbb{Z}$ with the pre-$\lambda$ structure of \cref{example.integers-pre-lambda-ring}, we obtain, in particular, a natural pre-$\lambda$ ring structure on $\Lambda$ itself.

\begin{example}
For $f \in \Lambda$, from \cref{lemma.monoid-pleth-is-pre-lambda-and-formulas} we find that $p_i \circ f=f(t_1^i, t_2^i, \ldots)$. Thus, comparing with \cite[Chapter 8, eq. (8.4) on p. 135]{Macdonald.SymmetricFunctionsAndHallPolynomials}, we have reconstructed the usual plethystic action of $\Lambda$ on itself. 
\end{example}

\subsection{Witt vectors and $\lambda$-rings}\label{ss.witt-and-lambda}
We now recall some standard results about Witt vectors and their relation with pre-$\lambda$ rings as in \cite[\S16]{Hazewinkel.WittVectorsI}. There is a natural coalgebra structure on $\Lambda$, where the additive counit $\Lambda \rightarrow \mbb{Z}$ sends $h_i$ to $0$, $i \geq 0$, the multiplicative counit $\Lambda \rightarrow \mbb{Z}$ sends $h_i$ to $1$, $i \geq 0$, the coaddition $\Lambda \rightarrow \Lambda \otimes_{\mbb{Z}} \Lambda$ is uniquely determined by 
\[ p_i \mapsto p_i \otimes 1 + 1 \otimes p_i, \]
or, equivalently,
\begin{equation}\label{eq.lambda-coadddition-complete-sym} h_k \mapsto \sum_{i+j=k}h_i \otimes h_j, \end{equation}
and the comultiplication $\Lambda \rightarrow \Lambda \otimes_{\mbb{Z}} \Lambda$ is uniquely determined by 
\[ p_i \mapsto p_i \otimes p_i.\]

For $R$ a ring, $W(R):=\Hom_{\mr{Ring}}(\Lambda, R)$ is thus a ring. It is moreover a pre-$\lambda$ ring for the plethystic action of $\Lambda$ by $(f \circ w) (g)= w(g \circ f)$. If $R$ is a pre-$\lambda$ ring, there is a canonical homomorphism of abelian groups
\[ R \rightarrow W(R), r \mapsto (f \mapsto f \circ r). \]
Conversely, to give a pre-$\lambda$ ring structure on $R$ it is equivalent to give such a homomorphism (by comparing the defining property with \cref{eq.lambda-coadddition-complete-sym}) that is a section of the ring homomorphism $W(R) \rightarrow R$ given by evaluation on $h_1=p_1$. 

\begin{definition}
    A pre-$\lambda$ ring $R$ is a \emph{$\lambda$-ring} if the additive homomorphism $R \rightarrow W(R)$ defining the pre-$\lambda$ structure is a morphism of pre-$\lambda$ rings. 
\end{definition}

\begin{example}
    For any ring $R$, $W(R)$ is a $\lambda$-ring. It is the cofree $\lambda$-ring on $R$, i.e., for any $\lambda$-ring $S$, $\Hom_{\lambda\mr{-Ring}}(S,W(R))=\Hom_{\mr{Ring}}(S,R)$ via evaluation at $h_1$. 
\end{example}

It follows, in particular, that if $R$ is a $\lambda$-ring then $f \circ (g \circ r)=(f \circ g) \circ r$ and the Adams operations $p_k \circ$ are $\lambda$-ring homomorphisms (note that a $\lambda$-ring homomorphism is the same thing as a pre-$\lambda$ ring homomorphism between $\lambda$-rings). Using the definition of the algebra structure and plethysm on $W(R)$, we have:

\begin{proposition}\label{prop.lambda-ring-adams-criterion}
    A $\mbb{Z}$-torsion free  pre-$\lambda$ ring $R$ is a $\lambda$-ring if and only if
    \begin{enumerate}
        \item For all $k,j$, $p_k \circ (p_j \circ r)= p_{jk} \circ r$
        \item For all $k$, $p_k \circ: R \rightarrow R$ is a ring homomorphism. 
    \end{enumerate}
\end{proposition}

In particular, using this criterion we obtain 

\begin{lemma}
    If $R$ is a $\mathbb{Z}$-torsion free $\lambda$-ring and $M$ is a commutative monoid, then $R[M]$ is a $\lambda$-ring. If $M$ is locally finite, then $R[[M]]$ is a $\lambda$-ring. In particular, if $R$ is a $\lambda$-ring then $\Lambda_R$ and $\Lambda_R^\wedge$ are both $\lambda$-rings.   
\end{lemma}
\begin{proof}
    For $R[M]$ and $R[[M]]$, the criteria of \cref{prop.lambda-ring-adams-criterion} follow immediately from \cref{lemma.monoid-pleth-is-pre-lambda-and-formulas}. The result then follows for $\Lambda_R$ and $\Lambda_R^\wedge$ since any sub-pre-$\lambda$ ring of a $\lambda$-ring is a $\lambda$-ring. 
\end{proof}

\subsection{Plethystic exponential and logarithm}
Let $R$ be a pre-$\lambda$ ring. For 
\[ A= R[[\ul{t}_{\mbb{N}}]] \textrm{ or } \Lambda_R^\wedge,\]
let $\mf{m}_A$ be the ideal of elements with constant term $0$. For $f \in \mf{m}_A$, we define 
\[ \Exp_{\sigma}(f)= \sum_{k \geq 0} h_k \circ f, \]
where the sum evidently converges. We have

\begin{lemma}\label{lemma.expsigma-group-iso}
    $\Exp_{\sigma}$ is a continuous group isomorphism from $\mf{m}_A$, equipped with addition, to $1+\mf{m}_A$, equipped with multiplication. 
\end{lemma}
\begin{proof}
    We first treat the case $A=R[[\ul{t}_{\mbb{N}}]]$. 
    That it is a homomorphism follows from $h_k(x+y)=\sum_{i+j=k} h_i(x)h_j(y)$. We then use \cref{lemma.monoid-pleth-is-pre-lambda-and-formulas} to compute
    \[ \Exp_{\sigma}\left(\sum_{\ul{j}} r_{\ul{j}} \ul{t}^{\ul{j}}\right) = \prod_{\ul{j}} \sum_{k\geq 0} h_k(r_{\ul{j}}) \ul{t}^{k\ul{j}}\]
    Since $h_1(r_{\ul{j}})=r_{\ul{j}}$, any element of $1+ \mf{m}_A$ has a unique expansion of the form on the right, so we deduce that $\Exp_\sigma$ is a bijection. It commutes with the action of $\Aut(\mbb{N})$, thus we also obtain the case of $\Lambda_R^\wedge$. 
\end{proof}

\begin{example}\label{example.how-to-expand-expsigmah1}
    Suppose $X \in R$. Then we compute 
    \[ \Exp_{\sigma}(Xh_1)=\Exp_{\sigma}(Xt_1+Xt_2+\ldots)=\prod_{i \geq 1} \Exp_\sigma(Xt_i) =\prod_{i\geq 1}\left(\sum_{k \geq 0} (h_k \circ X)t_i^k\right) \]
    where in the second equality we use \cref{lemma.expsigma-group-iso} and in the third equality we use \cref{lemma.monoid-pleth-is-pre-lambda-and-formulas} to obtain $h_k\circ (Xt_i)=(h_k \circ X)t_i^k$. Expanding the product, we obtain
    \begin{equation}\label{eq.exp-symmetric-monomial-expansion} \Exp_\sigma(Xh_1)=\sum_{\tau} (h_{\tau}\circ X) m_\tau. \end{equation}
\end{example}

From \cref{eq.complete-elementary-inversion}, we obtain
\begin{lemma}\label{lemma.exponential-minus} For $x \in \mathfrak{m}_A$, 
\[ \Exp_{\sigma}(-x)= \sum_{i=0}^\infty (-1)^i e_i \circ x. \]
\end{lemma}

From \cref{eq.complete-powersum-relation}, we obtain
\begin{lemma}\label{lemma.exponential-power-sum}
If $R$ is torsion-free then, for $x \in \mathfrak{m}_A$,  
\[\Exp_\sigma(x)=\prod_{i \geq 1} \exp\left( \frac{p_i \circ x}{i} \right) \textrm{ where $\exp(t):=\sum_{i\geq 0} \frac{t^i}{i!}$.} \]
\end{lemma}

\begin{definition}
We write $\Log_{\sigma}: 1+\mf{m}_A \rightarrow \mf{m}_A$ for the inverse map to $\Exp_\sigma$, which exists by \cref{lemma.expsigma-group-iso}. 
\end{definition}

\begin{lemma}\label{lemma.lambda-log}
If $R$ is a $\lambda$-ring then, for $x\in 1 + \mf{m}_A$, 
\[ \Log_\sigma(x) = \sum \ell_i \circ (x-1) \textrm{ for } \ell_i=\frac{-1}{i}\sum_{d|i} \mu(d) (-p_d)^{i/d} \in \Lambda\] 
\end{lemma}
\begin{proof} This follows from a plethystic inversion formula of Stanley \cite[Exercise 7.88e on p.479]{Stanley.EnumerativeCombinatoricsVol2} (see also \cite[Prop 8.6]{GetzlerKapranov.ModularOperads} and \cite[Prop 2.3.7]{BergstromDiaconuPetersenWesterland.HyperellipticCurvesTheScanningMapAndMomentsOfFamiliesOfQuadraticLFunctions}).
\end{proof}

\begin{remark}
Everything in the present section works for symmetric power series rings in arbitrary variables over pre-$\lambda$ rings, but we only need the countable case. In fact, the plethystic exponential and logarithm can be treated more generally for filtered pre-$\lambda$ rings, extending the case of filtered $\lambda$-rings in \cite{BergstromDiaconuPetersenWesterland.HyperellipticCurvesTheScanningMapAndMomentsOfFamiliesOfQuadraticLFunctions}.
\end{remark}

\subsection{Powers}\label{ss.powers}
Using the exponential we can also define powers. As in the previous subsection, we let 
\[ A= R[[\ul{t}_{\mbb{N}}]] \textrm{ or } \Lambda_R^\wedge.\]

\begin{definition}\label{def.powers}
    For $f \in 1+ \mf{m}_A$ and $N \in A$,
    \[ f^N:=\Exp_{\sigma}(N\Log_{\sigma}(f)). \]
\end{definition}

\begin{lemma}\label{lemma.powers-mult}
$f^{N_1+N_2}=f^{N_1}f^{N_2}$ and $(fg)^N=f^Ng^N$.
\end{lemma} 
\begin{proof}
    Immediate from \cref{lemma.expsigma-group-iso}. 
\end{proof}

\begin{remark}
This definition recovers the power structures of \cite{GuseinZadeLuengoEtAl.OnThePowerStructureOverTheGrothendieckRingOfVarietiesAndItsApplications} (cf. also \cite[\S2.6]{Howe.MotivicRandomVariablesAndRepresentationStabilityIConfigurationSpaces}) --- indeed, the compatibility follows from continuity of $\Exp_\sigma$ and $\Log_\sigma$ and the computation 
\[ \left(\frac{1}{1-\ul{t}^{\ul{j}}}\right)^N= \Exp_{\sigma}(N\ul{t}^{\ul{j}})=\sum_{k \geq 0} (h_k \circ N) \ul{t}^{\ul{j}k}, \]
which is immediate from the formulas in \cref{lemma.monoid-pleth-is-pre-lambda-and-formulas}. 
\end{remark}

\begin{example}\label{example.expsigmaNh1-as-power}
    \[ (1+h_1+h_2 +\ldots)^N=(\Exp_{\sigma}(h_1))^N=\Exp_{\sigma}(Nh_1). \]
\end{example}

We note that, if $f \in \Lambda^\wedge$, then the assignment 
$N \mapsto f^{N}=\Exp_{\sigma}(N \Log_{\sigma}(f))$ is functorial in maps of pre-$\lambda$ rings, in the sense that if $\varphi: R_1 \rightarrow R_2$ is a map of pre-$\lambda$ rings, then $f^{\varphi(N)}$ is obtained from $f^N$ by applying the continuous extension of $\varphi \otimes \Id$ from $\Lambda_{R_1}^\wedge$ to $\Lambda_{R_2}^\wedge$. We claim that it follows that there are unique elements $a_{\tau,f} \in \Lambda$ such that, for any pre-$\lambda$ ring $R$ and $N \in R$ 
\[ f^{N} = \sum (a_{\tau, f} \circ N) m_\tau. \]
Indeed, since $\Lambda$ is the free pre-$\lambda$ ring on one generator ($h_1=e_1=p_1$), taking $R$ to be $\Lambda$ itself and writing its variables as $s_1,s_2, \ldots$, we can take the unique expansion in $\Lambda_{\Lambda}^\wedge$ with respect to the $m_\tau$
\[ f^{h_1(s_1,s_2,\ldots)}=\sum a_{\tau,f}(\ul{s}) m_\tau. \]
Since $a_{\tau,f}(\ul{s})=a_{\tau,f} \circ h_1(\ul{s})$, the functoriality of the construction shows these are the desired elements. 

\begin{example}
    Combining \cref{example.expsigmaNh1-as-power} and \cref{example.how-to-expand-expsigmah1}, we find 
    \[ a_{\tau,(1+h_1+h_2+\ldots)}=h_\tau. \]
\end{example}

\begin{example}\label{example.configuration-symmetric-functions}
    We let $c_\tau:=a_{\tau,(1+h_1)}$, so that 
    \[ (1+h_1)^N=\sum_\tau (c_\tau \circ N) m_\tau. \]
    This recovers the configuration symmetric functions denoted $c_\tau$ in \cite[Lemma 3.5 and surrounding discussion]{Howe.MotivicRandomVariablesAndRepresentationStabilityIConfigurationSpaces}. We can express the $h_\tau$ in terms of the $c_\tau$ as follows: we consider $\Lambda[[\ul{v}_{\mbb{N}}]]$, but where we name our countable family of variables $v_{\ul{a}}$, for $\ul{a} \in \mathbb{Z}_{\geq 0}^\mathbb{N}\backslash \vec{0}$, and where we write $s_i$ for the variables in $\Lambda$. Then, if we specialize\footnote{One can make this specialization using \cref{lemma.monoid-ring-functoriality} because it is induced by the map of monoids $\oplus_{\mathbb{Z}_{\geq 0}^\mathbb{N}} \mbb{Z}_{\geq 0} \rightarrow \oplus_{i \geq 1} \mbb{Z}_{\geq 0}$ sending $e_{\ul{a}}$ to $\sum a_i e_i$, for $e_{\ul{a}}$ (resp. $e_i$) the vector that is $1$ in the ${\ul{a}}$th (resp. $i$th) component and zero in all other components, which has finite fibers.} 
    \begin{equation}\label{eq.pre-specialization-config-exp} (1+\sum_{\ul{a}} v_{\ul{a}})^{h_1(\ul{s})} \end{equation}
    by substituting $\prod t_i^{a_i}$ for $v_{\ul{a}}$, then we obtain
    \[ (1+h_1 + h_2 + \ldots)^{h_1(\ul{s})}=\Exp_{\sigma}( h_1(\ul{s})h_1(\ul{t}))= \sum h_\tau(\ul{s}) m_\tau,\]
    On the other hand, if we expand \cref{eq.pre-specialization-config-exp} first and then specialize, we obtain    
 obtains expressions for $h_\tau$ as $\mbb{Z}$-linear combinations of the $c_\tau$. In particular, one finds $h_\tau=c_\tau + \ldots$ where the other partitions $\tau'$ appearing with non-trivial coefficients have $|\tau'|<|\tau|$, thus the transition matrix from the $c_\tau$ to the $h_\tau$ can be inverted and we find that the $c_\tau$ are also a basis for $\Lambda$ (as explained from an essentially equivalent geometric perspective in \cite{Howe.MotivicRandomVariablesAndRepresentationStabilityIConfigurationSpaces}). We note that this computation is also closely related to constructions used in the theory of motivic Euler products \cite{Bilu.MotivicEulerProductsAndMotivicHeightZetaFunctions}. 
\end{example}

\subsection{The Hall inner product}\label{ss.HallInnerProduct}

The Hall inner product on $\Lambda$ is the inner product $\langle, \rangle$ determined by
\begin{equation}\label{eq.hall-monomial-complete} \langle m_\tau, h_{\mu}\rangle = \delta_{\tau\mu}. \end{equation}
For $s_\tau$ the Schur symmetric functions (see, e.g., \cite[\S I.3]{Macdonald.SymmetricFunctionsAndHallPolynomials}), it satisfies
\[ \langle s_\tau, s_{\mu}\rangle=\delta_{\tau \mu}. \]
On power sum symmetric functions it satisfies
\begin{equation}\label{eq.power-sum-inner-product} \langle p_\tau,p_{\mu}\rangle=z_\tau \delta_{\tau\mu} \textrm{ where } z_\tau=\prod_i \tau(i)! i^{\tau(i)},\end{equation}
recalling that $\tau(i)$ denotes the multiplicity of $i$ in $\tau$. 

For $R$ any ring, the Hall inner product induces an isomorphism of $R$-modules 
\begin{equation}\label{eq.hall-inner-product-isomorphism} \Lambda_R^\wedge \xrightarrow{f \mapsto \langle f, \bullet\rangle} \Hom_{\mathbb{Z}}(\Lambda, R). \end{equation}

\subsection{The involution $\omega$}\label{ss.involution-omega}

\newcommand{\even}{\mathrm{even}}
We write $\omega$ for the standard involution on $\Lambda$, so that 
\[ e_i \overset{\omega}{\leftrightarrow} h_i,\; p_i \overset{{\omega}}{\leftrightarrow} (-1)^{i-1} p_i,\; s_{\tau}\overset{{\omega}}{\leftrightarrow} s_{\tau'}, \]
where $\tau'$ is the conjugate partition to $\tau$. For any ring $R$, $\omega$ extends to an involution of $\Lambda_R^\wedge$ which we continue to write as $\omega$. 

We say $f \in \Lambda_R^\wedge$ is even if it is a formal sum of monomials of even degree. It is immediate that the even elements are a subring and, if $R$ is a pre-$\lambda$ ring, even a pre-$\lambda$ subring. 
 
\begin{lemma}\label{lemma.lambda-involution-even}
Let $R$ be a torsion-free pre-$\lambda$ ring. If $f \in \mf{m}_{\Lambda_R^{\wedge}}$ is even, then 
\[ \Exp_{\sigma}(\omega(f))=\omega(\Exp_{\sigma}(f)) \textrm{ and } \Log_{\sigma}(\omega(1+f))=\omega(\Log_{\sigma}(1+f)). \]
\end{lemma}
\begin{proof}
It suffices to treat $\Exp_{\sigma}$, since $\Log_{\sigma}$ is its inverse.

We can write $f=\sum_{|\lambda| \textrm{ even}} a_\lambda p_\lambda$ where each $a_\lambda \in R$.  
Then 
\[ \omega(\Exp_{\sigma}(f))=\prod \omega(\Exp_{\sigma}(a_\lambda p_\lambda)) \textrm{ and } \Exp_{\sigma}(\omega(f))=\prod \Exp_{\sigma}( a_\lambda \omega(p_{\lambda})) \]
so it suffices to treat the case $f=a_\lambda p_\lambda$, where $|\lambda|$ is even. 

By \cref{lemma.exponential-power-sum},
\[ \Exp_\sigma(a_\lambda p_\lambda) = \prod_{i \geq 1} \sum_{k \geq 0} \frac{(p_i \circ a_\lambda)^k p_{i\lambda}^k}{i^k k!} \textrm{ and } \Exp_\sigma(a_\lambda \omega(p_\lambda)) = \prod_{i \geq 1} \sum_{k \geq 0} \frac{(p_i \circ a_\lambda)^k (p_i \circ \omega(p_{\lambda}))^k}{i^k k!}.  \]
Thus it suffices to show that, for each $i$, 
\[ \omega(p_{i\lambda})=p_i\circ \omega(p_\lambda). \]
The left-hand side expands as $(-1)^{i|\lambda|-||\lambda||}p_{i\lambda}$ while the right-hand side expands to $p_i \circ (-1)^{|\lambda|-||\lambda||}p_\lambda=(-1)^{|\lambda|-||\lambda||} p_{i\lambda}$, so this is equivalent to showing  
\[ i|\lambda|-||\lambda|| \equiv |\lambda|-||\lambda|| \mod 2. \]
The two sides differ by $(i-1)|\lambda|$, which is even since $|\lambda|$ is even by assumption. 
\end{proof}

\section{Pre-$\lambda$ probability}\label{s.lambda-prob}
In this section we develop the basic theory of probability in $\lambda$-rings. In fact, we work everywhere with pre-$\lambda$ rings: although all of the rings of random variables in the examples of interest in the present work are $\lambda$-rings, in \cite{BiluHowe.MotivicRandomVariables} we will want to work also with rings of random variables given by relative Grothendieck rings of varieties, and for these the natural pre-$\lambda$ ring structure is not known to be a $\lambda$-ring structure. Working in the added generality of pre-$\lambda$ rings costs us nothing, as all of the basic theory depends only on the properties of a pre-$\lambda$ ring. 

After giving basic definitions in \cref{ss.basic-definition-pre-lambda-prob}, we arrive at the meat of the theory in our study of moment generating functions in \cref{ss.moment-generating-functions}. We establish there the basic properties of the $\sigma$-moment generating function, and also explain how to construct more general moment generating functions using the powers of \cref{ss.powers}, including a falling (or factorial) moment generating function that plays the same role as the classical falling moment generating function. We note that, for random variables of a combinatorial nature, the falling moment generating function often takes a simpler form (see \cref{remark.falling-vs-sigma}). 

In \cref{ss.some-distributions}, we discuss binomial and Poisson distributions and their $\sigma$-moment and falling moment generating functions. It seems remarkable to us that there is a single definition in the $\lambda$-setting that matches so closely with the simple forms of the moment generating functions of binomial and Poisson distributions in classical probability theory. We note that these $\lambda$-analogs of binomial distributions are ubiquitous in computations in arithmetic statistics because they arise from summing identically distributed independent Bernoulli random variables (just as in the case of classical binomial random variables, but see \cref{remark.independence-and-summing}). 

Finally, in \cref{ss.convergence}, we make our notions of convergence of sequences of random variables precise for use in our main results.  

\begin{remark}
    A shadow of the theory described here was developed already in \cite[\S4]{Howe.MotivicRandomVariablesAndRepresentationStabilityIConfigurationSpaces}. The systematic use of the pre-$\lambda$ ring structure and the $\sigma$-moment generating function makes the version developed here much more robust. 
\end{remark}

\subsection{Basic definitions}\label{ss.basic-definition-pre-lambda-prob}
\begin{definition}\hfill
\begin{enumerate}
\item A pre-$\lambda$ probability space is a pair $(R, \mbb{E})$ where $R$ is a pre-$\lambda$ ring (the random variables) and $\mbb{E}$ is a $\mbb{Z}$-linear functional on $R$ (the expectation functional) valued in another ring $C$ and satisfying $\mbb{E}[1_R]=1_C$.
\item For $(R,\mbb{E})$ a pre-$\lambda$ probability space and $X \in R$ a random variable, the \emph{$\Lambda$-distribution of $X$} is the $C$-valued functional $\mu_{X}$ on $\Lambda$ sending $f\in \Lambda$ to $\mbb{E}[f \circ X]$. 
\end{enumerate}
When $R$ is a $\lambda$-ring, we will sometimes write $\lambda$-probability space, etc., in place of pre-$\lambda$ probability space, etc.  
\end{definition}

\begin{definition}
For $(R,\mbb{E})$ a pre-$\lambda$ probability space and $X_1, X_2, \ldots X_n \in R$ random variables,
\begin{enumerate}
\item the \emph{joint distribution} of $X_1, X_2, \ldots, X_n$ is the the element $\mu_{X_1,\ldots,X_n}$ of $\Hom_{\mathbb{Z}}(\Lambda^{\otimes{n}}, C)$ determined by 
\[ f_1 \otimes \ldots \otimes f_n \mapsto 
\mathbb{E}[(f_1 \circ X_1)(f_2 \circ X_2) \ldots(f_n \circ X_n)], \textrm{ and } 
\]
\item we say $X_1, X_2, \ldots X_n$ are independent if $\mu_{X_1, \ldots, X_n}$ is equal to the element of $\Hom_{\mathbb{Z}}(\Lambda^{\otimes{n}}, C)$ determined by 
\[ f_1 \otimes \ldots \otimes f_n \mapsto \mathbb{E}[f_1 \circ X_1]\cdot \mathbb{E}[f_2 \circ X_2] \cdot \ldots \cdot \mathbb{E}[f_n \circ X_n].\]
\end{enumerate}
More generally, if $(X_i)_{i\in I}$ is a possibly infinite family of random variables, we say $(X_i)_{i \in I}$ is independent if any finite subset is independent. 
\end{definition}

\subsection{Moment generating functions}\label{ss.moment-generating-functions}

The main tool we will use for studying $\Lambda$-distributions is the $\sigma$-moment generating function. 

\begin{definition} For $(R,\mbb{E}: R \rightarrow C)$ a pre-$\lambda$ probability space, the $\sigma$-moment generating function of $X \in R$ is 
\[ \mbb{E}[\Exp_\sigma(X h_1)]  \in \Lambda_C^\wedge,\]
where here $X h_1$ is viewed as an element of $\Lambda_R^\wedge$ and $\mbb{E}$ is applied coefficient-wise, i.e. it is  the continuous extension of $\Id \otimes \mbb{E}$ on $\Lambda_R=\Lambda \otimes R$ to $\Lambda_R^\wedge$.  
\end{definition}

By \cref{example.how-to-expand-expsigmah1}, we find
\begin{equation}\label{eq.sigma-moment-expansion} \mbb{E}[\Exp_{\sigma}(Xh_1)]=\sum_\tau \mbb{E}[h_\tau \circ X]m_\tau. \end{equation}

Because the $m_\tau$ and the $h_\tau$ are both $\mbb{Z}$-bases of $\Lambda$, it follows that the $\sigma$-moment generating function of $X$ determines its $\Lambda$-distribution and vice-versa. The $\Lambda$-distribution can be obtained from the $\sigma$-moment generating function by the Hall inner product via the isomorphism \cref{eq.hall-inner-product-isomorphism}:
\begin{lemma}\label{lemma.extract-distribution-via-hall} Let $(R,\mbb{E}: R \rightarrow C)$ be a pre-$\lambda$ probability space. For any $f \in \Lambda$ and $X \in R$, 
\[ \mu_X(f)=\mbb{E}[f \circ X]=\langle \mbb{E}[\Exp_{\sigma}(Xh_1)], f \rangle. \]
\end{lemma}
\begin{proof}
It suffices to verify the two sides agree on $h_\gamma$ for all partitions $\gamma$. We obtain this by invoking \cref{eq.sigma-moment-expansion} and \cref{eq.hall-monomial-complete}: 
\[ \langle \mbb{E}[\Exp_{\sigma}(Xh_1)], h_\gamma \rangle = \langle\sum_\tau  \mbb{E}[h_\tau \circ X]m_\tau, h_\gamma \rangle = \sum_\tau \langle \mbb{E}[h_\tau \circ X]m_\tau, h_\gamma \rangle= \mbb{E}[h_\gamma \circ X]. \]
\end{proof}

We can also construct other moment generating functions: for $f\in 1+\mf{m}_{\Lambda^\wedge}$, with notation as in \cref{ss.powers} we obtain a moment generating function 
\begin{equation}\label{eq.general-moment-generating-expansion} \mbb{E}[f^X]=\sum_\tau (a_{\tau,f}\circ X) m_\tau. \end{equation}
\begin{example}\label{example.power-moment-generating functions}\hfill
    \begin{enumerate}
        \item If $f=1+h_1+h_2+\ldots=\Exp_{\sigma}(h_1)$, then, by \cref{example.expsigmaNh1-as-power}, $f^X=\Exp_{\sigma}(X)$ and we recover the $\sigma$-moment generating function.
        \item If $f=(1+h_1)$, then, with notation as in \cref{example.configuration-symmetric-functions}, 
        \[ \mbb{E}[(1+h_1)^{X}]=\sum_\tau \mbb{E}[c_\tau \circ X]m_\tau. \]
        We refer to this as the \emph{falling moment generating function of $X$}. Since, as explained in \cref{example.configuration-symmetric-functions}, the $c_\tau$ are a $\mbb{Z}$-basis of $\Lambda$, we find that to give a $\Lambda$-distribution it is also equivalent to give its falling moment generating function. 
    \end{enumerate}
\end{example}

\begin{lemma}\label{lemma.independent-moment-generating-functions}
If random variables $X_1$ and $X_2$ in $R$ are independent then, for any $f \in 1+\mf{m}_{\Lambda^\wedge}$, 
\[ \mbb{E}\left[f^{(X_1+X_2)}\right]=\mbb{E}[f^{X_1}]\mbb{E}[f^{X_2}]. \]
In particular, taking $f=1+h_1+h_2+\ldots=\Exp_{\sigma}(h_1)$,
\[ \mbb{E}[\Exp_\sigma((X_1+X_2)h_1)]=\mbb{E}[\Exp_\sigma(Xh_1)]\mbb{E}[\Exp_\sigma(Xh_1)]. \]
\end{lemma}
\begin{proof}
    By \cref{lemma.powers-mult}, 
    \[ \mbb{E}\left[f^{(X_1+X_2)}\right] = \mbb{E}[f^{X_1}f^{X_2}]. \]
    Expanding as in \cref{eq.general-moment-generating-expansion} and applying the definition of independence, we find that we can commute the expectation with the product. 
\end{proof}

\begin{remark}\label{remark.falling-vs-sigma}
In practice, it is simpler to work with the falling moment generating function for random variables of a combinatorial nature (e.g. valued in finite sets or varieties over a fixed variety), and with the $\sigma$-moment generating function for random variables of a linear nature (e.g. valued in multisets of complex numbers or sheaves on a fixed variety). We will default to working with the $\sigma$-moment generating function when we expect to mix the two types of random variable: in the combinatorial setting the $\sigma$-moment generating function is typically only mildy more complicated (see, e.g., the difference in \cref{def.binomial-and-poisson}), whereas falling moment generating functions tend to be completely intractable in the linear setting. 
\end{remark}

\subsection{Some $\Lambda$-distributions}\label{ss.some-distributions}

We recall that, in classical probability theory, a binomial random variable $X$ with parameters $p$ and $N$ is the sum of $N$ independent Bernoulli random variables, each of which takes values $1$ and $0$ with probability $p$ and $(1-p)$, and has falling moment and moment generating functions
\[ \mbb{E}[(1+t)^X]=(1+pt)^N \textrm{ and } \mbb{E}[e^{Xt}]=(1+p(e^t-1))^N. \]
We also recall that, in classical probability theory, a Poisson random variable $X$ of mean $\mu$ has falling moment and moment generating functions 
\[ \mbb{E}[(1+t)^X]=e^{\mu t} \textrm{ and } \mbb{E}[e^{Xt}]=e^{\mu(e^t-1)}. \]

The following shows that there are perfect $\Lambda$-analogs of these distributions.
\begin{lemma}\label{lemma.poisson-and-binomial-distributions}
    Let $R$ be a pre-$\lambda$ ring.
    \begin{enumerate}
        \item Let $p, N \in R$. The unique $\Lambda$-distribution with falling moment generating function 
        \[ ``\mbb{E}[(1+h_1)^X]" = (1+ph_1)^N \]
        has $\sigma$-moment generating function
        \[ ``\mbb{E}[\Exp_{\sigma}(Xh_1)]" = (1+p\left(\Exp_{\sigma}(h_1)-1\right))^N=(1+p(h_1+h_2+\ldots))^N. \]
        \item Let $\mu \in R$. The unique $\Lambda$-distribution with falling moment generating function 
        \[ ``\mbb{E}[(1+h_1)^X]" = \Exp_{\sigma}(\mu h_1)\] has $\sigma$-moment generating function
        \[ ``\mbb{E}[\Exp_{\sigma}(Xh_1)]" = \Exp_{\sigma}(\mu(\Exp_{\sigma} (h_1)-1)) = \Exp_{\sigma}(\mu(h_1+h_2+\ldots)).\] 
    \end{enumerate} 
    Here we put quotes to indicate that we have not actually fixed a random variable $X$ with the given $\Lambda$-distribution. 
\end{lemma}
\begin{proof}
    Both follow by making the same substitution as in \cref{example.configuration-symmetric-functions}.
\end{proof}

\begin{definition}\label{def.binomial-and-poisson}
The following notions are well-defined by \cref{lemma.poisson-and-binomial-distributions}.
    \begin{enumerate}
        \item For $R$ a pre-$\lambda$ ring and $p,n \in R$, the $(p,N)$-binomial $\Lambda$-distribution is the unique $\Lambda$-distribution with falling moment generating function
    \[ (1+ph_1)^N \]
    and $\sigma$-moment generating function
    \[ 1+p\left(\Exp_{\sigma}(h_1)-1\right))^N=(1+p(h_1+h_2+\ldots))^N. \]
    \item For $R$ a pre-$\lambda$ ring and $\mu \in R$, the $\mu$-Poisson $\Lambda$-distribution is the unique $\Lambda$-distribution with falling moment generating function
    \[ \Exp_{\sigma}(\mu h_1) \]
    and $\sigma$-moment generating function 
    \[ \Exp_{\sigma}(\mu(\Exp_{\sigma} (h_1)-1)) = \Exp_{\sigma}(\mu(h_1+h_2+\ldots)). \]
    \end{enumerate}
\end{definition}

\begin{remark}\label{remark.independence-and-summing}
The random variables we will encounter that have binomial $\Lambda$-distributions can be interpreted as sums of $N$ independent Bernoulli random variables. When $N$ is a positive integer, this makes sense already using \cref{lemma.independent-moment-generating-functions}. When $N$ is the class in $W(\mathbb{C})$ of an admissible $\mathbb{Z}$-set (see \cref{s.abstract-pc}), this can be made precise using the results of \cite{Howe.RandomMatrixStatisticsAndZeroesOfLFunctionsViaProbabilityInLambdaRings}. In fact, as we will explain in \cite{BiluHowe.MotivicRandomVariables}, this can be made precise more generally whenever $R$ is a
Grothendieck ring decategorifying a category of algebraic varieties and $N$ is an effective class. 

We note also that we can view a $\mu$-Poisson $\Lambda$-distribution as a limit of binomial distributions with parameters $p$ and $N$ such that $pN \rightarrow \mu$. This will also be explained precisely and treated in examples in \cite{BiluHowe.MotivicRandomVariables}. 
\end{remark}

\subsection{Convergence}\label{ss.convergence}

\begin{definition}
When $\mbb{E}$ is valued in a Hausdorff topological ring, we say a sequence  of  tuples $(X_{1,d}, \ldots, X_{n,d})$ of $n$ random variables indexed by natural numbers $d$ converges in distribution to a joint $\Lambda$-distribution $\mu$ if, for each $f \in \Lambda^{\otimes n}$, 
\[  \lim_{d \rightarrow \infty} \mu_{X_{1,d}, X_{2,d}, \ldots, X_{n,d}}(f) = \mu(f). \]
It suffices to consider just $f$ of the form $f=f_1 \otimes f_2 \otimes \ldots \otimes f_n$, for which this says
\[ \lim_{d \rightarrow \infty} \mbb{E}[f_1 \circ X_{1,d} \cdot f_2 \circ{X}_{2,d} \cdot \ldots \cdot f_n \circ X_{n,d}]=\mu(f). \]
In fact, it suffices to consider just these simple tensors for the $f_i$ ranging over any fixed basis of $\Lambda$. 
\end{definition}

\begin{definition}
When $\mbb{E}$ is valued in a Hausdorff topological ring, we say a sequence indexed by natural numbers $d$ of  tuples $(X_{i,d})_{i \in I}$  of random variables is asymptotically independent if there exists for each $i \in I$ a $\Lambda$-distribution $\mu_{i,\infty}$ such that, for any finite collection of distinct indices $i_1, \ldots, i_n \in I$, $(X_{i_1,d}, \ldots, X_{i_n,d})$ converges in distribution to the joint $\Lambda$-distribution $\mu_{i_1,\infty} \otimes \mu_{i_2,\infty} \otimes \ldots \otimes \mu_{i_n,\infty}$. 
\end{definition}

Arguing as in \cref{lemma.independent-moment-generating-functions}, we obtain
\begin{lemma}\label{lemma.asympt-ind-moment-gen}
    If $(X_{1,d}, \ldots, X_{n,d})$ is a sequence of asymptotically independent random variables, then 
    \[ \lim_{d\rightarrow \infty} \mbb{E}[\Exp_{\sigma}( (X_{1,d}+\ldots +X_{n,d})h_1)]=\prod_{i=1}^n \lim_{d\rightarrow \infty}\mbb{E}[\Exp_{\sigma}(X_{i,d})h_1] \]
    where the limits in $\Lambda_{C}^\wedge$ can be taken individually in $C$ for each coefficient of $m_\tau$. 
\end{lemma}

\section{Random matrices}\label{s.random-matrices}

In this section we prove our main result on random matrices, \cref{main.random-matrix} (see \cref{theorem.body-random-matrices}), and give some complements. First, in \cref{ss.random-matrix-probability-space}, we situate the result more squarely within the theory of $\lambda$-probability that was developed in \cref{s.lambda-prob}. Then, in \cref{ss.random-matrix-mg-functions}, we prove the main result, admitting a computation of invariants described in \cref{ss.invariants}. In \cref{ss.traces}, we deduce as a corollary (\cref{corollary.trace-distributions}) the Diaconis-Shahshahani theorem on traces of powers of random matrices in classical groups. Finally, in \cref{ss.cycle-counting}, we explain the relation between the symmetric group case of \cref{main.random-matrix}/\cref{theorem.body-random-matrices} and cycle-counting. 

\subsection{The probability space}\label{ss.random-matrix-probability-space}
Let $G$ be a compact topological group. The Grothendieck ring $K_0(\Rep_{\mathbb{C}} G)$ of complex representations of $G$ admits a natural structure as a $\lambda$-ring such that $h_k \circ [V]=[\Sym^k V]$. For $V$ a representation of $G$ we write $\chi_V$ for its character, $\chi_V(g)=\mathrm{Tr}(g \textrm{ acting on V})$, and we write $\mbb{E}$ for the expectation functional $K_0(\Rep G) \rightarrow \mathbb{Z}$ which sends $[V]$ to 
\[ \int_G \chi_V dg \textrm{ for $dg$ the Haar measure on $G$}.\]
These integrals are a priori complex numbers, but in fact they evaluate to integers since, by the usual orthogonality relations, 
\begin{equation}\label{eq.rep-expectation-is-dimension} \mathbb{E}[[V]] = \int_G \chi_V dg = \dim_{\mathbb{C}} V^G. \end{equation}
In other words, viewing $K_0(\Rep_{\mathbb{C}} G)$ as the free $\mathbb{Z}$-module on the isomorphism classes of irreducible representations of $G$, $\mbb{E}$ is the unique functional sending the trivial representation to $1$ and any other irreducible representation to $0$.  We consider the $\lambda$-probability space $(K_0(\Rep_{\mathbb{C}} G), \mbb{E})$. 

\begin{remark}\label{remark.repG-embedding}
The assignment $[V] \mapsto \chi_V$ extends to an embedding of $K_0(\Rep_{\mathbb{C}} G)$ into the ring of continuous $\mathbb{C}$-valued functions on $G$, and the classical $\mbb{C}$-valued expectation functional on continuous functions on $G$ given by integration against Haar measure restricts to our $\mathbb{Z}$-valued measure on $K_0(\Rep_{\mathbb{C}} G)$ under this embedding. 

Studying $K_0(\Rep_{\mathbb{C}} G)$ with its $\mathbb{Z}$-valued measure will suffice for the results we wish to prove here. However, because the ring of continuous $\mathbb{C}$-valued functions on $G$ is not naturally a $\lambda$-ring, to bring more general problems in random matrix theory into the realm of $\lambda$-probability it is natural to study instead the $\lambda$-ring of continuous $W(\mathbb{C})\cong\prod_{i=1}^\infty \mathbb{C}$-valued functions on $G$ (where $W(\mathbb{C})$ is equipped with the product topology). The expectation on this ring is valued in $W(\mathbb{C})$ and is given by integration in each term, and the $\lambda$-ring structure is the pointwise structure $(f \circ F)(g)=f \circ (F(g))$. We claim the ring of random variables $K_0(\Rep G)$ embeds naturally into $\mathrm{Cont}(G, W(\mathbb{C}))$ compatibly with this $\lambda$-ring structure and with the expectation functionals (under the unique map $\mathbb{Z} \rightarrow W(\mathbb{C})$). Indeed, we can first view $[V]$ as the function sending $g \in G$ to its multiset of eigenvalues of $g$ acting on $V$ to obtain a function valued in $\mathbb{Z}[\mathbb{C}]$, and then compose with the usual embedding of $\mathbb{Z}[\mathbb{C}]$ into $W(\mathbb{C})$ (which sends $[z]$ to $(z, z^2, z^3, \ldots)$). Thus $[V]$ is mapped to the function which sends $g$ to $(\chi_V(g^i))_i$, and it is straightforward to check this is a map of $\lambda$-rings. We will not use this further here, but other rings of $W(\mbb{C})$-valued functions play an essential role in our study of $L$-functions starting in \cref{s.abstract-pc}. 
\end{remark}

\subsection{Random matrix moment generating functions}\label{ss.random-matrix-mg-functions}

We now rephrase \cref{main.random-matrix} in the language of \cref{ss.random-matrix-probability-space}, then give a proof. In the following, we consider the multiplicity filtration on $\Lambda^\wedge$
\[ \Lambda^\wedge_{\mr{mult}>i} = \left\{ \sum_{||\tau||>i} a_\tau m_\tau\, |\, a_\tau \in \mathbb{Z} \right\} \]
where we recall that, for a partition $\tau$, $||\tau||$ is the number of distinct entries in $\tau$ so that, e.g., $||(3,1)||=2$. Similarly, we consider the disjoint multiplicity filtration on $(\Lambda \otimes \Lambda)^\wedge$,
\[ (\Lambda \otimes \Lambda)^\wedge_{\mr{mult}>i} = \left\{ \sum_{||\tau||> i \textrm{ or } ||\overline{\tau}||>i} a_{\tau,\overline{\tau}} m_\tau m_{\overline{\tau}}\, |\, a_{\tau,\overline{\tau}} \in \mathbb{Z} \right\}.\]

We will also consider the degree filtration on $\Lambda^{\wedge}$, which can be written
\[ \Lambda^{\wedge}_{\mr{deg}>i}=\left\{ \sum_{|\tau|>i} a_\tau m_\tau\, |\, a_\tau \in \mathbb{Z} \right\} \]
where we recall $|\tau|$ is the sum of the partition so that, e.g., $|(3,1)|=4.$

\begin{theorem}\label{theorem.body-random-matrices}\hfill
\begin{enumerate}
\item For $n \geq 1$, let $V_n$ be the standard representation of $O(n)$ on $\mathbb{C}^n$. Then, 
\[ \mbb{E}[\Exp_{\sigma}([V_n]h_1)]\equiv \Exp_\sigma(h_2)  \mod \Lambda^\wedge_{\mr{mult}>n} \]
where we note 
\[ \Exp_{\sigma}(h_2)=\prod_{i\leq j}\frac{1}{1-t_it_j} \in \Lambda^\wedge \subseteq \mathbb{Z}[[\ul{t}_{\mbb{N}}]].\] 
\item For $n \geq 1$ even, let $V_n$ be the standard representation of $\Sp(n)$ on $\mathbb{C}^n$. Then,
\[ \mbb{E}[\Exp_{\sigma}([V_n]h_1)]=\Exp_{\sigma}(e_2) \mod \Lambda^\wedge_{\mult>n} \]
where we note 
\[ \Exp_{\sigma}(h_2)=\prod_{i<j}\frac{1}{1-t_it_j} \in \Lambda^\wedge \subseteq \mathbb{Z}[[\ul{t}_{\mbb{N}}]]. \]
\item For $n \geq 1$, let $V_n$ be the standard permutation representation of $\Sigma_n$ on $\mathbb{C}^n$. Then,
\[ \mbb{E}[\Exp_{\sigma}([V_n]h_1)] \equiv \Exp_{\sigma}(\Exp_{\sigma}(h_1)-1) \mod \Lambda^{\wedge}_{\deg>n} \]
where we note
\begin{align*}\Exp_{\sigma}(\Exp_{\sigma}(h_1)-1)&=\Exp_{\sigma}(h_1 + h_2 + \ldots)\\&=\prod_{\substack{\textrm{monomials $m$}\\\textrm{in $t_1, t_2, \ldots$}}} \frac{1}{1-m} \in \Lambda^\wedge \subseteq \mathbb{Z}[[\ul{t}_{\mbb{N}}]].\end{align*} 
\item For $n \geq 1$, let $V_n$ be the standard representation of $U(n)$ on $\mathbb{C}^n$. Then, 
\[ \mbb{E}[\Exp_{\sigma}([V_n]\cdot h_1 + [V_n^*]\cdot \overline{h}_1)] \equiv \Exp_{\sigma} (h_1\overline{h}_1) \mod (\Lambda\otimes_{\mathbb{Z}}  \Lambda)^\wedge_{\mult > n} \]
where we note 
\[ \Exp_{\sigma} (h_1\overline{h}_1)=\prod_{i,j} \frac{1}{1-t_i\overline{t}_j} \in (\Lambda \otimes_{\mathbb{Z}} \Lambda)^\wedge \subseteq \mathbb{Z}[[\ul{t}_{\mbb{N}}, \overline{\ul{t}}_{\mbb{N}}]]. \]
\end{enumerate}
\end{theorem}
\begin{proof}
We first discuss cases (1)-(3), writing $G_n$ for the relevant group. For $\tau=(\tau_1,\tau_2,\ldots,\tau_k)$, we write $\Sym^\tau V_n= \Sym^{\tau_1} V_n \otimes \ldots \otimes \Sym^{\tau_k} V_n$. It follows from the definition of the $\lambda$-ring structure on $\Rep_{\mathbb{C}} G$, where $h_k \circ [V]=[\Sym^k V]$ and multiplication is given by tensor product, that
\[ h_\tau \circ [V_n]=(h_{\tau_1}\ldots h_{\tau_k})\circ [V_n]=(h_{\tau_1} \circ [V_n])\ldots(h_{\tau_k} \circ [V_n])=[\Sym^{\tau} V_n].\]
Thus, using the expansion of \cref{eq.sigma-moment-expansion}, we have
\begin{align*} \mbb{E}[\Exp_{\sigma}([V_n]h_1)] & = \sum_{\tau} \mathbb{E}[h_\tau \circ [V_n]] m_\tau \\ &= 
\sum_{\tau} \mathbb{E}[ [\Sym^{\tau} V_n]] m_\tau \\
&= \sum_{\tau} \dim_{\mathbb{C}} (\Sym^{\tau} V_n)^{G_n} m_\tau.
\end{align*}
where the final equality follows from \cref{eq.rep-expectation-is-dimension}. The result in these cases then follows from \cref{lemma.invariant-theory} below after verifying that the given identities hold for the $\sigma$-exponentials in each case. This is clear by working in the non-symmetric power series ring $\mbb{Z}[[\ul{t}_{\mbb{N}}]]$ containing $\Lambda^\wedge$ of \cref{ss.symmetric-power-series}, where, e.g., in the orthogonal case we find
\[ \Exp_{\sigma}(h_2)=\Exp_{\sigma}(\sum_{i\leq j} t_it_j)=\prod_{i\leq j} \Exp_{\sigma}(t_it_j)=\prod_{i\leq j} \frac{1}{1-t_it_j}. \]

In case (4), writing $G_n=U(n)$ and expanding similarly we find 
\[ \mbb{E}[\Exp_{\sigma}([V_n]\cdot h_1 + [V_n^*]\cdot \overline{h}_1)] = \sum_{\tau, \overline{\tau}} \dim_{\mathbb{C}} \left(\Sym^\tau V_n \otimes \Sym^{\overline{\tau}}V_n^*\right)^{G_n} m_\tau \overline{m}_\tau \]
and the result in this case then follows similarly from \cref{lemma.invariant-theory}.
\end{proof}

\subsection{Computation of invariants}\label{ss.invariants}

In the proof of \cref{theorem.body-random-matrices}, we used a computation of the generating functions of invariants to be established as \cref{lemma.invariant-theory} below. In the orthogonal, symplectic, and unitary cases, it will be a simple consequence of the first fundamental theorem of classical invariant theory, which we recall as \cref{prop.first-fundamental-theorem} below. In the symmetric case we make the relevant computation in \cref{prop.symmetric-invariants}. 

\begin{proposition}[First fundamental theorem of classical invariant theory]\hfill\label{prop.first-fundamental-theorem}
\begin{enumerate}
	\item Let $V$ be the standard representation of $O(n,\mathbb{C})$ on $\mathbb{C}^n$. Then, for any $m \geq 0$, the map of algebraic varieties
	\begin{align*} V^{\oplus m} / O(n,\mathbb{C}) & \rightarrow \prod_{1\leq i \leq j \leq m}\mathbb{A}^1\\ (v_1, \ldots, v_m) &\mapsto (\langle v_i, v_j \rangle)_{i\leq j}\end{align*}
	is a closed immersion. Equivalently, the map 
	\[ \mathbb{C}[t_{ij}]_{1\leq i \leq j \leq m} \xrightarrow{t_{ij} \mapsto \langle v_i, v_j \rangle} \left(\Sym^\bullet (V^{\oplus m})^* \right)^{O(n,\mathbb{C})}\] 
	is surjective, where the action of $O(n,\mathbb{C})$ is diagonal. 
	\item Let $n$ be even and let $V$ be the standard representation of $\Sp(n, \mathbb{C})$ on $\mathbb{C}^n$. Then, for any $m\geq 0$, the map of algebraic varieties
	\begin{align*} V^{\oplus m} / \Sp(n,\mathbb{C}) &\rightarrow \prod_{1\leq i < j \leq m}\mathbb{A}^1\\ (v_1, \ldots, v_m) &\mapsto \prod_{i< j} \langle v_i, v_j \rangle\end{align*}
	is a closed immersion. Equivalently, the map 
	\[ \mathbb{C}[t_{ij}]_{1\leq i < j \leq m} \xrightarrow{t_{ij} \mapsto \langle v_i, v_j \rangle} \left(\Sym^\bullet (V^{\oplus m})^* \right)^{\Sp(n,\mathbb{C})}\] 
	is surjective, where the action of $Sp(n,\mathbb{C})$ is diagonal. 	
	\item Let $V$ be the standard representation of $\GL_n=U(n,\mathbb{C})$ on $\mathbb{C}^n$. Then, for any $m_1, m_2 \geq 0$, the map of algebraic varieties 
	\begin{align*} (V^{\oplus m_1} \oplus {V^*}^{\oplus m_2})/\GL_n \rightarrow \prod_{\substack{1 \leq i \leq m_1 \\ 1 \leq j \leq m_2}} \mathbb{A}^1 & \\
	 (v_1, \ldots, v_{m_1}, \delta_1,\ldots,\delta_{m_2})\mapsto  (\langle v_{i}, \delta_{j}\rangle)_{\substack{1 \leq i \leq m_1 \\ 1 \leq j \leq m_2}} \end{align*}
	 is a closed immersion. Equivalently, the map 
	 \[ \mathbb{C}[t_{ij}]_{\substack{1 \leq i \leq m_1 \\ 1 \leq j \leq m_2}} \xrightarrow{ t_{ij} \mapsto \langle v_i,\delta_j \rangle} \left(\Sym^{\bullet} \left(V^{\oplus m_1} \oplus {V^*}^{\oplus m_2}\right)^*\right)^{\GL_n}\]
	 is surjective, where the action of $\GL_n$ is diagonal on $V^{\oplus m_1} \oplus {V^*}^{\oplus m_2}$.
\end{enumerate}
\end{proposition}
\begin{proof}
Part (1) follows from \cite[Theorem 2.9.A]{Weyl.TheClassicalGroupsTheirInvariantsAndRepresentations} or, for a more modern reference, \cite[10.2-(1) on p.114]{KraftProcesi.ClassicalInvariantTheoryAPrimer}.
Part (2) follows from \cite[Theorem 6.1.A]{Weyl.TheClassicalGroupsTheirInvariantsAndRepresentations} or, for a more modern reference, \cite[10.3 on p.114]{KraftProcesi.ClassicalInvariantTheoryAPrimer}. Part (3) follows from \cite[Theorem 2.6.A]	{Weyl.TheClassicalGroupsTheirInvariantsAndRepresentations} or, for a more modern reference, \cite[First Fundamental Theorem for $\mathbf{GL}(V)$ in 2.1 on p.16]{KraftProcesi.ClassicalInvariantTheoryAPrimer}. 
\end{proof}

\begin{proposition}\label{prop.symmetric-invariants}
For $V=\mathbb{C}^n$ with the permutation representation of $\Sigma_n$, 
\[ \dim \left(\Sym^{k_1} V \otimes \Sym^{k_2} V \otimes \ldots \otimes \Sym^{k_\ell} V\right)^{\Sigma_n} \]
is equal to the number of partitions 
\[ (k_1, \ldots, k_\ell) = \sum_{\vec{v} \in \mathbb{Z}_{\geq 0}^\ell} a_{\vec{v}} \vec{v} \textrm{ such that $\sum_{\vec{v} \in \mathbb{Z}_{\geq 0}^\ell} a_{\vec{v}}=n$.}\]
 
\end{proposition}
\begin{proof}
Note that $\Sym^k V$ can be viewed the permutation representation associated to the $\Sigma_n$-set of labelings of $\{1, \ldots, n\}$ by elements of $\mathbb{Z}_{\geq 0}$ such that the total sum of the labels is $k$. The tensor product 
\[ \Sym^{k_1} V \otimes \Sym^{k_2} V \otimes \ldots \otimes \Sym^{k_\ell} V \]
is thus the permutation representation associated to the product of these $\Sigma_n$-sets, which can be viewed as the permutation representation associated to the $\Sigma_n$-set of labelings of $\{1, \ldots, n\}$ by vectors $\vec{v} \in \mathbb{Z}^\ell_{\geq 0}$ whose sum is $(k_1,\ldots,k_l)$. The dimension of the $\Sigma_n$-invariants is thus equal to the number of orbits in the permutation action of $\Sigma_n$ on this set of labelings. These orbits are determined uniquely by the multiplicities $a_{\vec{v}}$ with which each vector $\vec{v}$ occurs as a label, and since there are $n$ elements in the set we are labeling, the sum of the multiplicities $a_{\vec{v}}$ is equal to $n$.

\end{proof}

The following computation of generating functions that was used in the proof of \cref{theorem.body-random-matrices} is a simple consequence of the above results in invariant theory. We state the results in non-symmetric power series rings, where one has multiplicity and degree filtrations defined similarly (that restrict to the multiplicity and degree filtrations on the symmetric power series rings). 

\begin{lemma}\label{lemma.invariant-theory}\hfill
\begin{enumerate} 
\item Let $G=O(n)$ and let $V$ denote its standard representation on $\mathbb{C}^n$. Then
 \[ \sum_{\tau} \dim_{\mathbb{C}} (\Sym^{\tau} V)^{G} m_\tau \equiv \prod_{i\leq j} \frac{1}{1-t_it_j} \mod \mathbb{Z}[[\ul{t}_{\mbb{N}}]]_{\mult>n}. \]
\item For $n$ even, let $G=Sp(n)$ and let $V$ denote its standard representation on $\mathbb{C}^n$. Then
 \[ \sum_{\tau} \dim_{\mathbb{C}} (\Sym^{\tau} V)^{G} m_\tau \equiv \prod_{i< j} \frac{1}{1-t_it_j} \mod \mathbb{Z}[[\ul{t}_{\mbb{N}}]]_{\mult>n}. \]
 \item Let $G=\Sigma_n$ and let $V$ denote its standard permutation representation on $\mathbb{C}^n$. Then
  \begin{align*} \sum_{\tau} \dim_{\mathbb{C}} (\Sym^{\tau} V)^{G} m_\tau &= \prod_{\substack{\textrm{non-constant}\\\textrm{monomials $m$}\\\textrm{in $t_1, t_2, \ldots$}}}^{n-term} \frac{1}{1-m} \\
  & \equiv \prod_{\substack{\textrm{non-constant}\\\textrm{monomials $m$}\\\textrm{in $t_1, t_2, \ldots$}}} \frac{1}{1-m} \mod \mathbb{Z}[[\ul{t}_{\mbb{N}}]]_{\deg>n}\end{align*}
  where the superscript ``$n$-term" denotes that, when expanding the product, we only consider terms corresponding to at most $n$ entries not equal to 1.
  
\item Let $G=U(n)$ and let $V$ denote its standard representation on $\mathbb{C}^n$. Then 
\[\sum_{\tau, \overline{\tau}} \dim_{\mathbb{C}} \left(\Sym^\tau V_n \otimes \Sym^{\overline{\tau}}V_n^*\right)^{G_n} m_\tau m_{\overline{\tau}} \equiv  \prod_{i, j} \frac{1}{1-t_i\overline{t}_j} \mod \mathbb{Z}[[\ul{t}_{\mbb{N}}, \overline{\ul{t}}_{\mbb{N}}]]_{\mult>n}.\]
\end{enumerate}
\end{lemma}
\begin{proof}
The symmetric group case (3) is a direct rephrasing of \cref{prop.symmetric-invariants} using generating functions.  

In cases (1) and (2) (resp. (4)), taking $m=n$ (resp. $m_1=m_2=n$) in \cref{prop.first-fundamental-theorem}, we claim the given closed immersion is actually an isomorphism. Since both sides are reduced, it suffices to see that the maps are dominant. But this is easily verified by considering the pairings of the standard basis vectors. For example, in the orthogonal case (1), to hit $(a_{ij})_{i \leq j}\in \prod_{i\leq j}\mathbb{C}$ with nonzero diagonal terms $a_{ii} \neq 0$ as the image of $v_1,\ldots,v_n$, we may take 
\[ v_1=\sqrt{a_{11}} \cdot e_1, v_2=\frac{a_{12}}{\sqrt{a_{11}}} \cdot e_1 + \sqrt{a_{22}} \cdot e_2, v_3=\frac{a_{13}}{\sqrt{a_{11}}}\cdot e_1 + \frac{a_{12} - \frac{a_{13}a_{12}}{a_{11}}}{\sqrt{a_{22}}}\cdot e_2 + \sqrt{a_{33}}\cdot e_3\ldots. \]
It follows that, in each of these cases, the map of rings in \cref{prop.first-fundamental-theorem} is an isomorphism. Then, to conclude cases (1) and (2), we expand 
\[ \Sym^\bullet(V^{\oplus n}) = \sum_{(k_1,\ldots,k_n)\in \mathbb{Z}_{\geq 0}^n} \Sym^{k_1} V \otimes \Sym^{k_2} V \otimes \ldots \otimes \Sym^{k_n} V ,\]
and the computation of the generating functions then follows immediately from the isomorphisms on the rings of invariants: in the notation of \cref{prop.first-fundamental-theorem}, each monomial $\prod t_{ij}^{b_{ij}}$ in the variables $t_{ij}$ in the ring of invariants contributes a one dimensional space in $\Sym^{\vec{k}} V$ where $\vec{k}=\sum b_{ij}\cdot (e_i+e_j)$. The verification of case (4) is similar starting from the analogous decomposition of $\Sym^\bullet(V^{\oplus n} \oplus {V^*}^{\oplus n})$.

\end{proof}

\subsection{Traces of random matrices}\label{ss.traces}
We now deduce the results of Diaconis and Shahshahani \cite{DiaconisShahshahani.OnTheEigenvaluesOfRandomMatrices} on traces of Haar random matrices in classical groups  from \cref{main.random-matrix}/\cref{theorem.body-random-matrices}-(1),(2), and (4). 

\begin{corollary}[Diaconis and Shahshahani\footnote{The lower bounds on $n$ in these results as stated in \cite{DiaconisShahshahani.OnTheEigenvaluesOfRandomMatrices} are not correct and do not follow from the proofs given there. The correct bounds are stated, e.g., in \cite[Propositions 2.5 and 3.6]{Johansson.OnRandomMatricesFromTheCompactClassicalGroups}.} \cite{DiaconisShahshahani.OnTheEigenvaluesOfRandomMatrices}, Theorems 2, 4, and 6]\label{corollary.trace-distributions} 
\[ \textrm{Let } g_j(a)= \begin{cases} 0 & \textrm{if $j$ and $a$ are both odd.} \\ j^{a/2}(a-1)(a-3)\cdots 1 & \textrm{if $j$ is odd and $a$ is even.} \\
\sum_k \binom{a}{2k} j^k (2k-1)(2k-3)\cdots 1 & \textrm{if $j$ is even.}
\end{cases} \]
\begin{enumerate}
\item For $\sum i a_i \leq n$ and $M$ a Haar-random matrix in $O(n)$,
\begin{equation}\label{eq.ds-trace-orthogonal} \mbb{E}\left[\prod \mr{Tr}(M^i)^{a_i}\right] = \prod_i g_i(a_i)\end{equation}
\item For $n$ even, $\sum i a_i \leq n$, and $M$ a Haar-random matrix in $\Sp(n)$,
\begin{equation}\label{eq.ds-trace-symplectic} \mbb{E}\left[\prod\mr{Tr}(M^i)^{a_i}\right] = \prod_i (-1)^{a_i-1} g_i(a_i)\end{equation}
\item For $\sum i a_i \leq n$, $\sum_i ib_i \leq n$, and $M$ a Haar-random matrix in $U(n)$, 
\begin{equation}\label{eq.ds-trace-unitary} \mbb{E}\left[\prod\mr{Tr}(M^i)^{a_i}\mr{Tr}(\overline{M}^i)^{b_i}\right] = \begin{cases} 0 & \textrm{if $(a_i) \neq (b_i)$} \\ \prod_{i} i^{a_i} a_i! & \textrm{if $(a_i)=(b_i)$.} \end{cases}\end{equation}
\end{enumerate}
\end{corollary}
\begin{proof}[Proof of \cref{corollary.trace-distributions}]
We first treat the orthogonal case. For $V$ the standard representation of $O(n)$ and with notation as in \cref{theorem.body-random-matrices}, we have 
\[ \mbb{E}\left[\prod \mr{Tr}(M^i)^{a_i}\right] = \mbb{E}\left[\prod_i p_i^{a_i} \circ [V]\right]=\langle \mbb{E}\left[\Exp_{\sigma}([V]h_1)\right], \prod_i p_i^{a_i} \rangle, \]
where the last equality follows from \cref{lemma.extract-distribution-via-hall}. We claim that, under the condition $\sum i a_i \leq n$, \cref{theorem.body-random-matrices} implies this final inner product is equal to 
\[\langle \Exp_{\sigma}(h_2), \prod_i p_i^{a_i} \rangle. \]
Indeed, $\prod_i p_i^{a_i}$ lives in degree $\sum i a_i \leq n$, and the degree $\leq n$ part of the moment generating function and $\Exp_{\sigma}(h_2)$ agree by \cref{theorem.body-random-matrices} (since $|\tau|\geq ||\tau||$). 

We now compute this inner product by expanding $\Exp_{\sigma}(h_2)$ in terms of power sum monomials. Using \cref{lemma.exponential-power-sum}, we have
\begin{align*}
    \Exp_{\sigma}(h_2)&=\Exp_{\sigma}\left(\frac{p_1^2+p_2}{2}\right)\\
    &=\prod_{i=1}^\infty \exp\left(\frac{p_i}{i}\circ \frac{p_1^2+p_2}{2}  \right) \\
    &=\prod_{i=1}^\infty \exp\left(\frac{p_i^2+p_{2i}}{2i}\right)  \\
    &=\prod_{\textrm{$i$ odd}}\exp\left(\frac{p_i^2}{2i}\right) \prod_{\textrm{$i$ even}}\exp\left(\frac{p_i^2}{2i}\right)\exp\left(\frac{p_i}{i}\right) \\
    &=\prod_{\textrm{$i$ odd}}\left(\sum_{k\geq 0}\frac{p_i^{2k}}{(2i)^k k!}\right) \prod_{\textrm{$i$ even}}\left(\sum_{k\geq 0}\frac{p_i^{2k}}{(2i)^k k!}\right) \left(\sum_{k\geq 0}\frac{p_i^{k}}{k!i^k }\right)\\
    &=\prod_{\textrm{$i$ odd}}\left(\sum_{k\geq 0}\frac{p_i^{2k}}{(2k)!!i^k}\right) \prod_{\textrm{$i$ even}}\left(\sum_{k\geq 0}\frac{p_i^{2k}}{(2k)!!i^k}\right) \left(\sum_{k\geq 0}\frac{p_i^{k}}{k! i^k }\right)
\end{align*}
where $(2k)!!=(2k)(2k-2)(2k-4)\ldots... (2).$ Expanding, the coefficient of $\prod p_i^{a_i}$ is
\begin{align*} \prod_{i \textrm{ odd}} \left(\begin{cases} \frac{1}{a_i!! i^{a_i/2}}  & \textrm{ if $a_i$ even} \\
0 & \textrm{ if $a_i$ odd} \end{cases} \right) \prod_{i \textrm{  even}} \left(\sum_{2k \leq a_i} \frac{1}{(2k)!!(a_i-2k)!i^{a_i - k}} \right)
=\prod_i \frac{g_i(a_i)}{a_i! i^{a_i}}. \end{align*}
We then obtain the formula on the right-hand side of \cref{eq.ds-trace-orthogonal} by taking the Hall inner product with $\prod_i p_i^{a_i}$ and using \cref{eq.power-sum-inner-product} (which shows that distinct power sum monomials are orthogonal and that $\langle \prod_i p_i^{a_i}, \prod_i p_i^{a_i}\rangle=\prod_i a_i!i^{a_i}$). 

The argument in the symplectic case is the same except for the inner product computation, which differs only in a sign (alternatively, one can apply \cref{lemma.lambda-involution-even} to the computation above since $\omega(h_2)=e_2$):
\begin{align*} \Exp_{\sigma}(e_2)&=\Exp_{\sigma}\left(\frac{p_1^2-p_2}{2}\right)\\
    &=\prod_{i=1}^\infty \exp\left(\frac{p_i}{i}\circ \frac{p_1^2-p_2}{2}  \right) \\
    &=\prod_{i=1}^\infty \exp\left(\frac{p_i^2-p_{2i}}{2i}\right)  \\
    &=\prod_{\textrm{$i$ odd}}\exp\left(\frac{p_i^2}{2i}\right) \prod_{\textrm{$i$ even}}\exp\left(\frac{p_i^2}{2i}\right)\exp\left(\frac{-p_i}{i}\right) \\
    &=\prod_{\textrm{$i$ odd}}\left(\sum_{k\geq 0}\frac{p_i^{2k}}{(2i)^k k!}\right) \prod_{\textrm{$i$ even}}\left(\sum_{k\geq 0}\frac{p_i^{2k}}{(2i)^k k!}\right) \left(\sum_{k\geq 0}\frac{(-p_i)^{k}}{k!i^k }\right)\\
    &=\prod_{\textrm{$i$ odd}}\left(\sum_{k\geq 0}\frac{p_i^{2k}}{(2k)!!i^k}\right) \prod_{\textrm{$i$ even}}\left(\sum_{k\geq 0}\frac{p_i^{2k}}{(2k)!!i^k}\right) \left(\sum_{k\geq 0}\frac{(-p_i)^{k}}{k! i^k }\right)
\end{align*}
so that the coefficient of $\prod p_i^{a_i}$ is 
\begin{align*} \prod_{i \textrm{ odd}} \left(\begin{cases} \frac{1}{a_i!! i^{a_i/2}}  & \textrm{ if $a_i$ even} \\
0 & \textrm{ if $a_i$ odd} \end{cases} \right) \prod_{i \textrm{  even}} (-1)^{a_i}\left(\sum_{2k \leq a_i} \frac{1}{(2k)!!(a_i-2k)!i^{a_i - k}} \right) \\
=\prod_i \frac{(-1)^{a_i}g_i(a_i)}{a_i! i^{a_i}}. \end{align*}
and we obtain the desired inner product giving \cref{eq.ds-trace-symplectic} exactly as in the orthogonal case treated above. 

Finally we treat the unitary case, following the same strategy. For $V$ the standard representation of $U(n)$ and with notation as in \cref{theorem.body-random-matrices}, we have 
\begin{align*} \mbb{E}\left[\prod\mr{Tr}(M^i)^{a_i}\mr{Tr}(\overline{M}^i)^{b_i}\right] &= \mbb{E}\left[ \left(\prod p_i^{a_i} \circ [V]\right) \left(\prod p_i^{b_i} \circ [V^*]\right)\right] \\&= \langle \mathbb{E}[\Exp_{\sigma}([V]h_1 + [V^*] \overline{h}_i)], \prod_i p_i^{a_i} \prod_i \overline{p}_i^{b_i}.\rangle\end{align*}
where on $(\Lambda \otimes \Lambda)^\wedge$ we have taken the obvious tensor extension of the Hall inner product and are using the analog of \cref{lemma.extract-distribution-via-hall} in this setting. 

We claim that, under the condition $\sum i a_i \leq n$ and $\sum ib_i \leq n$, \cref{theorem.body-random-matrices} implies this final inner product is equal to 
\[\langle \Exp_{\sigma}(h_1 \overline{h}_1), \prod p_i^{a_i} \prod_i \overline{p}_i^{b_i} \rangle. \]
Indeed, $\prod_i p_i^{a_i}$ lives in degree $\sum i a_i \leq n$ while $\prod_i \overline{p}_i^{b_i}$ lives in degree $\sum i b_i \leq n$, and the part of the moment generating function where each degree is $\leq n$ and $\Exp_{\sigma}(h_1 \overline{h}_1)$ agree by \cref{theorem.body-random-matrices} (again since $||\tau_i||\leq|\tau_i|$). We now compute this inner product by expanding $\Exp_{\sigma}(h_1 \overline{h}_1)$ in terms of power sum monomials. 
\begin{align*} \Exp_{\sigma}(h_1 \overline{h}_1)&=\Exp_{\sigma}(p_1 \overline{p}_1)\\
    &=\prod_{i=1}^\infty \exp\left(\frac{p_i}{i}\circ p_1\overline{p}_1\right) \\
        &=\prod_{i=1}^\infty \exp\left(\frac{p_i\overline{p}_i}{i}\right)  \\
    &=\prod_{i}\left(\sum_{k\geq 0} \frac{p_i^k \overline{p}_i^k}{i^k k!} \right). 
\end{align*}
The coefficient of $\prod p_i^{a_i} \prod \overline{p}_i^{b_i}$ is thus zero if $(a_i) \neq (b_i)$, and $\prod_i \frac{1}{i^{a_i}a_i!}$ otherwise. Thus, by \cref{eq.power-sum-inner-product}, taking the inner product with $\prod p_i^{a_i} \prod \overline{p}_i^{b_i}$ gives zero if $(a_i) \neq (b_i)$ and $\prod_i \frac{(i^{a_i}a_i!)^2}{i^{a_i}a_i!}=\prod_i i^{a_i}a_i!$ if $(a_i)=(b_i)$, recovering the right hand side of \cref{eq.ds-trace-unitary}.    
\end{proof}

\begin{remark} \label{remark.gaussians-and-convergence-rate}
In \cite{DiaconisShahshahani.OnTheEigenvaluesOfRandomMatrices}, \cref{corollary.trace-distributions} is interpreted as an asymptotic description of traces of powers as independent Gaussian random variables. The control on the moments in \cref{corollary.trace-distributions} is enough to give a bound on the rate of convergence, and Johansson \cite{Johansson.OnRandomMatricesFromTheCompactClassicalGroups} has established a much stronger estimate. Our method of proving \cref{main.random-matrix} shows that each coefficient in the moment generating function is non-negative and bounded above by the limiting value. Using the second fundamental theorem of invariant theory, one can in principle control the discrepancy exactly. It would be interesting to understand what bounds this method gives on the rate of convergence, and to compare to the method of \cite{Johansson.OnRandomMatricesFromTheCompactClassicalGroups} and related work. 
\end{remark}

\subsection{Cycle counting in the symmetric group}\label{ss.cycle-counting}
We now explain how to recover \cite[Theorem 7]{DiaconisShahshahani.OnTheEigenvaluesOfRandomMatrices} on the joint moments of cycle counting functions on symmetric groups from \cref{main.random-matrix}/\cref{theorem.body-random-matrices}-(3).

We adopt the notation of \cref{theorem.body-random-matrices}-(3). 
First, we note that, by \cref{def.binomial-and-poisson}/\cref{lemma.poisson-and-binomial-distributions}, \cref{theorem.body-random-matrices} implies that 
\[ \mbb{E}\left[(1+h_1)^{[V_n]}\right] \equiv \Exp_{\sigma}(h_1) \mod \Lambda^{\wedge}_{\deg > n}. \]
Indeed, the degree for the congruence does not change when switching between moment generating functions because $\deg h_\tau=\deg c_\tau$ (note, however, that the $c_\tau$ are not homogeneous). 
\begin{remark}
    One can also make a combinatorial argument to directly compute the falling moments to obtain 
    \[ \mbb{E}\left[(1+h_1)^{[V_n]}\right] \equiv \Exp_{\sigma}(h_1) \mod \Lambda^{\wedge}_{\deg > n}.\]
    Moreover, in the spirit of \cref{remark.falling-vs-sigma}, because $[V_n]$ is secretly a combinatorial random variable, the argument is much simpler in the falling moment case: here to obtain the coefficient of $m_\tau$, one is counting orbits in the configuration space $C^\tau(\{1,\ldots,n\})$, and there is exactly one orbit if it is non-empty. 
\end{remark}

Letting $C_k$ denote the $k$th cycle counting function, we have
\[ (1+h_1)^{[V_n]}=\prod_{i=1}^\infty \exp( C_i \log(1+ p_i)) \]
where the $\exp$ and $\log$ on the right are computed as the classical exponential and logarithm series. This identity can be verified by hand, or viewed as a special case of \cite[Lemma 3.5]{Howe.MotivicRandomVariablesAndRepresentationStabilityIConfigurationSpaces}. Then, since $\Exp_{\sigma}(h_1)=\prod_{i=1}^\infty \exp(\frac{p_i}{i})$ by \cref{lemma.exponential-power-sum}, we deduce that
\[ \mbb{E}\left[\prod_{i\geq 1} \exp(C_i \log(1+p_i))]\right] \equiv \prod_{i \geq 1}\exp\left(\frac{p_i}{i}\right) \mod \Lambda^{\wedge}_{\deg > n} \]
Note that $\exp(C_i\log(1+p_i))=(1+p_i)^{C_i}$, where the power is interpreted in the classical sense. Since the $p_i$ are a polynomial basis for $\Lambda$, this congruence says that, up to terms of degree $>n$, the joint falling moments of the cycle counting functions $C_i$ are the joint falling moments of independent Poission distributions of mean $\frac{1}{i}$, which is precisely the statement of \cite[Theorem 7]{DiaconisShahshahani.OnTheEigenvaluesOfRandomMatrices}. 

\section{Abstract point-counting}\label{s.abstract-pc}
We are interested in studying the $L$-functions of $\ell$-adic local systems and constructible sheaves on varieties over a finite field $\fq$. Although a priori defined as Euler products over closed points, these are rational functions by the results of \cite{Deligne.LaConjectureDeWeilII}, and we want to study the distribution of their zeroes, viewed as multisets of complex numbers by a fixed isomorphism $\overline{\mbb{Q}_\ell} \cong \mbb{C}$, and discuss convergence properties for the archimedean topology on $\mbb{C}$. As in \cite{BiluDasHowe.SpecialValuesOfMotivicEulerProducts}, it is thus natural to pass from $\mbb{Z}[\mbb{C}]$ to its completion $W(\mbb{C})\cong 1+t\mbb{C}[[t]]$ (for the point-counting topology of loc. cit.), in which case we do not actually need to know that the $L$-functions are rational, since already the Euler products give us elements of $1+t\mbb{C}[[t]]$. Moreover, the notion of an Euler product does not depend on the actual variety, just on its set of closed points, or equivalently the $\fqbar$-points equipped with the action of $\mbb{Z}$ where $1$ acts by the (geometric) Frobenius, i.e. the action sending coordinates to their $q$-powers. Thus we find that the structure needed to study this problem via $\lambda$-probability is very simple and can be abstracted completely away from algebraic geometry: we need to look at $\mbb{Z}$-invariant $W(\mbb{C})$-valued functions on (admissible) $\mbb{Z}$-sets, and take an expectation modeled on an Euler product. This is the theory developed here. 

In \cref{ss.wittC} we recall some coordinates on $W(\mbb{C})$. In \cref{ss.admissible-z-sets} we define the category of admissible $\mathbb{Z}$-sets and explain the relation with zeta functions and other basic properties. In particular, in \cref{prop.powers-euler-product} we give a description of certain powers (in the sense of \cref{ss.powers}) that will be used in the proofs of \cref{main.dirichlet-characters} and \cref{main.hypersurface-L-functions}. In \cref{ss.wc-valued-functions} we introduce the relevant rings of $W(\mbb{C})$-valued functions and study their functoriality properties; in particular, we explain how to integrate these over an admissible $\mbb{Z}$-set. In \cref{ss.relation-admissible-finite-integration} (see also \cref{lemma.motivic-to-finite-probability-spaces}) we explain a key relation between integration over admissible $\mathbb{Z}$-sets and integration over finite sets that we will use to reduce the proofs of \cref{main.dirichlet-characters} and \cref{main.hypersurface-L-functions} to na\"{i}ve point-counting problems that can be accessed by Poonen's sieve. In \cref{ss.abstract-pc-prob-spaces}, we define the relevant $\lambda$-probability space for an admissible $\mbb{Z}$-set (generalizing the simple case of $W(\mbb{C})$-valued functions on a finite set with the uniform measure).  

Finally, in \cref{ss.varieties-elladic-sheaves-L} we briefly explain the compatibility with the theory of constructible $\ell$-adic sheaves and their $L$-functions; let us emphasize though that most of the results and proofs of this paper can be phrased so that they are completely independent of this interpretation!

\subsection{The $\lambda$-ring $W(\mbb{C})$ and zeta functions}\label{ss.wittC}
\newcommand{\bfk}{\mathbf{k}}
Recall from \cref{ss.witt-and-lambda} that $W(\mathbb{C})$ is the ring $\Hom_{\mr{Ring}}(\Lambda, \mathbb{C})$. We recall here some more explicit interpretations in this case (see, e.g., \cite[\S3]{BiluDasEtAl.ZetaStatisticsAndHadamardFunctions} for more details and several related constructions).

We identify $W(\mathbb{C})$ as an additive group with $1+t\mathbb{C}[[t]]$ by 
\[ \alpha \in \Hom_{\mr{Ring}}(\Lambda, \mathbb{C}) \mapsto \sum_{k\geq0} \alpha(h_k)t^k \in 1+t\mathbb{C}[[t]]. \]
We identify $W(\mathbb{C})$ as a ring with $\prod_{i \geq 1} \mathbb{C}$ by 
\[ \alpha \mapsto (\alpha(p_i))_{i \geq 1}, \]
i.e. by passing to the ghost coordinates (which gives an isomorphism $W(R)=\prod_{i \geq 1} R$ whenever $R$ is a $\mbb{Q}$-algebra). To pass from the first identification to the second, one applies $t d\log$ to an element of $1+t\mathbb{C}[[t]]$ and then takes the coefficients, i.e. for $\alpha \in W(\mathbb{C})$,
\[ td\log\left(\sum_{k\geq0} \alpha(h_k)t^k\right)=\sum_{i\geq 1} \alpha(p_i) t^i. \]
This identity follows from \cref{eq.complete-powersum-relation}. 
Note that, in the identification with $\prod_{i \geq 1} \mathbb{C}$, the plethystic action of $\Lambda$ on $W(\mathbb{C})$ is determined by 
\begin{equation}\label{eq.wc-coord-plethy-power-sum} p_j \circ (a_i)_{i}=(a_{ij})_i.\end{equation}

Given $f\in \mbb{C}[[t]]$ and $k$ a positive integer, we write $f(t^k)$ for the series obtained by substituting $t^k$ for $t$. Note that, for $f \in 1+t\mbb{C}[[t]]$, 
\begin{equation}\label{eq.t-dlog} td\log(f(t^k))=k( d\log(f(t)))(t^k). \end{equation}
Given $\alpha \in W(\mathbb{C})$, by $\alpha(t^k)$ we will always mean the element obtained by viewing $\alpha$ as an element of $1+t\mathbb{C}[[t]]$ and then substituting $t^k$. 
Under the identification $W(\mathbb{C})=\prod_{i \geq 1} \mathbb{C}$, \cref{eq.t-dlog} shows this substitution acts as
\begin{equation}\label{eq.subs-tk-ghost} (a_i)_{i} \mapsto (ka_{i/k})_i, \textrm{ where $a_{i/k}:=0$ if $i/k \not \in \mathbb{Z}_{\geq 1}$}. \end{equation}
It follows that, for $[\bfk]$ the element of $W(\mathbb{C})$ corresponding to $\frac{1}{1-t^k}\in 1+t \mathbb{C}[[t]]$ and thus to $(\overbrace{0,\ldots,0,k}^{k},\overbrace{0,\ldots,0,k}^{k}, \ldots) \in \prod_{i\geq 1}\mathbb{C}$,   
\begin{equation}\label{eq.subst-lambda-op} (p_k \circ \alpha)(t^k)= [\bfk] \cdot \alpha \textrm{ and } p_k \circ (\alpha(t^k))=k\alpha.\end{equation}

\subsection{Admissible $\mathbb{Z}$-sets}\label{ss.admissible-z-sets}
\begin{definition}A $\mathbb{Z}$-set is a set with an action of $\mathbb{Z}$.
\begin{enumerate}
    \item  An $\mathbb{Z}$-set $S$ is \emph{admissible}\footnote{An admissible $\mathbb{Z}$-set is, in particular, a $\hat{\mbb{Z}}$-set, and the terminology is by analogy with the theory of admissible representations of profinite groups. In the literature, admissible $\mathbb{Z}$-sets are also called almost finite cyclic $\mathbb{Z}$-sets.} if each point $s \in S$ is stabilized by a finite index subgroup $k\mathbb{Z}$ , and $S^{k\mathbb{Z}}$ is finite for each $k\in \mathbb{Z}_{>0}$. 
    \item For $k$ a positive integer, we write $\bfk$ for the admissible $\mathbb{Z}$-set $\mathbb{Z}/k\mathbb{Z}$.
    \item For $S$ an admissible $\mathbb{Z}$-set, and $s \in S$, we write $|s|$ for the orbit of $s$ and $\deg(s)=\deg(|s|)$ for the size of the orbit.  We write $S(\bfk):=\Hom_{\mathbb{Z}\mathrm{-set}}(\bfk, S)$. 
\end{enumerate}
\end{definition}

\begin{example}\label{example.adm-z-sets}\hfill
\begin{enumerate}
\item Any finite set is an admissible $\mathbb{Z}$-set when equipped with the trivial action. This identifies finite sets with a full subcategory of admissible $\mathbb{Z}$-sets. 
\item Let $\mbb{F}_q$ be a finite field and let $X/\mathbb{F}_q$ be a variety. Then $X(\overline{\mathbb{F}}_q)$ is an admissible $\mathbb{Z}$-set,
where $1 \in \mathbb{Z}$ acts by the geometric Frobenius (which, if $X$ is affine, acts by raising coordinates to the $q$th power). We have $|X(\overline{\mathbb{F}}_q)|=|X|$, the set of closed points of $X$, and $(X(\overline{\mathbb{F}}_q))(\bfk)=X(\mathbb{F}_{q^k})$. For $x \in X(\fqbar)$, $\deg (x) = [\mbb{F}_q(x):\mbb{F}_q]$. 
\end{enumerate}
\end{example}

For $S$ an admissible $\mathbb{Z}$-set, its zeta function is defined by 
\[ Z_S(t):= \prod_{s \in |S|} \frac{1}{1-t^{\deg(s)}} \in 1+t\mbb{C}[[t]]=W(\mathbb{C}). \]
We also write $[S]=Z_S(t)$. We note that, with this definition $[\bfk]=\frac{1}{1-t^k}$, so that this notation does not conflict with the notation introduced above in \cref{ss.wittC}. 
\begin{lemma}\label{lemma.admissible-z-set-class}
    Let $S$ be an admissible $\mathbb{Z}$-set. Then 
    \[ td\log(Z_S(t))=\sum_{k \geq 1} \#S(\bfk)t^k.\]
    In particular, under our identification $W(\mbb{C})= \prod_{i \geq 1}\mbb{C}$, 
    \[[S]=(\#S(\mbf{1}),\#S(\mbf{2}), \ldots)\] 
    and we have
    \[ Z_{S_1 \sqcup S_2}(t)=Z_{S_1}(t) +_W Z_{S_2}(t)=Z_{S_1}(t)Z_{S_2}(t) \textrm{ and } Z_{S_1 \times S_2}(t)=Z_{S_1}(t) \cdot_W Z_{S_2}(t), \]
    were the subscript $W$ emphasizes addition and multiplication in $W(\mbb{C})$ vs $\mbb{C}[[t]]$. 
\end{lemma}
\begin{proof}
    The computation of $td\log$ is the usual elementary expansion as for zeta functions of varieties over finite fields. The remaining equalities follow immediately since the coefficients of $td\log$ give the identification $W(\mathbb{C})=\prod_{i \geq 1}\mathbb{C}$. 
\end{proof}

\begin{remark}
For varieties over $\mathbb{F}_q$, the identities of \cref{lemma.admissible-z-set-class} are well known; see, e.g., \cite[Theorem 2.1]{Ramachandran.ZetaFunctionsGrothendieckGroupsAndTheWittRing} and the references therein. We note also that, by \cite[Corollary 1]{DressSiebeneicher.TheBurnsideRingOfProfiniteGroupsAndTheWittVectorConstruction} the assignment $S \mapsto [S]=Z_{s}(t)$ extends to an identification $K_0(\textrm{Admissible $\mathbb{Z}$-sets})=W(\mathbb{Z})\subseteq W(\mathbb{C})$, which explains the notation. The $\lambda$-ring structure is determined on the Grothendieck ring by $h_k \circ [S]=[\Sym^k S]$. 
\end{remark}

\begin{definition}\label{def.fiber-admissible-Z-set}
    Let $\varphi: V \rightarrow S$ be a morphism of admissible $\mathbb{Z}$-sets. For $s \in S$, we define the fiber of $\varphi$ over $s$, $V_s$ to be the admissible $\mathbb{Z}$-set $\varphi^{-1}(s)$ with $\mathbb{Z}$-action obtained from the $\mathbb{Z}$-action on $V$ by precomposition with multiplication by $\deg(s)$. 
\end{definition}

\begin{example}\label{example.fiber-base-change}
    If $S$ is an admissible $\mbb{Z}$-set, then the fiber $(S \times \bfk)_1$ over $1 \in \bfk$ is naturally identified with the set $S$ but where the $\mathbb{Z}$ action is precomposed with multiplication by $k$. In particular, let $\mbb{F}_q$ be a finite field and let $X/\mbb{F}_q$ be a variety, and let $X'/\mathbb{F}_{q^k}$ be the base change $X_{\mbb{F}_{q^k}}=X \times_{\Spec \mathbb{F}_q} \mathbb{F}_{q^k}$. Then $X(\fqbar)=X'(\fqbar)$ as sets, but as admissible $\mathbb{Z}$-sets we have $X'(\fqbar)= (X(\fqbar) \times \bfk)_1$. 
\end{example}

We will need the following result computing certain powers (in the sense of \cref{ss.powers}) in this setting.
\begin{proposition}\label{prop.powers-euler-product}
 Let $V$ be an admissible $\mathbb{Z}$-set, and let $F \in W(\mathbb{C})[[\ul{t}_{\mbb{N}}]]$ with constant coefficient $1$. Then, using subscripts to denote projection to coordinates in $W(\mbb{C})=\prod_{i\geq 1}\mbb{C}$, 
    \[ (F^{[V]})_i = \prod_{|v| \in | (V\times \mathbf{i})_1|} F_{i\deg(|v|)}(\ul{t}_{\mbb{N}}^{\deg(|v|)}) \textrm{ in } \mathbb{C}[[\ul{t}_{\mbb{N}}]]. \]
    In particular, for $V/\mathbb{F}_q$ a variety, using \cref{example.fiber-base-change},
    \[ (F^{[V(\fqbar)]})_i=\prod_{|v| \in |V_{\mathbb{F}_{q^i}}|} F_{i\deg(|v|)}(\ul{t}_{\mbb{N}}^{\deg(v)}) \] 
    where, $|V_{\mathbb{F}_{q^i}}|$ denotes the set of closed points and in computing the degree of a closed point $|v| \in |V_{\mathbb{F}_{q^i}}|$, $V_{\mathbb{F}_{q^i}}$ is treated as a variety over $\mathbb{F}_{q^i}$ (rather than over $\mathbb{F}_q$).   
\end{proposition}
\begin{proof}
We first note that 
\[ F^{[V]}=F^{\sum_{|v|\in|V|} |v|}=\prod_{|v| \in |V|} F^{|v|} .\]
Thus we are reduced to treating the case of a single orbit, i.e. $V=\bf{d}$, so that 
\[ [V]=[\mbf{d}]=\frac{1}{1-t^d}=(\overbrace{0,\ldots,0,d}^{\textrm{$d$ terms}}, \overbrace{0,\ldots,0,d}^{\textrm{$d$ terms}}, \ldots) \]

Now, let $\mu=\frac{\lcm(i,d)}{i}= \frac{d}{\gcd(i,d)}$. In the following we will write $p_i \ast F$ for the plethystic action of power sum functions on $F$ acting only on the coefficients and not the variables; we note that this commutes with the usual plethystic action $p_i \circ$ used in \cref{ss.symmetric-power-series}. Expanding via \cref{lemma.exponential-power-sum} and using that the plethysm $p_j \circ$ is a $\lambda$-ring homomorphism, we find 
\begin{align*} (\Exp_{\sigma}([V] \Log_{\sigma}(F)))_i &=  \prod_{j \geq 1} \exp\left( \frac{p_j \circ [V]\Log_{\sigma}(F)}{j}\right)_{i}\\ 
&=  \prod_{j \geq 1} \exp\left( \frac{(p_j \circ [V])\Log_{\sigma}(p_j \circ F)}{j}\right)_{i}\\ 
&=\prod_{j \geq 1} \exp\left( \frac{(p_{ij} \circ [V])\Log_{\sigma}(p_{j} \circ p_i \ast F))}{j}\right)_1 \\ 
&=\prod_{k \geq 1} \exp\left( \frac{ \Log_{\sigma}(p_{k\mu} \circ p_i \ast F)}{k\mu/d}\right)_1 \\
&=\prod_{k \geq 1} \exp\left( \frac{d}{\mu} \cdot \frac{ p_{k}\circ \Log_{\sigma}(p_{\mu} \circ p_i \ast F)}{k}\right)_1\\
&=(\Exp_{\sigma}(\Log_{\sigma} (p_{\mu} \circ p_i \ast F))^{d/\mu})_1\\
&=(p_{\mu} \circ p_i \ast F)_1^{d/\mu}  \\
&= (F_{\mu i}(\ul{t}_{\mathbb{N}}^\mu))^{d/\mu}.
\end{align*}

Since $(\mbf{d} \times \mbf{i})_{1}$ has $d/\mu=\gcd(i,d)$ orbits, each consisting of points of degree $\mu$ over $\mathbb{F}_{q^i}$, we conclude. 
\end{proof}

\subsection{$W(\mathbb{C})$-valued functions}\label{ss.wc-valued-functions}

\begin{definition}
For any admissible $\mathbb{Z}$-set $S$, we write $C(S, W(\mathbb{C}))$ for the $\lambda$-ring of $\mathbb{Z}$-equivariant functions from $S$ to $W(\mathbb{C})$, where the action of $\mathbb{Z}$ on $W(\mathbb{C})$ is trivial and the $\lambda$-ring structure is pointwise. 
\end{definition}

Note that, by the $\mathbb{Z}$-equivariance, any $f \in C(S,W(\mathbb{C}))$ is constant on each orbit of $\mathbb{Z}$ acting on $S$, so this ring is naturally identified with $C(|S|, W(\mathbb{C}))$. 

\begin{definition}\label{def.admissible-z-set-integration}
Given an admissible $\mathbb{Z}$-set $V$ and $f \in C(V, W(\mathbb{C}))$, we define 
\[ \int_V f= \sum_{|v|\in |V|} f(t^{\deg(|v|)}) \in W(\mbb{C})\]
where we emphasize that the addition law on the right is that of $W(\mbb{C})$. In particular,  if we identify $W(\mathbb{C})=1+t \mathbb{C}[[t]]$, then this becomes the Euler product
\[  \prod_{|v| \in |V|} f(t^{\deg(|v|)}), \]
so that the infinite sum in $W(\mathbb{C})$ makes sense since, in this product, there only finitely many terms contributing to each power of $t$. 
\end{definition}

\begin{lemma}\label{lemma.integration-components}
If we identify $W(\mathbb{C})=\prod_{i \geq 1} \mathbb{C}$, then, 
\[\left(\int_V f\right)_i = \sum_{|v| \in |V| \textrm{ s.t. } \deg(|v|) | i} \deg(|v|)f_{i/\deg(|v|)}(|v|) = \sum_{v \in V \textrm{ s.t. } \deg(v) | i} f_{i/\deg(v)}(v).  \]
\end{lemma}
\begin{proof}
    Immediate using \cref{eq.subs-tk-ghost}. 
\end{proof}

\begin{definition}
Given a map $\varphi: V \rightarrow S$ of admissible $\mbb{Z}$-sets, we define
\begin{enumerate}
\item pull back of functions $\varphi^*: C(S, W(\mathbb{C})) \rightarrow C(V, W(\mathbb{C}))$ by 
\[ \varphi^*f (v):= p_{\frac{\deg(|v|)}{\deg(|\varphi(v)|)}} \circ f(\varphi(v)) \textrm{ for $v \in V$} \]
\item integration over fibers $\varphi_!: C(V, W(\mathbb{C})) \rightarrow C(S,W(\mathbb{C}))$ by
\[ (\varphi_!f)(s):= \int_{V_s} f|_{V_s} \textrm{ for $s \in S$}, \]
where the fiber $V_s$ is as defined in \cref{def.fiber-admissible-Z-set}. Equivalently, viewing $f$ and $\varphi_!f$ as functions on orbits,
\[ (\varphi_!f)(|s|) =\sum_{|v| \subseteq \varphi^{-1}(|s|)}(f(|v|))(t^{\frac{\deg(|v|)}{\deg(|s|)}}). \]
\end{enumerate}
\end{definition}

\begin{example}\label{example.zeta-funciton-integral-of-constant}
    For the constant function $1$ in $C(V, W(\mathbb{C}))$,
    \[ \int_{V} 1=\sum_{|v|} 1(t^{\deg(|v|)})= \prod_{|v|} \frac{1}{1-t^{\deg(|v|)}}=Z_V(t)=[V], \]
    where in the second equality we use the identification of  $W(\mathbb{C})$ with $1+t\mathbb{C}[[t]]$. 
  
    More generally, if $\varphi: V \rightarrow S$ is a map of admissible $\mbb{Z}$-sets. Then we define $[V/S]$ to be the function sending $s$ to $[V_s]$ as in \cref{def.fiber-admissible-Z-set}, so that $[V/S]=\varphi_! 1$. 
\end{example}

Note that, since the power sum plethysms are $\lambda$-ring endomorphism on any $\lambda$-ring (see \cref{ss.witt-and-lambda}), $\varphi^*$ is a $\lambda$-ring morphism. 

\begin{example}\label{example.final-morphism}
    Let $\varphi: V \rightarrow \{\ast\}$ be the final map. Then, $\varphi^*: W(\mathbb{C}) \rightarrow C(V,W(\mathbb{C}))$ sends $w \in W(\mathbb{C})$ to the function $f_w$ defined by
 \begin{equation}\label{eq.witt-product-embedding} f_w(v) =  p_{\deg(v)}\circ w \textrm{ for } v\in V.\end{equation}
 In particular, the $W(\mathbb{C})$-algebra structure given by $\varphi^*$  agrees with the $W(\mathbb{C})$-algebra structure coming from literal constant $W(\mathbb{C})$-valued functions only when $V$ is a finite set viewed as an admissible $\mbb{Z}$-set under the embedding of \cref{example.adm-z-sets}-(1) (i.e. when $V$ contains only points of degree 1). Note however that, for elements of $\mbb{C}$ embedded diagonally in $W(\mbb{C})=\prod_{i \geq 1} \mbb{C}$, the pullback structure and the constant structure always agree (using \cref{eq.wc-coord-plethy-power-sum}). 
\end{example}

{\noindent\bf Warning}:
We always equip $C(V, W(\mathbb{C}))$ with the $W(\mathbb{C})$-$\lambda$-algebra structure obtained by pullback from the final map $V\rightarrow \{\ast\}$ as in \cref{example.final-morphism}, and never the algebra structure from literal constant functions. In particular, if $w \in W(\mathbb{C})$, if we write $\int_V w$ this means to treat $w$ as pulled back from the point, i.e. to integrate $f_w$ from \cref{example.final-morphism}.\\

For $\varphi: V \rightarrow S$ a map of admissible $\mathbb{Z}$-sets, the integration over fibers map $\varphi_!$ is additive but does not typically preserve products or plethysms. However, it does preserve the $W(\mathbb{C})$-module structure, and even more:

\begin{lemma}[Projection formula]\label{lemma.projection-formula}
Given a map $\varphi: V \rightarrow S$ of admissible $\mbb{Z}$-sets, we view $C(V,W(\mathbb{C}))$ as a $C(S,W(\mathbb{C}))$-module by $\varphi^*$. Then, $\varphi_!$ is a morphism of $C(S,W(\mathbb{C}))$-modules. In particular, $\varphi_!$ is always a morphism of $W(\mathbb{C})$-modules. 
\end{lemma}
\begin{proof}
The last ``in particular" follows from the first claim after observing that, for any composable morphisms $\varphi_1$ and $\varphi_2$, $(\varphi_2 \circ \varphi_1)^*f = \varphi_1^*\varphi_2^*f$. 

For the first claim, by taking fibers one immediately reduces to the case $S=\{\ast\}$ so that $\varphi_! = \int_V$. We for $w \in W(\mbb{C})$, we write $f_w$ for its pullback to $V$ as in \cref{example.final-morphism}. To verify $\int_V$ is $W(\mathbb{C})$-linear, we note that for $w = (a_1, a_2, \ldots) \in W(\mathbb{C})$ and $v \in |V|$ of degree $d$, $f_w(v)$ is identified, via $td\log$, with $(a_d, a_{2d}, \ldots)$. It follows that, for $v$ of degree $d$ and $g \in C(V,W(\mbb{C}))$, $f_w g (v)=(a_d g_1(v), a_{2d} g_2(v), \ldots)$. Applying \cref{lemma.integration-components}, we find
\begin{align*} \left(\int_{V} f_w g\right)_k &= \sum_{v \textrm{ s.t. } \deg(v)|k} (\deg(v) g_{\frac{k}{\deg(v)}}(v)a_{k})\\
&= a_k \sum_{v \textrm{ s.t. } \deg(v)|k} (\deg(v) g_{\frac{k}{\deg(v)}}(v) ) \\
&= a_k \left(\int_{V} g\right)_k.
\end{align*}
Since this holds for all $k$, the map is $W(\mbb{C})$-linear. 
\end{proof}

\subsection{A relation between integration over finite sets and integration over admissible $\mathbb{Z}$-sets}\label{ss.relation-admissible-finite-integration}
\newcommand{\rest}{\mathrm{res}}
For $S$ an admissible $\mathbb{Z}$-set and $k \geq 1$, there is a natural map of admissible $\mbb{Z}$-sets
\[ c_k: S(\bfk) \times \bfk \rightarrow S, (s, i) \mapsto (s(i)) \]
where here the finite set $S(\bfk)$ is viewed as an admissible $\mathbb{Z}$-set with the trivial action as in \cref{example.adm-z-sets}-(1). By pullback we obtain a map of $W(\mathbb{C})$-$\lambda$-algebras.  
\[ c_k^*: C(S, W(\mathbb{C})) \rightarrow C( S(\bfk) \times \bfk, W(\mathbb{C})) \]
We can identify the right-hand side with $W(\mathbb{C})$-valued functions on $|S(\bfk)\times \bfk|=S(\bfk)$ in the usual way, or equivalently by restricting to $S(\bfk)\times 1$. We thus obtain a map of $W(\mathbb{C})$-$\lambda$-algebras 
\[ \rest_k: C(S, W(\mathbb{C})) \rightarrow C(S(\bfk), W(\mathbb{C}))^{(k)} \]
where the superscript on the right indicates that we take the $W(\mathbb{C})$-algebra structure via $p_k \circ: W(\mathbb{C}) \rightarrow W(\mathbb{C})$. Explicitly,
\[ \rest_k(f)(s: \bfk \rightarrow S) =p_{\frac{k}{\deg(s(|\bfk|))}} \circ f(s(|\bfk|)). \]

\begin{proposition}\label{prop.integration-finite-set-relation}
    For $S$ an admissible $\mathbb{Z}$-set, $k\geq 1$, and $f \in C(S, W(\mathbb{C}))$, 
    \[ \left(\int_S f\right)_k = \left(\int_{S(\bf{k})} \rest_k(f)\right)_1. \]
    where the subscripts denote components under $W(\mbb{C})=\prod_{i \geq 1} \mbb{C}$. 
\end{proposition}
\begin{proof}
For $f=(f_1, f_2, \ldots)$ in $C(S,W(\mathbb{C}))$, we apply \cref{lemma.integration-components} to see
    \begin{align*} \left(\int_S f\right)_k &= \sum_{s \in S(\bfk)} f_{\frac{k}{\deg(s(\bfk))}}(s(\bfk)) \\
    &=\sum_{s \in S(\bfk)} (p_{\frac{k}{\deg(s(\bfk))}}\circ f(s(\bfk)))_1 \\
    &= \sum_{s \in S(\bfk)} (\rest_k(f) (s))_1. 
    \end{align*}    
\end{proof}

\subsection{Probability spaces}\label{ss.abstract-pc-prob-spaces}
Let $S$ be an admissible $\mathbb{Z}$-set. Then, it follows from \cref{lemma.admissible-z-set-class} that 
\[ [S] =Z_S(t) \in 1+t\mbb{C}[[t]]]=W(\mathbb{C})=\prod_{i\geq 1} \mathbb{C} \]
is invertible if and only if $S(\bf{1}) \neq \emptyset$ (since this is equivalent to  $S({\bfk})\neq \emptyset$ for all $k \geq 1$). In this case, we can define 
\[ \mathbb{E}: C(S, W(\mathbb{C})) \rightarrow W(\mathbb{C}), E[f]=\frac{\int_S f}{[S]}. \]
By \cref{example.zeta-funciton-integral-of-constant}, this is an expectation (i.e. it sends $1$ to $1$) and is $W(\mathbb{C})$-linear.  Note that, for $k \geq 1$, the projection to the $k$th component $\mathbb{E}_k$ is given by $\frac{(\int_S f)_k}{\#S(\bfk)}$, so that we can define $\mathbb{E}_{k}$ as long as $S(\bfk)\neq\emptyset$. 

Using the restriction maps of \cref{prop.integration-finite-set-relation}, we can relate the projections $\mathbb{E}_k$ to probability on the finite set of order $k$ points of $S$:
\begin{lemma}\label{lemma.motivic-to-finite-probability-spaces}
Let $S$ be an admissible $\mathbb{Z}$-set and $k \geq 1$. If $S(\bfk) \neq \emptyset$, then the map $\rest_k$ of \cref{prop.integration-finite-set-relation} induces a map of $\lambda$-probability spaces 
\[ (C(S, W(\mathbb{C})), \mathbb{E}_k) \rightarrow (C(S(\bfk), W(\mathbb{C}))^{(k)}, \mathbb{E}_1). \] 
\end{lemma}
\begin{proof}
    Immediate by \cref{prop.integration-finite-set-relation}.
\end{proof}

\subsection{Varieties, $\ell$-adic sheaves, and $L$-functions}\label{ss.varieties-elladic-sheaves-L}
\newcommand{\Cons}{\mathrm{Cons}}
Let $\fq$ be a finite field of order $q$, fix an algebraic closure $\fqbar/\fq$, and let $S/\fq$ be a variety. For $\ell$ a prime coprime to $q$ and $\overline{\mathbb{Q}}_\ell$ an algebraic closure of $\mathbb{Q}_\ell$, we write $K_0(\Cons(S,\overline{\mbb{Q}}_\ell))$ for the Grothendieck ring of constructible $\overline{\mathbb{Q}}_\ell$ sheaves on $S$ (as in, e.g., \cite[\S6]{BhattScholze.TheProEtaleTopologyForSchemes}). By definition, any element of this ring can be written as a formal sum of the extensions by zero of local systems on locally closed subsets, and this sum is unique up to passing to a common constructible refinement. There is thus a unique $\lambda$-ring structure on $K_0(\Cons(S,\overline{\mbb{Q}}_\ell))$ such that, for $\mc{V}$ a $\overline{\mbb{Q}}_\ell$  local system on a locally closed $Z \subseteq S$, $h_k \circ [j_!\mc{V}] = [j_! \Sym^k \mc{V}]$ for $j: Z \hookrightarrow S$ the inclusion. For $f: V \rightarrow S$ a morphism of varieties over $\fq$, there is a natural map 
\[ f^*: K_0(\Cons(S,\overline{\mbb{Q}}_\ell)) \rightarrow K_0(\Cons(V,\overline{\mbb{Q}}_\ell)) \]
induced by pullback of constructible sheaves and an integration over fibers map
\[ f_!: K_0(\Cons(V,\overline{\mbb{Q}}_\ell)) \rightarrow K_0(\Cons(S,\overline{\mbb{Q}}_\ell)) \]
induced by (derived) proper pushforward of complexes. The map $f^*$ is a map of $\lambda$-rings, and by the projection formula, $f_!$ is $K_0(\Cons(V,\overline{\mbb{Q}}_\ell))$-linear.

We fix an isomorphism $\overline{\mbb{Q}}_\ell \rightarrow \mbb{C}$. Then, we obtain a map 
\[ \iota_S: K_0(\Cons(S,\overline{\mbb{Q}}_\ell))) \rightarrow C(S(\fqbar), W(\mbb{C}))\]
by sending $[\mc{V}]$ for $\mc{V}$ a constructible sheaf to 
\[ f_{\mc{V}}: s \in S(\fqbar) \mapsto L_{\mc{V}_s}(t) \in 1+t\mbb{C}[[t]]]=W(\mbb{C})\]
where $L_{\mc{V}_s}(t)$ is the characteristic power series of the (geometric) Frobenius element for $\mbb{F}_q(s)$ acting on the finite dimensional $\overline{\mbb{Q}}_\ell$ vector space $\mc{V}_s$, viewed as a $\mbb{C}$-vector space by the fixed isomorphism. Viewed as a sequence in $\prod_{i \geq 1} \mbb{C}$, $f_{\mc{V}}(s)$ is given by the sequences of traces of powers of the $q^{\deg(s)}$-power (geometric) Frobenius on $\mc{V}_s$. 

\begin{proposition}
    The map $\iota_S$ is an injective map of $\lambda$-rings, and, for each map of varieties $f: V \rightarrow S$, $\iota_V \circ f^*=(f|_{V(\fqbar)})^*\iota_S$ and $\iota_S \circ f_!=(f|_{V(\fqbar)})_! \circ \iota_V$.   
\end{proposition}
\begin{proof}
    To see that it is a map of $\lambda$-rings, we can note that any $f_{\mc{V}}$ as above takes values in $\mbb{Z}[\mbb{C}] \subseteq W(\mbb{C})$ and compare the symmetric power plethysms. That it is injective follows from \cite[Th\'{e}or\'{e}me (1.1.2)]{Laumon.TransformationDeFourierConstantsDEquationsFonctionellesEtConjectureDeWeil}. The commutativity for $f^*$ is trivial, and the commutativity for $f_!$ follows, in the usual way, from the Grothendieck-Lefschetz fixed point formula.  
\end{proof}

\begin{example}
Let $f: V \rightarrow S$ be a map of varieties. Then, for $[V(\fqbar)/S(\fqbar)]$ as in \cref{example.zeta-funciton-integral-of-constant}, 
\[ [V(\fqbar)/S(\fqbar)] = \iota_S\left( \sum (-1)^i [R^if_!\overline{\mathbb{Q}}_\ell]\right). \]
\end{example}

\section{An application of Poonen's sieve}\label{s.application-of-poonen}
In this section, we apply \cite[Theorem 1.3]{Poonen.BertiniTheoremsOverFiniteFields} to compute the $\lambda$-distributions of some natural $W(\mbb{C})$-valued probability problems on finite sets of homogeneous polynomials over a finite field $\mbb{F}_q$. The main result is \cref{prop.point-counting-asymptotics}. This ``na\"{i}ve" point-counting result will be used in conjunction with \cref{lemma.motivic-to-finite-probability-spaces} in the next sections to prove our motivic point-counting results, \cref{main.dirichlet-characters} and \cref{main.hypersurface-L-functions}.

\subsection{Setup}\label{ss.poonen-sieve-setup}
We fix an integer $n$ and let $S_d(\fq)$ denote the set of degree $d$ homogeneous polynomials in $n+1$ variables $x_1, \ldots, x_{n+1}$ with coefficients in  $\mathbb{F}_q$. For every closed point $P \in |\mathbb{P}^n_{\mbb{F}_q}|$, we fix a nonnegative integer $M_P$ and a nonvanishing coordinate $x_{j_P}$. Given $F \in S_d(\fq)$, we write $F_P$ for the image of $F/x_{j_P}^d$ in 
$\mathcal{O}_{\mathbb{P}^n_{\mathbb{F}_q},P}/\mf{m}_P^{M_P+1}.$ Given $A_P \subseteq \mathcal{O}_{\mathbb{P}^n_{\mathbb{F}_q,P}}/\mf{m}_P^{M_P+1}$,
we write
\[ \mu_P(A_P):=\frac{\# A_P}{\#\left(\mathcal{O}_{\mathbb{P}^n_{\mathbb{F}_q},P}/\mf{m}_P^{M_P+1}\right)}. \]

We fix, for each $P$, a non-empty subset $A_P \subseteq \mathcal{O}_{\mathbb{P}^n_{\mathbb{F}_q},P}/\mf{m}_P^{M_P+1}$ of ``allowable" Taylor expansions. We suppose furthermore that these have been chosen so that there exist smooth quasi-projective subvarieties $X_1, \ldots, X_u \subseteq \mathbb{P}^n$ of dimensions $m_i=\dim X_i$, such that, for all but finitely many $P$, $A_P$ contains $F_P$ for any $F \in S_d(\fq)$ such that $V(F) \cap X_i$ is smooth of dimension $m_i-1$ at $P$ for all $i$. We let $U_d(\mathbb{F}_q)$ denote the subset of $S_d(\fq)$ consisting of all $F \in S_d(\fq)$ such that $F \in A_P$ for all $P \in |\mathbb{P}^n_{\mathbb{F}_q}|$. By \cite[Theorem 1.3 and subsequent remark]{Poonen.BertiniTheoremsOverFiniteFields},
\[ \lim_{d \rightarrow \infty} \frac{ \# U_d(\mathbb{F}_q) }{\# S_d}=\prod_{P \in |\mathbb{P}^n_{\mbb{F}_q}|} \mu_P(A_P) \neq 0. \]
In particular, for $d \gg 0$, $U_d(\mathbb{F}_q) \neq 0$, and we can consider the $\lambda$-probability space
\[ ( C(U_d(\fq), W(\mbb{C})), \mbb{E}_1) \textrm{ for } \mbb{E}_1(X)=\frac{1}{\#U_d(\fq)}\sum_{u \in U_d(\fq)} X(u)_1, \]
where the subscript $1$ denotes projection to the first coordinate in $W(\mbb{C})=\prod_{i \geq 1} \mbb{C}$. 

\begin{remark}
    We could take the $W(\mbb{C})$-valued expectation $\mbb{E}$ everywhere below instead of its first component, but it simplifies the proof to use $\mbb{E}_1$ and this will suffice for computing the motivic $\lambda$-distributions via \cref{prop.integration-finite-set-relation}. 
\end{remark}

\subsection{Random variables}
Continuing with the above notation, for each $P$, we fix a function $X_P: A_P \rightarrow W(\mathbb{C}) = 1 + t \mathbb{C}[[t]]$. We may view $X_P$ as a $W(\mathbb{C})$-valued random variable on the finite set $A_P$, equipped with the uniform measure. It induces a random variable $X_{P,d}$ on $U_d(\mathbb{F}_q)$ sending $F$ to $X_P(F_P)$. 

\begin{proposition}\label{prop.point-counting-asymptotics}
    With notation as above, suppose that, for any $M > 0$, for all but finitely many closed points $P \in |\mathbb{P}^n_{\mbb{F}_q}|$, $X_P$ factors through $1+t^M\mathbb{C}[[t]]$. Then, for any $d$,
    $\sum_{P \in |\mathbb{P}^n_{\mathbb{F}_q}|} X_{P,d}$ converges to a function $X_d: U_d(\mathbb{F}_q) \rightarrow W(\mathbb{C})$, where we note that the sum is point-wise in $W(\mathbb{C})$, i.e. product of power series in $1+t \mathbb{C}[[t]]$, and 
    \[ \lim_{d \rightarrow \infty} \mathbb{E}_1[\Exp_{\sigma}(X_{d}h_1)]=\prod_{P \in |\mathbb{P}^n_{\mathbb{F}_q}|} \mathbb{E}_1[\Exp_{\sigma}(X_Ph_1)]. \]

    More generally, suppose for each point $P$ given $X_{i,P}$, $i = 1, \ldots, m$ satisfying this property, and define $X_{i,P,d}$, $X_{i,d}$ similarly. Then 
     \begin{multline*} \lim_{d \rightarrow \infty} \mathbb{E}_1[\Exp_{\sigma}(X_{1,d}h_1(t_{1,1},t_{1,2}+\ldots ) + X_{2,d}h_1(t_{2,1},t_{2,2}+\ldots )+ \ldots)]= \\ \prod_{P \in |\mathbb{P}^n_{\mathbb{F}_q}|} \mathbb{E}_1[\Exp_{\sigma}(X_{1,P}h_1(t_{1,1},t_{1,2}+\ldots ) + X_{2,P}h_1(t_{2,1},t_{2,2}+\ldots )+ \ldots)]. \end{multline*} 
\end{proposition}
\begin{proof}
We treat only the case of the $\sigma$-moment generating function when given a single random variable at each point; the argument for the joint $\sigma$-moment generating function given $m$ random variables is essentially the same. 

 Let $S_m$ be the finite collection of points such that $X_P$ factors through $1+t^{m+1}\mathbb{C}[[t]]$ for all $P \not\in S_m$. We claim that, for all $f \in \Lambda$  with $\deg f \leq m$ and any $d>0$, we have an identity of $\mathbb{C}$-valued functions on $U_d(\mathbb{F}_q)$
\[ (f \circ X_d)_1=  (f \circ \sum_{P \in S_m} X_{d,P})_1.\]

To establish this claim, it suffices to show that, for any power sum symmetric function $p_i$ with $i \leq m$ and any $P\not\in S_m$, $(p_i \circ X_P)_1=0$ (since this implies that, for all $d$, $(p_i \circ X_{P,d})_1=0$). If we write $X_P=(X_{P,1}, X_{P,2}, \ldots)$, then for $P \not\in S_m$, $X_{P,k}=0$ for all $k \leq m$ since $td\log (1+t^{m+1} \mathbb{C}[[t]]) \subseteq t^{m+1}\mathbb{C}[[t]]$. Then, as $(p_i \circ X_P)_1=X_{P,i}$, we find that $(p_i \circ X_P)_1=0$ for $P \not\in S_m$. 

With this claim established, we obtain from \cref{eq.sigma-moment-expansion} that
\[  \mathbb{E}_1 [\Exp_{\sigma}(X_{d}h_1)] \equiv  \mathbb{E}_1\left[\Exp_{\sigma}\left( \left(\sum_{P \in S_m}X_{d,P}\right)h_1\right)\right]  \mod  \Lambda_{\mbb{C}, \deg > m}^\wedge. \]

Similarly, we find that
\[ \prod_{P \in |\mathbb{P}^n_{\mathbb{F}_q}|} \mathbb{E}_1[\Exp_{\sigma}(X_Ph_1)] \equiv  \prod_{P \in S_m} \mathbb{E}_1[\Exp_{\sigma}(X_Ph_1)]  \mod \Lambda_{\mbb{C}, \deg > m}^\wedge. \]

Thus, by \cref{lemma.asympt-ind-moment-gen}, it suffices to show that for any $P$, for the expectation $\mathbb{E}_{1}$, the $\Lambda$-distribution of $X_{P,d}$ limits as $d \rightarrow \infty$ to the $\Lambda$-distribution of $X_{P}$, and, for any $m$, the sequence of collections of random variables $(X_{P,d})_{P \in S_m}$ are asymptotically independent as $d \rightarrow \infty$. 

To that end, we note that, by \cite[Theorem 1.3]{Poonen.BertiniTheoremsOverFiniteFields}, for any choice of subsets $B_{P} \subseteq A_P$ for each $P \in S_m$,
\[ \lim_{d \rightarrow \infty} \mbb{P}(\substack{F_P \in B_{P} \textrm{ for all $P \in S_m$ and }\\ F_P\in A_P \textrm{ for all $P \not\in S_m$}})=\prod_{P \in S_m} \mu_P(B_P) \prod_{P \not\in S_m} \mu_P(A_P). \]
It follows that for any collection $\mathbb{C}$-valued functions $g_P$, $P\in S_m$, the classical random variables $F \mapsto g_P(F_P)$ on $U_d(\mathbb{F}_q)$ with the uniform measure are asymptotically independent, and the distribution of each approaches the distribution of $g_P$ treated as a classical random variable on $A_P$ with the uniform measure. Applying this to $g_P$ of the form $(f_P \circ X_P)_1$ for $f_P \in \Lambda$, we obtain the result. 
\end{proof}

\section{$L$-functions of characters}\label{s.L-characters}

In this section we prove \cref{theorem.motivic-characters}, which computes the limiting distributions in \cref{main.dirichlet-characters}, and discuss some complements --- in particular, in \cref{example.stable-homology-quadratic}, we compare the case of quadratic characters to some results in \cite{BergstromDiaconuPetersenWesterland.HyperellipticCurvesTheScanningMapAndMomentsOfFamiliesOfQuadraticLFunctions} (see also \cref{remark.quadratic-case-stable-traces}). The proof of \cref{main.dirichlet-characters} will be completed in \cref{s.comparison} by comparing the limit $\sigma$-moment generating functions to those in the associated random matrix statistics (\cref{prop.statistics-comparison-characters}). 

To prove \cref{theorem.motivic-characters}, we use \cref{lemma.motivic-to-finite-probability-spaces} to reduce to a problem over a finite probability space, which we analyze using \cref{prop.point-counting-asymptotics}. 

\subsection{Setup and a note about a normalization}\label{ss.characters-setup}
We fix a finite field $\mbb{F}_q$, a prime $\ell | q-1$, and a non-trivial character $\chi:\mu_\ell(\mbb{F}_q) \rightarrow \mbb{C}^\times$. For any $\ell$-power free monic $f \in \mbb{F}_q[x]$, we write $\chi_f$ for the character $\Gal(\overline{\fq(x)}/\fq) \rightarrow \mbb{C}^\times$ obtained by composing $\chi$ with the map $\Gal(\overline{\fq(x)}/\fq) \rightarrow \mu_{\ell}(\mbb{F}_q)=\mu_\ell(\mbb{F}_q(x))$ determined from $f$ by Kummer theory. In other words, $\chi_f(\sigma)=\chi\left(\sigma(f^{1/\ell})/f^{1/\ell}\right).$  

We write $L^\infty(\chi_f,s)$ for the associated $L$-function with the Euler factor at $\infty$ removed. If $\ell \nmid \deg f$, then $\chi_f$ is ramified at $\infty$ and $L^\infty(\chi_f,s)=L(\chi_f,s)$. In general, if we set $\chi(0)=0$, then we have an Euler product 
\[ L^\infty(\chi_f, s)= \mc{L}(\chi_f, t)|_{t=q^{-s}} \textrm{ for } \mc{L}(\chi_f, t)=\prod_{|z| \in |\mbb{A}^1_{\mbb{F}_q}|} \frac{1}{1-t^{\deg(z)}\chi\left(f(|z|)^{\frac{q^{\deg(|z|)}-1}{\ell}}\right)}. \]

The function $\mc{L}(\chi_f, t)$ is a polynomial with constant term $1$ whose degree depends only on the degree of $f$ and, when $\ell \nmid \deg f$, its roots have absolute value $q^{-\frac{1}{2}}$ so that the roots of $\mc{L}(\chi_f, t q^{-\frac{1}{2}})$ all have absolute value $1$ (this can be deduced from the Weil conjectures for the curve with affine model $y^\ell=f(x)$, but we do not actually need this fact for anything below except motivation). 

We would like to study the distribution of symmetric functions of these roots using $\lambda$-probability. To that end, let us highlight an issue regarding normalizations: if we write the multiset of roots of $\mc{L}(\chi_f, tq^{-\frac{1}{2}})$ as $\zeta_1,\ldots, \zeta_m$, then the class 
\[ [\zeta_1^{-1}] + \ldots + [\zeta_m^{-1}] = [\overline{\zeta}_1] + \ldots + [\overline{\zeta}_m] \]
in $\mbb{Z}[\mbb{C}]$ is identified, by our embedding into $W(\mbb{C})=1+t\mbb{C}[[t]]$, with $\mc{L}(\chi_f, tq^{-\frac{1}{2}})^{-1}$. This is the element of $W(\mbb{C})$ that will appear below. The statement of \cref{main.dirichlet-characters} is symmetric in passing to the conjugate roots, so this does not effect the final result. 

\subsection{The point-counting result}

Continuing with the notation of the previous subsection, we let $U_{d,\ell}(\fq)$ be the set of degree $d$ monic $\ell$-power free polynomials in $\mbb{F}_q[x]$. From the discussion above, it follows that the random variable on $U_{d,\ell}(\fq)$ sending $f$ to the multiset of zeroes of $\mc{L}(\chi_f, tq^{-\frac{1}{2}})$, viewed as an element of $Z[\mbb{C}]$ is identified, after embedding $Z[\mbb{C}]$ in to $W(\mbb{C})$, with the random variable $X_{\chi,q,d}$ on $U_{d,\ell}(\fq)$ defined by
\begin{multline*} X_{\chi,q,d}(f)=\mc{L}(\chi_f, tq^{-\frac{1}{2}})^{-1}=\prod_{|z| \in |\mbb{A}^1_{\mbb{F}_q}|} \left( 1-(tq^{-\frac{1}{2}})^{\deg(|z|)}\chi\left(f(|z|)^{\frac{q^{\deg(|z|)}-1}{\ell}}\right) \right)\\ \in 1+t\mbb{C}[[t]]=W(\mbb{C}), \end{multline*}
where, for $|z|$ a closed point of $\mbb{A}^1_{\mbb{F}_q}$, $f(|z|)$ denotes the image of the $f$ in the residue field $\fq(|z|)$. 
In particular, for any $g \in \Lambda$, $\mbb{E}_1[g \circ X_{\chi,q,d}]$ is the average value on $U_{d,\ell}(\mbb{F}_q)$ of the symmetric function $g$ evaluated on the roots of $\mc{L}(\chi_f, tq^{-\frac{1}{2}})$. 

\begin{remark}
Let us emphasize this random variable can be defined without knowing that $\mc{L}(\chi_f,t)$ is a polynomial or the absolute value of its roots, since the Euler product makes sense immediately as an element of $W(\mbb{C})$! 
\end{remark}

Since the infinite product of power series is a sum in $W(\mbb{C})$, 
\[ X_{\chi,q,d}=\sum_{|z| \in |\mbb{A}^1_{\mbb{F}_q}|} X_{\chi,q,d,|z|} \textrm{ where } 
 X_{\chi,q,d,|z|}(f)= 1-(tq^{-\frac{1}{2}})^{\deg(|z|)}\chi\left(f(|z|)^{\frac{q^{\deg(|z|)}-1}{\ell}}\right). \]
We will compute its asymptotic $\sigma$-moment generating function for the expectation $\mbb{E}_1$ on $C(U_{d,\ell}(\fq),W(\mbb{C}))$ using \cref{prop.point-counting-asymptotics}. 

\begin{theorem}\label{theorem.point-counting-characters}Let \[ c_\ell=\frac{[q^{\ell-1}]}{[q]^{\ell-1}+[q]^{\ell-2} + \ldots + 1} \in W(\mbb{C}) \textrm{ so that } (c_\ell)_i=\frac{q^{i(\ell-1)}}{q^{i(\ell-1)}+q^{i(\ell-2)} + \ldots + 1} \in \mbb{C}.\]
For $\chi$ of order $\ell=2$, the asymptotic $\sigma$-moment generating function of $X_{\chi,q,d}$ is
\begin{align*}
    \lim_{d\rightarrow \infty} \mbb{E}_1[\Exp_{\sigma}(X_{\chi,q, d}h_1)] & = \left(\left(1 + c_2 \sum_{k\geq 1} [q^{-k}] e_{2 k}\right)^{[q]}\right)_1 \\
    &= \prod_{P \in |\mbb{A}^1_{\mbb{F}_q}|} \left(1+(c_{2})_{\deg(P)}\sum_{k\geq 1}q^{-k\deg(P)} (p_{\deg(P)}\circ e_{2k})\right).
\end{align*}
where the subscript denotes projection to the first component in $W(\mbb{C})$ (note that the power on the right-hand side lies, by definition, in $\Lambda_{W(\mbb{C})}^\wedge)$). 

Similarly, for $\chi$ of order $\ell > 2$ prime, the asymptotic joint $\sigma$-moment generating function of $X_{\chi,q,d}$ and its complex conjugate $\overline{X}_{\chi, q,d}$ is 
\begin{multline*}
\lim_{d\rightarrow \infty} \mbb{E}_1[\Exp_{\sigma}(X_{\chi,q,d}h_1+\overline{X}_{\chi,q,d}\overline{h}_1)] = \\\left(\left(1 + c_\ell \sum_{\substack{k_1 \geq 1\textrm{ or } k_2\geq1 \\ k_1 \equiv k_2 \mod \ell}} (-1)^{\ell k}[q^{-\frac{k_1+k_2}{2}}] e_{k_1}\overline{e}_{k_2}\right)^{[q]}\right)_1  =\\
\prod_{P \in |\mbb{A}^1_{\fq}|}\left(1 + (c_\ell)_{\deg(P)} \sum_{\substack{k_1 \geq 1\textrm{ or } k_2\geq1 \\ k_1 \equiv k_2 \mod \ell}} (-1)^{k_1+k_2}q^{-\frac{k_1+k_2}{2}\deg(P)} (p_{\deg(P)}\circ (e_{k_1}\overline{e}_{k_2}))\right).
\end{multline*}
\end{theorem}
\begin{proof}
We first note that the first components of the powers appearing in the statement can be expanded as Euler products as claimed by invoking the $i=1$ case of \cref{prop.powers-euler-product} (to see this, note that, for any $z\in \mbb{C}$, $[z]=(z,z^2,z^3,\ldots)$ as an element of $W(\mbb{C})=\prod_{i\geq}\mbb{C}$, and that $p_{\deg P} \circ f=f(t_1^{\deg P}, t_2^{\deg P}, \ldots)$ for $f \in \Lambda$). 

We now explain how to put ourselves into the setup of \cref{prop.point-counting-asymptotics}. In the notation of \cref{s.application-of-poonen}, we work with $n=1$, i.e. on $\mathbb{P}^1$ with variables $[x_0:x_1]$. For $\infty=[0:1]$, we set $M_\infty=0$, $j_\infty=0$, and $A_\infty=1 \in \mbb{F}_q=\mc{O}_{\mbb{P}^1_{\fq},\infty}/\mf{m}_\infty$. For $P \in |\mbb{A}^1_{\fq}|=|\mbb{P}^1_{\fq}|\backslash \infty$, we set $M_P=\ell-1$, $j_\infty=1$, and $A_P$ to be the complement of $0 \in \mc{O}_{\mbb{P}^1_{\fq},P}/\mf{m}_P^\ell$. 

The condition at $\infty$ says that, for $F \in S_d(\fqbar)$ and $x:=\frac{x_0}{x_1}$, $F_\infty \in A_\infty$ if and only if $F/x_1^d$ is monic of degree $d$ in $\mbb{F}_q[x]$. Thus, the set $U_d(\fq)$ of \cref{s.application-of-poonen} can be viewed, by $F \mapsto F/x_1^d$, as a subset of the degree $d$ monic polynomials. Then, the conditions at the other primes says exactly that $U_d(\fq)$ consists of those monic polynomials $f$ that are $\ell$th-power free, i.e. $U_d(\fq)=U_{d,\ell}(\fq).$ 

Now, we define our random variable $X_\infty$ on $A_\infty$ to be the constant random variable with value $1\in 1+t\mbb{C}[[t]]=W(\mbb{C})$ (i.e. the zero element of $W(\mbb{C})$), and, for any $P\in |\mbb{A}^1_{\mbb{F}_q}|$, we define a random variable $X_P$ on $A_P$ as follows: first, to simplify notation, we write $q_P=q^{\deg(P)}$ and, for $g$ the germ of a function at $P$, we write $\chi_P(g):=\chi\left(g(P)^{\frac{q_P-1}{\ell}}\right)$. Then, we set
\[ X_P(g) = 1- (q_P^{-\frac{1}{2}}\chi_P(g))t^{\deg(P)}  \in 1+t\mbb{C}[[t]] = W(\mbb{C}). \]

The random variables $X_{\chi,q,d,P}$ we have defined above are the associated random variables $X_{P,d}$ of \cref{s.application-of-poonen}, thus by \cref{prop.point-counting-asymptotics}, 
\begin{align*} \textrm{for $\ell=2$:} \lim_{d\rightarrow \infty} \mbb{E}_1[\Exp_{\sigma}(X_{\chi,q,d}h_1)] &= \prod_{P \in |\mbb{P}^1_{\mbb{F}_q}|} \mbb{E}_1[\Exp_{\sigma}(X_Ph_1)] \textrm{ and,} \\
\textrm{for $\ell >2$:}
\lim_{d\rightarrow \infty} \mbb{E}_1[\Exp_{\sigma}(X_{\chi,q,d}h_1 + \overline{X}_{\chi,q,d}\overline{h}_1)] &= \prod_{P \in |\mbb{P}^1_{\mbb{F}_q}|} \mbb{E}_1[\Exp_{\sigma}(X_Ph_1 + \overline{X}_P\overline{h}_1)]. \end{align*}
The factors at $P=\infty$ on the right are $1$. We must compute them for $P \in |\mbb{A}^1_{\mbb{F}_q}|.$

We first observe that 
\begin{equation}\label{eq.character-proof-local-moments} \left(\Exp_{\sigma}(X_P h_1)\right)_1=1+ \sum_{k \geq 1} (-q_P^{-\frac{1}{2}}\chi_P)^k p_{\deg(P)}\circ e_k\end{equation}

Indeed, when we write an element $\alpha \in  W(\mbb{C}) = \Hom(\Lambda, \mbb{C})$ as a power series in $1+t\mbb{C}[[t]]$, then the coefficient of $t^k$ is $\alpha(h_k)$. On the other hand, when we pass to the first component, we are evaluating on $p_1$, and thus 
\[ (h_k \circ \alpha)_1=(h_k \circ \alpha)(p_1)=\alpha(p_1 \circ h_k)=\alpha(h_k). \]
Since in our power series only the coefficient of $t^{\deg(P)}$ is non-zero, in the expansion of \cref{eq.exp-symmetric-monomial-expansion},
\[ \Exp_\sigma(X_Ph_1)=\sum_{\tau} (h_{\tau}\circ X_P) m_\tau, \]
only the terms corresponding to $h_\tau$ of the form $h_{\deg(P)}^k$ are non-zero. These are the terms corresponding to the monomial symmetric functions \[ m_{\deg(P)^k}=p_{\deg(P)} \circ m_{1^k}=p_{\deg(P)}  \circ e_k \]
where $i^k$ denotes the partition $(\overbrace{i,\ldots,i}^k)$. Moreover, the coefficient of such a term is exactly the coefficient of $t^{\deg(P)}$ to the $k$th power, giving \cref{eq.character-proof-local-moments}. 

We now specialize to the case $\ell=2$. Then, 
\[ \mbb{E}_1[\Exp_{\sigma}(X_P h_1)]=\mbb{E}[ \left(\Exp_{\sigma}(X_P h_1)\right)_1] \]
where the right-hand side is classical expectation (coefficient-wise in the power-series) for the uniform probability measure on the finite set $A_P$. Now, for each $\zeta=\pm 1$, $\chi_P\left(g\right)$ is equal to $\zeta$ with probability 
\[ \frac{1}{2} \frac{(q_P-1)q_P}{q_P^{2}-1} = \frac{1}{2}(c_{2})_{\deg(P)}, \]
and the rest of the time it is equal to $0$. The expectation of the coefficient $(-q_P^{-\frac{1}{2}}\chi_P)^k$ of $p_{\deg(P)}\circ e_k$, $k\geq 1$ in \cref{eq.character-proof-local-moments} is thus $0$ if $k$ is odd and $q_P^{-\frac{k}{2}}$ if $k$ is even. So,
\[ \mbb{E}_1[\Exp_{\sigma}(X_Ph_1)] = 1+ (c_{2})_{\deg(P)}\sum_{k\geq 1}q_P^{-k} p_{\deg(P)}\circ e_{2k}. \] 
Plugging this back into the product and recalling $q_P=q^{\deg(P)}$, we find
\[ \lim_{d\rightarrow \infty} \mbb{E}_1[\Exp_{\sigma}(X_{d,q}h_1)] = \prod_{P \in |\mbb{A}^1_{\mbb{F}_q}|} \left(1+ (c_{2})_{\deg(P)} \sum_{k\geq 1} q^{-k\deg(P)} p_{\deg(P)}\circ e_{2k}\right).\]
This concludes the proof for $\ell=2$. 

We now treat the case $\ell >2$, arguing similarly. Then 
\begin{align*} \Exp_{\sigma}(X_Ph_1 + \overline{X}_P\overline{h}_1) &= \Exp_{\sigma}(X_P h_1)\Exp_{\sigma}(\overline{X}_P\overline{h}_1)\\
&= 1+ \sum_{k_1 \geq 1\textrm{ or } k_2\geq1} (-q_P^{-1/2})^{k_1+k_2} \chi_P^{k_1-k_2} (p_{\deg(P)}\circ e_{k_1}\overline{e}_{k_2}). \end{align*}
Note that, for each $\ell$th root of unity $\zeta$, $\chi_P(g)$ is equal to $\zeta$ with probability
\[ \frac{1}{\ell} \frac{(q_P-1)q_P^{(\ell-1)}}{q_P^{\ell}-1} = \frac{1}{2}(c_{\ell})_{\deg(P)}, \]
and the rest of the time it is equal to $0$. The expectation of the coefficient $(-q_P^{-1/2})^{k_1+k_2} \chi_P^{k_1-k_2}$ of $(p_{\deg(P)}\circ e_{k_1}\overline{e}_{k_2})$ is thus $0$ if $k_1 \not\equiv k_2 \mod \ell$ (since the sum of all $\ell$th roots of unity is zero), and $(c_{\ell})_{\deg(P)}(-q_P^{-1/2})^{k_1+k_2}$ if $k_1 \equiv k_2 \mod \ell$. So, 
\[ \mbb{E}_1[\Exp_{\sigma}(X_Ph_1)] = 1+ (c_{\ell})_{\deg(P)}\sum_{\substack{k_1 \geq 1\textrm{ or } k_2\geq1 \\ k_1 \equiv k_2 \mod \ell}} (-1)^{k_1+k_2}q_P^{-\frac{k_1+k_2}{2}} (p_{\deg(P)}\circ e_{k_1}\overline{e}_{k_2}). \] 
Plugging this back into the Euler product, we conclude. 
\end{proof}

\begin{example}\label{example.stable-homology-quadratic}
Consider, in \cref{theorem.point-counting-characters}, the case $\ell=2$ for the quadratic character $\chi$. Then, for $d$ odd, the random variable $[q^{1/2}]X_{\chi,d,q}$ for $X_{\chi, d,q}$ as in \cref{theorem.point-counting-characters} is the random variable associated to $\mc{V}_d$, the local system of first cohomology of the universal hyperelliptic curve $y^2=f(x)$ over $C_d$, the space of monic square free polynomials. We note that, because the Schur polynomials $s_\tau$ are an orthonormal basis for the Hall inner product, \cref{lemma.extract-distribution-via-hall} implies that, for any random variable $X$, its $\sigma$-moment generating function can be written as
\[ \mbb{E}_1[\Exp_{\sigma}(Xh_1)]] = \sum_{\tau} \mbb{E}_1[s_\tau\circ X]s_\tau.\]

We note that $s_\tau \circ [\mc{V}_d] = [S^\tau (\mc{V}_d)]$ for $S^\tau$ the Schur functor associated to $\tau$. Applying this expression of the moment generating function to $[\mc{V}_d]$ and using \cref{theorem.point-counting-characters}, we find 
\[ \sum_{\tau} \lim_{d \rightarrow \infty}\left(\mbb{E}_1[S^\tau(\mc{V}_d)]\right)s_\tau = \left(1 + \frac{[q]}{[q]+1} \sum_{k\geq 1} e_{2k}\right)^{[q]}  \]
where $S^\tau(\bullet)$ is the Schur functor associated to $\tau$. By the Grothendieck-Lefschetz formula, the left-hand side expands as\footnote{We are projecting everywhere from $W(\mbb{C})$ to its first component in $\mbb{C}$ below, but we drop the subscripts $1$ since the formulas are already complicated. Using our motivic results in the next section, one finds the computation holds already in $W(\mbb{C})$!}
\[ \sum_{\tau} \lim_{d \rightarrow \infty}\left(\mbb{E}_1[S^\tau(\mc{V}_d)]\right)s_\tau =\sum_{\tau} \lim_{d \rightarrow \infty}\left(\frac{1}{[C_d]}\sum_{i=0}^\infty (-1)^i [H^i_c(C_d, S^\tau(\mc{V}_d))]\right)s_\tau. \]
Using that $\lim_{d \rightarrow \infty} [C_d]/[q^d]=1-[q^{-1}]$, and noting that, by Poincar\'{e} duality,
\[ [H_i(C_d,S^{\tau}(\mc{V}_d))]=\frac{[H^i_c(C_d,S^\tau(\mc{V}_d))]}{[q^d]},\]
we find that 
\begin{equation}\label{eq.quad-char-cohom-generating-function-1}\sum_{\tau} \lim_{d \rightarrow \infty}\left(\sum_{i=0}^\infty (-1)^i [H_i(C_d, S^\tau(\mc{V}_d))]\right)s_\tau = (1-[q^{-1}])\left(1 + \frac{[q]}{[q]+1} \sum_{k\geq 1} e_{2k}\right)^{[q]}.  \end{equation}
We claim this  formula is equivalent to a computation of the stable traces of Frobenius for Schur functors applied to $\mc{V}_d$ given in \cite[Theorem 11.2.3]{BergstromDiaconuPetersenWesterland.HyperellipticCurvesTheScanningMapAndMomentsOfFamiliesOfQuadraticLFunctions}\footnote{Note that to form the exponential in the formula in loc. cit., $q^{-1}$ and $q$ should be viewed as elements of $\Lambda((q^{-1}))$ with $h_k(q^{\pm{1}})=q^{\pm k}.$}:
\begin{equation}\label{eq.quad-char-cohom-generating-function-2} \sum_{\tau} \lim_{d \rightarrow \infty}\left(\sum_{i=0}^\infty (-1)^i [H_i(C_d, S^\tau(\mc{V}_d))]\right)s_{\tau'} = \Exp_{\sigma}\left([q]\Log_{\sigma}\left([q^{-1}]+\sum_{k\geq0}h_{2k}\right)-1\right).\end{equation}

Here $\tau'$ denotes the partition conjugate to $\tau$. Indeed, \cref{eq.quad-char-cohom-generating-function-2} can be obtained from \cref{eq.quad-char-cohom-generating-function-1} by applying the standard involution $\omega$ on symmetric functions of \cref{ss.involution-omega} (which extends naturally to $\Lambda_{W(\mathbb{C})}^\wedge)$. This is immediate on the left-hand side since $\omega(s_\tau)=s_{\tau'}$. We now explain how to make this comparison on the right-hand side. We begin by rewriting the right-hand side of \cref{eq.quad-char-cohom-generating-function-1} by substituting $(1-[q^{-1}])=\Exp_{\sigma}(-[q^{-1}])$ and expressing the $[q]$th-power in terms of the $\sigma$-exponential:
\[ (1-[q^{-1}])\left(1 + \frac{[q]}{[q]+1} \sum_{k\geq 1} e_{2k}\right)^{[q]} = \Exp_{\sigma}\left([q] \Log_\sigma\left(1 + \frac{[q]}{[q]+1} \sum_{k\geq 1} e_{2k}\right) -[q^{-1}]\right). \]
We then factor out $\frac{[q]}{[q]+1}$ inside the logarithm to obtain
\[\Exp_{\sigma}\left([q] \Log_\sigma\left(\frac{[q]+1}{[q]} +  \sum_{k\geq 1} e_{2k}\right)+[q]\Log_{\sigma}\left(\frac{[q]}{[q]+1}\right)-[q^{-1}]\right) \]
We now simplify the fraction in the first logarithm term, and make the substitution $\Log_{\sigma}\left(\frac{[q]}{[q]+1]}\right)=-[q^{-1}]+[q^{-2}]$, which follows from the computation 
\[ \Exp_{\sigma}(-[q^{-1}]+[q^{-2}])=\frac{\Exp_{\sigma}(-[q^{-1}])}{\Exp_{\sigma}(-[q^{-2}])}=\frac{1-[q^{-1}]}{1-[q^{-2}]}=\frac{1}{1+[q^{-1}]}=\frac{[q]}{[q]+1}, \]
to obtain
\[\Exp_{\sigma}\left([q] \Log_\sigma\left([q^{-1}] +  \sum_{k\geq 0} e_{2k}\right)-1\right). \]
Now, applying $\omega$, we obtain the right-hand side of \cref{eq.quad-char-cohom-generating-function-2}: because all of the symmetric functions appearing are of \emph{even} degree, we can apply \cref{lemma.lambda-involution-even} to commute $\omega$ with the $\sigma$-exponential and $\sigma$-logarithm, and the comparison then follows because $\omega(e_i)=h_i$. 
\end{example}

\subsection{The motivic result}
We now upgrade \cref{theorem.point-counting-characters} to a motivic version that encodes the result simultaneously for all finite extensions of $\mbb{F}_q$ and has the advantage that it is a statement at the level of $W(\mbb{C})$ rather than a statement after projection to the first component. 

We fix a finite field $\fq$, an algebraic closure $\fqbar$, a prime $\ell | q-1$, and a non-trivial character $\chi:\mu_\ell(\mbb{F}_q) \rightarrow \mbb{C}^\times$. For $d>0$, let $U_{d,\ell}(\fqbar)$ be the admissible $\mbb{Z}$-set of degree $d$ monic $\ell$th-power free polynomials in $\fqbar[x]$. It admits a degree $1$ point, so we have the $\lambda$-probability space of \cref{ss.abstract-pc-prob-spaces},
\[ (C(U_{d,\ell}(\fqbar), W(\mbb{C})), \mbb{E} ) \]

For $z \in \fqbar$, we write 
\[ \mc{X}_{\chi,d} \in C(U_{d,\ell}(\fqbar) \times \mbb{A}^1(\fqbar), W(\mbb{C})) \] 
for the function sending $(f,z)$ to 
\[ \left( 1-(tq^{-\frac{1}{2}})^{[\mbb{F}_q(z,f):\mbb{F}_q(f)]}\chi_f\left(f(z)^{\frac{q^{[\mbb{F}_q(z,f):\mbb{F}_q]-1}}{\ell}}\right) \right) \in 1+t\mbb{C}[[t]]=W(\mbb{C}) \]
We view $\mc{X}_{\chi,d}$ as a family of random variables on $U_{d,\ell}(\fqbar)$ parameterized by $\mbb{A}^1(\fqbar)$. The random variable we are interested is then obtained by integrating this family over $\mbb{A}^1(\fqbar)$: 
\[ X_{\chi,d}:={\pi_1}_! \mc{X}_{\chi,d} \in C(U_{d,\ell}(\fqbar), W(\mbb{C})). \]
It follows from the definition of ${\pi_1}_!$ and the discussion of \cref{ss.characters-setup} that $X_{\chi,d}$ sends $f\in U_{d,\ell}(\fqbar)$ to $\mc{L}(\chi_f, tq^{-[\fq(f):\fq]/2})^{-1} \in 1+t\mbb{C}[[t]]=W(\mbb{C})$, where $\chi_f$ is a representation of $\Gal(\overline{\fq(f)(x)}/\fq(f)(x))$ for $\fq(f)$ the field generated by the coefficients of $f$. As in \cref{ss.characters-setup}, under the embedding $\mbb{Z}[\mbb{C}]\hookrightarrow W(\mbb{C})$, this is identified with the multiset of zeroes of $\mc{L}(\chi_f, tq^{-[\fq(f):\fq]/2})$, so that this is precisely the random variable appearing in \cref{main.dirichlet-characters}.

Recall the restriction functors $\rest_i$ of \cref{ss.relation-admissible-finite-integration}. 
\begin{lemma}\label{lemma.restriction-random-variable-character}
    For each $i \geq 1$, $\rest_i(X_{\chi,d})$ is the random variable $X_{\chi,d,q^{i}}$ on 
    $U_{d,\ell}(\mbb{F}_{q^i})$ of \cref{theorem.point-counting-characters} (where in \cref{theorem.point-counting-characters} we work over $\mbb{F}_{q^i}$). 
\end{lemma}
\begin{proof}
    This follows from the definition of $\rest_i$, and the observation that, if we restrict $\chi_f$ to a character of $\mbb{F}_{q^i}(x)$ rather than $\mbb{F}_q(f)(x)$, then the $L$-function changes by an application of $p_{\frac{i}{\deg(x)}} \circ$ ---  this is a very general statement, but if one prefers to deduce it directly from the Euler product, it suffices to apply \cref{eq.subst-lambda-op} to see that the local factors at the closed points of $\mbb{A}^1_{\mbb{F}_{q^i}}$ above a closed point $|z|$ of $\mbb{A}^1_{\mbb{F}_q(f)}$ give a contribution equal to $p_{\frac{i}{\deg(x)}} \circ$ applied to the local factor at $|z|$. 
\end{proof}

Combining this with \cref{theorem.point-counting-characters}, we obtain
\begin{theorem}\label{theorem.motivic-characters}
Let \[ c_\ell=\frac{[q^{\ell-1}]}{[q]^{\ell-1}+[q]^{\ell-2} + \ldots + 1} \in W(\mbb{C}).\]
For $\chi$ of order $\ell=2$, the asymptotic $\sigma$-moment generating function of $X_{\chi,d}$ is
\begin{align*}
    \lim_{d\rightarrow \infty} \mbb{E}[\Exp_{\sigma}(X_{\chi,d}h_1)] = \left(1 + c_2 \sum_{k\geq 1} [q^{-k}] e_{2 k}\right)^{[q]}. \\
\end{align*}
For $\chi$ of order $\ell > 2$ prime, the asymptotic joint $\sigma$-moment generating function of $X_{\ell,d}$ and its complex conjugate $\overline{X}_{\ell, d}$ is
\begin{multline*}
\lim_{d\rightarrow \infty} \mbb{E}[\Exp_{\sigma}(X_{\ell,d}h_1+\overline{X}_{\ell,d}\overline{h}_1)] = \\\left(1 + c_\ell \sum_{\substack{k_1 \geq 1\textrm{ or } k_2\geq1 \\ k_1 \equiv k_2 \mod \ell}} (-1)^{k_1+k_2}[q^{-\frac{k_1+k_2}{2}}] e_{k_1}\overline{e}_{k_2}\right)^{[q]}.
\end{multline*}
\end{theorem}
\begin{proof}
    For the case $\ell=2$, combining \cref{lemma.motivic-to-finite-probability-spaces} and \cref{lemma.restriction-random-variable-character}, we find the $i$th component of the left-hand side is 
    \[ \lim_{d \rightarrow \infty} \mbb{E}_1[\Exp_{\sigma}(X_{2,d,q^i}h_1)].\]
    By \cref{theorem.point-counting-characters}, this is 
    \[ \prod_{P \in |\mbb{A}^1_{\mbb{F}_{q^i}}|} \left(1+  (c_2)_{i\deg(P)} \sum_{k\geq 1}(q^i)^{-k{\deg(P)}} (p_{\deg(P)}\circ e_{2k}) \right). \]
    By \cref{prop.powers-euler-product}, this is the $i$th component of 
    \[ \left(1 + c_2 \sum_{k\geq 1} [q^{-k}] e_{2 k}\right)^{[q]} \]
   This concludes the case $\ell=2$. 
    For $\ell>2$, the argument is almost identical. 
\end{proof}

\begin{remark}
    In the statement \cref{theorem.point-counting-characters}, the moment generating function was also expanded as an Euler product over the closed points of $\mbb{A}^1_{\mbb{F}_q}$. As explained in \cite{BertucciHowe.EquidistributionAndArithmeticLambdaDistributions}, the powers appearing in \cref{theorem.motivic-characters} can also be expanded as \emph{motivic} Euler products over $\mbb{A}^1_{\mbb{F}_q}$.
\end{remark}

\section{$L$-functions for the vanishing cohomology of smooth hypersurface sections}\label{s.L-vanishing}

In this section we compute the limiting $\Lambda$-distribution in \cref{main.hypersurface-L-functions} (here as \cref{theorem.vanishing-cohom-L-functions}). The remainder of \cref{main.hypersurface-L-functions}, which gives a comparison between this limiting $\Lambda$-distribution and the random matrix statistics of \cref{main.random-matrix} / \cref{theorem.body-random-matrices}, is treated in \cref{s.comparison} as \cref{prop.statistics-comparison-hypersurface}. 

To compute the limit in question, it would be possible to argue directly with the random variable sending a smooth hypersurface section to the reciprocal roots of the $L$-function of vanishing cohomology as in \cref{s.L-characters}. However, the argument is much clearer if we instead first prove a result about the geometric random variable which counts points on smooth hypersurface sections, then subtract off the constant part of cohomology to obtain the random variable we are interested in. We prove the point counting version for this geometric random variable in \cref{theorem.geometric-hypersurface-point-counting} and then the motivic version for this geometric random variable in \cref{theorem.geometric-hypersurface-motivic}.  Note that part of the point-counting version can also be deduced from \cite[Theorem C]{Howe.MotivicRandomVariablesAndRepresentationStabilityIIHypersurfaceSections} (see \cref{remark.proof-via-mrvii}), and that \cref{main.hypersurface-L-functions}/\cref{theorem.vanishing-cohom-L-functions} implies, by projection to the first component, a much more precise version of \cite[Theorem B]{Howe.MotivicRandomVariablesAndRepresentationStabilityIIHypersurfaceSections} (see \cref{remark.vanishing-cohomology-relation-to-other-work}). 

\subsection{Setup}
We fix a finite field $\mbb{F}_q$ and an algebraic closure $\fqbar/\fq$. 

We fix a prime $\ell$ coprime to $q$ in order to use $\ell$-adic \'{e}tale cohomology, and for any variety $V/\mbb{F}_q$, we write $H^i(V)$ as an abbreviation for $H^i_\et(V_{\fqbar}, \mbb{Q}_\ell)$, which is equipped with an action of $\mbb{Z}$ where the generator $1$ acts by the (geometric) $q$-power Frobenius. We write $[H^i(V)]$ for the associated element of $W(\mbb{C})$, which, viewed as an element of $1+t\mbb{C}[[t]]$, is the characteristic power series of Frobenius, or, viewed as an element of $\prod_{i\geq 1} \mbb{C}$, is the sequence of traces of powers of Frobenius. If $V$ is smooth and projective of dimension $n$, $H^i(V)$ is zero for $i<0$ or $i> 2n$, and, fixing an embedding $\mbb{Q}_\ell \hookrightarrow \mbb{C}$, by the Weil conjectures the eigenvalues of Frobenius on $H^i(V)$ have absolute value $q^{i/2}$; by the Grothendieck-Lefschetz formula, we have 
\begin{equation}\label{eq.groth-lefschetz-decomp} [V]=\sum (-1)^i [H^i(V)] \end{equation}
where we recall that we write $[V]=[V(\fqbar)]$ for the element of $W(\mbb{C})$ that, viewed as an element of $1+t\mbb{C}[[t]]$, is the zeta function $Z_V(t)$, and, viewed as element of $\prod_{i \geq 1}\mbb{C}$, is the sequence $(\#V(\mbb{F}_q), \#V(\mbb{F}_{q^2}), \ldots).$ 

Let $Y \subseteq \mbb{P}^m_{\mbb{F}_q}$ be an $n+1$-dimensional smooth 
closed irreducible subvariety. For $F$ a degree $d$ homogeneous polynomial in $m+1$ variables with coefficients in $\mbb{F}_q$, we write $V(F)$ for the vanishing locus of $F$, a closed subvariety of $\mbb{P}^m_{\mbb{F}_q}$. If $V(F)$ intersects $Y$ transversely (in the scheme-theoretic sense; equivalently, the first order Taylor expansion of $F|_Y$ at any closed point of $Y$ is non-zero), then $Z=V(F) \cap Y \subseteq Y$ is a smooth projective subvariety of dimension $n$. By the weak Lefschetz theorem \cite[4.1.6]{Deligne.LaConjectureDeWeilII}, we have $H^i(Z)=H^i(Y)$ for $i < n$ and $H^{n}(Y) \hookrightarrow H^n(Z)$. The vanishing cohomology $\mc{V}_Z \subseteq H^n(Z)$ is a complement to this inclusion (see \cite[\S4]{Howe.MotivicRandomVariablesAndRepresentationStabilityIIHypersurfaceSections} for more details and further references); in particular, we have $[H^n(Z)]=[H^n(Y)]+[\mc{V}_Z]$ (note that for our purposes below we could also simply define the class $[\mc{V}_Z]$ to be $[H^n(Z)]-[H^n(Y)]$ without knowing anything about vanishing cohomology). 
The remaining degrees of cohomology of $Z$ are also determined by Poincar\'{e} duality, or, even more simply, the hard Lefschetz theorem \cite[Th\'{e}or\`{e}me 4.1.1]{Deligne.LaConjectureDeWeilII}:
\[ \textrm{ for $i<n$, } [H^{2n-i}(Z)]=[H^{i}(Z)][q^{n-i}]. \]
Combining these computations with \cref{eq.groth-lefschetz-decomp}, we find that 
\begin{equation}\label{eq.hypersurface-section-cohomology} [Z]=(-1)^n[\mc{V}_Z] + (-1)^n[H^n(Y)] + \sum_{i=0}^{n-1} (-1)^i (1+[q^{n-i}]) [H^i(Y)]. \end{equation}

The reciprocal roots of the normalized $L$-function for vanishing cohomology $\mc{L}_Z(t q^{-n/2})$ as appearing in \cref{main.hypersurface-L-functions}, viewed as a multiset of complex numbers and then, via $\mbb{Z}[\mbb{C}] \hookrightarrow W(\mbb{C})$, an element of $W(\mbb{C})$, is given by $[q^{-n/2}][\mc{V}_Z]$. We can solve for this quantity in \cref{eq.hypersurface-section-cohomology} to obtain 
\begin{equation}\label{eq.vanishing-cohom-computation} [q^{-n/2}][\mc{V}_Z] = [q^{-n/2}]\left ((-1)^{n} [Z] - [H^n(Y)] - (-1)^n\sum_{i=0}^{n-1} (-1)^i (1+[q^{n-i}]) [H^i(Y)]\right). \end{equation}

In particular, if we vary $Z$ and treat $[\mc{V}_Z]$ as a random variable, then the only term on the right that is not constant is $[Z]$. Since any random variable is independent to a constant random variable, we can compute the $\sigma$ moment generating function of $[q^{-n/2}][\mc{V}_Z]$ as soon as we can compute the $\sigma$-moment generating function of $(-1)^n[q^{-n/2}][Z]$, or, since scaling by a power of $[q]$ has a simple effect, just $(-1)^n [Z]$. The difference in parity here accounts for the differences between the odd and even cases in \cref{main.hypersurface-L-functions}/\cref{theorem.vanishing-cohom-L-functions}. The strategy is thus to first compute the distribution of $(-1)^n[Z]$ for all $d$, then apply \cref{eq.vanishing-cohom-computation} to deduce the distribution of $(-1)^n[q^{-n/2}][\mc{V}_Z]$. As in \cref{s.L-characters}, to compute the motivic distribution of $(-1)^n[Z]$ we first make a computation on finite probability spaces using \cref{prop.point-counting-asymptotics} and then apply \cref{lemma.motivic-to-finite-probability-spaces} and \cref{prop.powers-euler-product} to obtain the motivic distribution. 

\begin{remark}\footnote{The question raised in this remark has now been resolved in  \cite{Howe.TheNegativeSigmaMomentGeneratingFunction}.}\label{remark.sigma-moment-minus}
For a random variable $X$, we do not have a nice way to extract the $\sigma$-moment generating function for $-X$ from the $\sigma$-moment generating function for $X$. Indeed, by \cref{lemma.exponential-minus}, to obtain the $\sigma$-moment generating function for $-X$, one needs to know the value of the $\Lambda$-distribution on the elementary symmetric functions $e_\tau$, however, by \cref{lemma.extract-distribution-via-hall}, this amounts to computing the inner products with the $e_\tau$, and it is not clear to us how to do this in a useful way in general. Even for a binomial random variable, we do not see how to do it without using a stronger notion of independence as discussed in \cref{remark.independence-and-summing} that allows to make the interpretation as a sum of independent Bernoulli random variables parameterized by an admissible $\mbb{Z}$-set literal. Thus below we make the computations for both $[Z]$ and $-[Z]$ at the same time; fortunately neither is more difficult. 
\end{remark}

\subsection{The geometric random variable: point-counting}
We fix an integer $m$ and a smooth closed subvariety $Y \subseteq \mbb{P}^m_{\mbb{F}_q}$. We let $U_d(\mbb{F}_q)$ be the set degree $d$ homogeneous equations $F$ in $m+1$ variables such that $V(F)$ intersects $Y$ transversely. By \cite[Theorem 1.1]{Poonen.BertiniTheoremsOverFiniteFields}, $U_d(\mbb{F}_q)\neq \emptyset$ for $d \gg 0$, so we can consider the $\lambda$-probability space
\[ (C(U_d(\mbb{F}_q), W(\mbb{C})), \mbb{E}_1). \]
We write $Z_{q,d}$ for the random variable on $U_d(\mbb{F}_q)$ that sends $F$ to $[V(F) \cap Y] \in W(\mbb{C})$, i.e. to $Z_{V(F) \cap Y}(t) \in 1+t\mbb{C}[[t]]=W(\mbb{C}).$ 

\begin{theorem}\label{theorem.geometric-hypersurface-point-counting}
    With notation as above and denoting with a subscript $1$ the projection from $W(\mbb{C})=\prod_{i \geq 1} \mbb{C}$ to the first component, 
    \begin{align*} \lim_{d\rightarrow \infty}\mbb{E}_1[ \Exp_{\sigma}(Z_{q,d}h_1) ]&= \left(\left(1+\frac{[q]^{n+1}-1}{[q]^{n+2}-1}(h_1 + h_2 + h_3 + ...)\right)^{[Y]}\right)_1 \\ 
    &=\prod_{P \in |Y|} \left( 1+ \frac{q^{(n+1)\deg(P)}-1}{q^{(n+2)\deg(P)}-1}p_{\deg(P)}\circ(h_1+h_2+h_3+ \ldots) \right) \end{align*}
    and 
    \begin{align*} \lim_{d\rightarrow \infty}\mbb{E}_1[ \Exp_{\sigma}(-Z_{q,d}h_1)] &= \left(\left(1+\frac{[q]^{n+1}-1}{[q]^{n+2}-1}(-e_1 + e_2 -e_3 + ...)\right)^{[Y]} \right)_1 \\
        &=\prod_{P \in |Y|} \left( 1+ \frac{q^{(n+1)\deg(P)}-1}{q^{(n+2)\deg(P)}-1}p_{\deg(P)}\circ(-e_1+e_2-e_3+ \ldots) \right) 
    \end{align*}
\end{theorem}
\begin{proof}
The identity of the first components of the powers and the claimed Euler products follows from \cref{prop.powers-euler-product}. To compute the expectations, we will use \cref{prop.point-counting-asymptotics}. 

We explain how to put ourselves into the situation of \cref{prop.point-counting-asymptotics}. In the notation of \cref{s.application-of-poonen}, for each $P \in |\mbb{P}^m_{\mbb{F}_q}|$, we take $M_P=1$ so that we are looking at first-order Taylor expansions; the choice of $j_P$ to fix a trivializing coordinate $x_{j_P}$ does not effect the definition of anything below so we ignore it. For $P \not\in |Y|$ we take $A_P=\mc{O}_{\mbb{P}^m_{\mbb{F}_q},P}/\mf{m}_P^2$ to be all first order expansions and for $P \in |Y|$ we take $A_P$ to be the pre-image in $\mc{O}_{\mbb{P}^m_{\mbb{F}_q},P}/\mf{m}_P^2$ of the complement of $0$ in $\mc{O}_{Y,P}/\mf{m}_P^2$. For $P \not\in|Y|$ we take $X_P$ the be the constant random variable with value $0 \in W(\mbb{C})$ (i.e. $1 \in 1+t\mbb{C}[[t]]$) and for $P \in |Y|$ we define, 
\[ X_P(g) = \begin{cases} \frac{1}{1-t^{\deg(P)}} \in 1+t\mbb{C}[[t]]=W(\mbb{C}) & \textrm{ if $g(P)=0$, i.e. $g \in \mf{m}_P/\mf{m}_P^2$} \\
1 \in 1+t\mbb{C}[[t]]=W(\mbb{C}) & \textrm{otherwise}. \end{cases}\]

Then, for $X_{P,d}$ the associated random variables as in \cref{s.application-of-poonen}, we have
\[ Z_{q,d}=\sum_{P \in |\mbb{P}^n_{\mbb{F}_q}|} X_{P,d} = \sum_{P \in |Y|} X_{P,d}. \]
Indeed: the sum is in $W(\mbb{C})$ so is multiplication of power series in $1+t\mbb{C}[[t]]$ and thus, when evaluating on $F$ this is just the usual Euler product expansion of $Z_{q,d}(F)=Z_{V(F) \cap Y}(t)$. Then, \cref{prop.point-counting-asymptotics} implies
\[ \lim_{d \rightarrow \infty} \mbb{E}_1[\Exp_{\sigma}(Z_{q,d}h_1)]=\prod_{P \in |\mbb{P}^n_{\mbb{F}_q}|}\mbb{E}_1[\Exp_{\sigma}(X_P h_1)]=\prod_{P \in |Y|} \mbb{E}_1[\Exp_{\sigma}(X_P h_1)]. \]
Arguing as in the proof of \cref{theorem.point-counting-characters}, we have 
\[ (\Exp_{\sigma}(X_P h_1))_1= \begin{cases} 1+ p_{\deg(P)} \circ (h_1+h_2 + h_3+\ldots) & \textrm{ if $g(P)$=0 } \\
1 & \textrm{otherwise}. 
\end{cases}\]
Indeed, the argument given in the proof of \cref{theorem.point-counting-characters} shows that $(h_k \circ \frac{1}{1-t^{\deg(P)}})_1$ is $1$ if $\deg(P)|k$ and $0$ otherwise, so that in \cref{eq.exp-symmetric-monomial-expansion}, only the terms $m_\tau$ with $\tau$ a multiple of $\deg(P)$ survive, and each with coefficient $1$; this gives exactly the series written. Since $g(P)=0$ with probability $\frac{q^{(n+1)\deg(P)}-1}{q^{(n+2)\deg(P)}-1}$, taking expectation and passing to the product we obtain 
\[ \lim_{d \rightarrow \infty} \mbb{E}_1[\Exp_{\sigma}(Z_{q,d}h_1)] = \prod_{P \in |Y|} \left( 1+ \frac{q^{(n+1)\deg(P)}-1}{q^{(n+2)\deg(P)}-1}p_{\deg(P)}\circ(h_1+h_2+h_3+ \ldots) \right). \] 
This concludes the computation for $Z_{q,d}$. 

To treat $-Z_{q,d}$, we argue in exactly the same way using the sum
\[ -Z_{q,d}=\sum_{P \in |Y|} -X_{P,d}. \]
Then, since, for $P \in |Y|$, 
\[ -X_P(g) = \begin{cases} {1-t^{\deg(P)}} & \textrm{ if $g(P)=0$, i.e. $g \in \mf{m}_P/\mf{m}_P^2$} \\
1 & \textrm{otherwise}. \end{cases}\]
arguing as in the positive case above (cf. also the proof of \cref{theorem.point-counting-characters} where a similar minus sign appears), we find
\[ (\Exp_{\sigma}(-X_P h_1))_1= \begin{cases} 1+ p_{\deg(P)} \circ (-e_1+e_2 - e_3+\ldots) & \textrm{ if $g(P)$=0 } \\
1 & \textrm{otherwise}. 
\end{cases}\]
Passing to expectation and taking the product, we conclude exactly as in the positive case.
\end{proof} 

\begin{remark}\label{remark.proof-via-mrvii}
    We can also deduce the computation of $\lim_{d\rightarrow \infty}\mbb{E}_1[ \Exp_{\sigma}(Z_{q,d}h_1) ]$ in \cref{theorem.geometric-hypersurface-point-counting} from \cite[Theorem C]{Howe.MotivicRandomVariablesAndRepresentationStabilityIIHypersurfaceSections} by an application of \cref{lemma.poisson-and-binomial-distributions} --- indeed, the statement of the result in loc. cit. can be shown to be equivalent to the computation of the falling moment generating function $\mbb{E}[(1+h_1)^{Z_{q,d}}]$ (see \cref{ss.moment-generating-functions}). However, the method we use to prove \cref{theorem.geometric-hypersurface-point-counting} is already closely related to the proof \cite[Theorem C]{Howe.MotivicRandomVariablesAndRepresentationStabilityIIHypersurfaceSections}, just made much better organized and more robust via the use of $\lambda$-probability. Similarly the method of proof in \cref{theorem.geometric-hypersurface-point-counting} could also be applied to directly compute the falling moment generating function and, since $Z_{q,d}$ is a ``combinatorial random variable," the computation of the local factors at closed points $P$ would be even simpler for the falling moments (see \cref{remark.falling-vs-sigma} for more on this along with an explanation of why we did not proceed in this way). 
\end{remark}

\subsection{The geometric random variable motivically}

\begin{theorem}\label{theorem.geometric-hypersurface-motivic}Let $Y \subseteq \mbb{P}^m_{\mbb{F}_q}$ be an $(n+1)$-dimensional smooth closed sub-variety, and let $U_d/\mbb{F}_q$ be the variety of nonzero homogeneous degree $d$ polynomials $F$ such that $V(F)$ intersects $Y$ transversely. Let $Z_d$ be the $W(\mbb{C})$-valued random variable on $U_d(\fqbar)$ sending $F \in U_d(\fqbar)$ to $Z_{V(F) \cap Y}(t)$, where $V(F)\cap Y$ is a variety over $\fq(F)$, the extension of $\fq$ generated by the coefficients of $F$. Then the sequence of random variables $Z_d$ converge in distribution to a binomial distribution with parameters $N=[Y/\mbb{F}_q]$ and $p=\frac{[q]^{n+1}-1}{[q]^{n+2}-1}$ as in \cref{def.binomial-and-poisson}, i.e.  
\[ \lim_{d \rightarrow \infty}\mbb{E}[\Exp_{\sigma}(Z_dh_1)] = \left(1+\frac{[q]^{n+1}-1}{[q]^{n+2}-1}(h_1 + h_2 + h_3 + ...)\right)^{[Y]}\]
Moreover, we also have 
\[ \lim_{d \rightarrow \infty}\mbb{E}[\Exp_{\sigma}(-Z_d h_1)] = 
\left(1+\frac{[q]^{n+1}-1}{[q]^{n+2}-1}(-e_1 + e_2 -e_3 + ...)\right)^{[Y]}.
\]
\end{theorem}
\begin{proof}
    As in the proof of \cref{theorem.motivic-characters}, one computes immediately that $\rest_i Z_d=Z_{q^i,d}$, for the latter the random variable on $U_d(\mbb{F}_q^{i})$ of \cref{theorem.geometric-hypersurface-point-counting} for $\mbb{F}_{q^i}$.  Then, applying \cref{lemma.motivic-to-finite-probability-spaces}, we find 
    \[ \lim_{d \rightarrow \infty}\mbb{E}_i[\Exp_{\sigma}(Z_d h_1)] = \lim_{d\rightarrow \infty}\mbb{E}_1[\Exp_{\sigma}(Z_{q^i,d}h_1)]. \]
    We then find that this is the $i$th component of the power by comparing the description of the limit as an Euler product in \cref{theorem.geometric-hypersurface-point-counting} (keeping in mind we are applying it over $\mbb{F}_{q^i}$) to the description of the $i$th component of the power in \cref{prop.powers-euler-product}. This concludes for $Z_d$; for $-Z_d$ we argue the same way. 
\end{proof}

\subsection{The cohomological random variable} Let $Y \subseteq \mathbb{P}^n_{\mbb{F}_q}$ be a smooth closed irreducible subvariety. For $Z_d$ the $W(\mbb{C})$-valued random variable on $U_d(\fqbar)$ of \cref{theorem.geometric-hypersurface-motivic}, we let $X_d$ be the $W(\mbb{C})$-valued random variable on $U_d(\fqbar)$ defined by 
\begin{equation}\label{eq.vanishing-cohom-rv} X_d=[q^{-n/2}]\left ((-1)^{n} Z_d - [H^n(Y)] - (-1)^n\sum_{i=0}^{n-1} (-1)^i (1+[q^{n-i}]) [H^i(Y)]\right) \end{equation}
By \cref{eq.vanishing-cohom-computation} and the surrounding discussion this is precisely the random variable sending $F$ to the reciprocal roots of $L_{V(F)\cap Y}( t(\#\mbb{F}_q(F))^{-n/2})$, where $V(F)\cap Y$ is viewed as a variety over $\mbb{F}_q(F)$, the extension generated by the coefficients of $F$. 

\begin{remark}
Note that this identification depends crucially on the fact that when we view an element of $W(\mbb{C})$ such as $[H^n(Y)]$ as an element of $C(U_d(\fqbar),W(\mbb{C}))$ it is always via pullback from the point, which includes an application of $p_{\deg(z)} \circ$ at any point $z$, instead of as the literal constant function! (See \cref{example.final-morphism} and the warning immediately afterwards). 
\end{remark}

In particular, $X_d$ is the random variable appearing in the statement of \cref{main.hypersurface-L-functions}. In the present notation, the computation of the limit in \cref{main.hypersurface-L-functions} becomes:
\begin{theorem}\label{theorem.vanishing-cohom-L-functions} With notation as above,
    \[ \lim_{d \rightarrow \infty}\mbb{E}[\Exp_{\sigma}(X_dh_1)]= (1 + p\sum_{i \geq 1} [q^{-in/2}] \epsilon^i f_i)^{[Y/\mbb{F}_q]} \cdot \Exp_\sigma(\mu h_1)  \]
where $f_i=e_i$ and $\epsilon=-1$ if $n$ is odd and $f_i=h_i$ and $\epsilon=1$ if $n$ is even, 
\[ p = \frac{[q]^{n+1}-1}{[q]^{n+2}-1} \textrm{, and } \mu=-\epsilon \left(\sum_{i=0}^{n-1} (-1)^i\left([q^{\frac{-n}{2}}]+[q^{\frac{n-2i}{2}}]\right)[H^i(Y)]\right) - [q^{-n/2}][H^{n}(Y)]. \]
\end{theorem}
\begin{proof}
    We simplify the expression \cref{eq.vanishing-cohom-rv} to
\[  X_d=\epsilon [q^{-n/2}]Z_d +\mu \]
    Since $\mu$ is a constant and constant random variables are independent to any other random variable, by \cref{lemma.independent-moment-generating-functions} we have
\[ \mbb{E}[\Exp_{\sigma}(X_d h_1)]=\mbb{E}\left[\Exp_{\sigma}([q^{-n/2}]\epsilon Z_d)\right] \Exp_{\sigma}(\mu h_1).   \] 
The computation of the limit then follows from \cref{theorem.geometric-hypersurface-motivic} (since scaling a random variable by $[z]$ just changes the degree $d$ part of its $\sigma$-moment generating function by $[z]^d$). \end{proof}

\section{Comparison of $L$-function statistics and random matrix statistics}\label{s.comparison}

In this section we complete the proofs of \cref{main.dirichlet-characters} and \cref{main.hypersurface-L-functions} by comparing the limiting expressions for the $\sigma$-moment generating functions obtained in \cref{theorem.motivic-characters} and \cref{theorem.vanishing-cohom-L-functions} to the random matrix statistics of \cref{main.random-matrix} / \cref{theorem.body-random-matrices} (and \cref{example.standard-rep}). 

As in \cref{remark.sidways-equidistribution}, we note that, in the usual approach to comparing random matrix distributions to $L$-functions as in \cite{KatzSarnak.RandomMatricesFrobeniusEigenvaluesAndMonodromy}, one would first apply Deligne equidistribution to show that, for a fixed degree $d$ and varying over $\mbb{F}_{q^n}$ as $n \rightarrow \infty$, one obtains an equidistribution of conjugacy classes in the monodromy group, and then afterwards take the degree $d$ to $\infty$. In particular, the large $q$ equidistribution turns the second limit in $d$ into a problem purely in the study of random matrices. In our results the order of the limits is exchanged: in \cref{theorem.motivic-characters} and \cref{theorem.vanishing-cohom-L-functions} we have taken the limit as $d \rightarrow \infty$ for a fixed $q$, and we would like to know what happens if afterwards we take the limit over $q^n$ as $n \rightarrow \infty$. 

Now, because we have expressed our results using $W(\mbb{C})$, the limiting formulas actually include the relevant information over $\mbb{F}_{q^n}$ for any $n$ by projection to the $n$th coordinate. A convenient way to express the rate of convergence to random matrix statistics as $q \rightarrow \infty$ is thus to work in the sub-$\lambda$-ring $W(\mbb{C})^\bdd$ consisting of elements whose components are uniformly bounded. In \cref{ss.bounded-witt} we introduce this ring, establish some basic lemmas that will be useful for establishing congruences of moment generating functions, and explain the relation to big $O$-notation. 

In the case of families of id\`{e}le class characters as in \cref{main.dirichlet-characters}, we show in \cref{prop.statistics-comparison-characters} that our large $d$ limits agree with the random matrix distributions modulo $[q^{-1}]\Lambda_{W(\mbb{C})^\bdd}^\wedge$. This is equivalent to saying that, for any symmetric function $f \in \Lambda$, if over $\mbb{F}_{q^n}$ we take the large $d$-limit of the average value of $f$ evaluated on the normalized roots of the $L$-functions in the family, then the result agrees with the relevant random matrix distribution up to $O(q^{-n})$. For vanishing cohomology as in \cref{main.hypersurface-L-functions} we obtain in \cref{prop.statistics-comparison-hypersurface} a congruence modulo $[q^{-\frac{1}{2}}] \Lambda_{W(\mbb{C})^\bdd}^\wedge$, which is equivalent to agreement of the statistics over $\mathbb{F}_{q^n}$ with the random matrix distribution up to $O(q^{-\frac{n}{2}}).$

\subsection{The $\lambda$-ring $W(\mbb{C})^\bdd$}\label{ss.bounded-witt}
Let
\[ W(\mbb{C})^\bdd :=\left\{ (a_1,a_2,\ldots) \in W(\mbb{C})=\prod_{i \geq 1} \mbb{C} \textrm{ s.t. } \exists M>0 \textrm{ s.t. } |a_i| \leq M\, \forall i\right\}.\] 
Then $W(\mbb{C})^\bdd$ is evidently a sub-ring of $W(\mbb{C})$. It is also a sub-$\lambda$ ring: because it is a $\mbb{Q}$-algebra, it suffices to show that it is preserved by power sum plethysms $p_i \circ$. Since $p_i \circ (a_1, a_2, \ldots)=(a_i, a_{2i}, \ldots)$, the boundedness property is preserved. 

The following example shows that congruences can encode big $O$ asymptotics.
\begin{example}
    Suppose $z \in \mathbb{C}$ with $0< |z| < 1$. Then $[z]=(z, z^2, z^3, \ldots)$ is contained in $W(\mbb{C})^\bdd$, but $[z^{-1}]$ is evidently not!  Thus $[z]W(\mbb{C})^\bdd \subseteq W(\mbb{C})^\bdd$ is a proper ideal; we can also view it as a subgroup of $W(\mbb{C})$ under addition. In particular, if $(a_1, a_2, \ldots)$ and $(b_1, b_2, \ldots )$ are elements of $W(\mbb{C})$, then 
    \[ (a_1,a_2, \ldots) \equiv (b_1, b_2, \ldots) \mod [z]W(\mbb{C})^\bdd \]
    if and only if $a_i = b_i + O(z^i)$, i.e. if and only if there exists a constant $M>0$ such that $|a_i - b_i| \leq M|z^i|$ for all $i$ (or, equivalently by changing $M$, for $i\gg 0$).  
\end{example}

The point of stating our big $O$ asymptotics in this way is that it is relatively straightforward to understand how such congruences interact with the plethystic exponential and logarithm. In particular, we have the following analogs of the usual first order approximations for the classical exponential and logarithm: 

\begin{proposition}\label{prop.exp-log-cong}
Let $z \in \mathbb{C}$ with $0< |z| < 1$.
\begin{enumerate}
    \item Suppose $f,g \in W(\mbb{C})^\bdd[[\ul{t}_{\mbb{N}}]]$ have constant term $0$. If 
    \[ f \equiv g \mod [z]\cdot W(\mbb{C})^\bdd[[\ul{t}_{\mbb{N}}]],\]
    then 
    \begin{align*} \Exp_{\sigma}(f) & \equiv \Exp_{\sigma}(g) \mod [z]\cdot W(\mbb{C})^\bdd[[\ul{t}_{\mbb{N}}]] \textrm{ and }\\
    \Log_{\sigma}(1+f) &\equiv \Log_{\sigma}(1+g) \mod [z]\cdot W(\mbb{C})^\bdd[[\ul{t}_{\mbb{N}}]]. \end{align*}
    \item Suppose $a \in W(\mbb{C})^\bdd[[\ul{t}_{\mbb{N}}]]$ has constant term zero, and $a \in [z]\cdot W(\mbb{C})^\bdd[[\ul{t}_{\mbb{N}}]]$. Then 
    \begin{align*} \Exp_{\sigma}(a) & \equiv 1+a \mod [z^2]\cdot W(\mbb{C})^\bdd[[\ul{t}_{\mbb{N}}]] \textrm{ and }\\
    \Log_{\sigma}(1+a) & \equiv a \mod [z^2]\cdot W(\mbb{C})^\bdd[[\ul{t}_{\mbb{N}}]]. \end{align*}
\end{enumerate}
In particular, the same statements hold also in $\Lambda_{W(\mbb{C})^\bdd}^\wedge$ or $(\Lambda \otimes \Lambda)_{W(\mbb{C})}^\wedge.$ 
\begin{proof}
The ``in particular" at the end holds given the first statements because, for any $\lambda$-ring $R$, we can view $\Lambda_{R}^\wedge$ or $(\Lambda \otimes \Lambda)_R^{\wedge}$ as a sub-$\lambda$ ring of $R[[\ul{t}_{\mbb{N}}]]$ as in \cref{s.pre-lambda}. 

    To verify the exponential case of (1), note that
    \[ \Exp_{\sigma}(g)=\Exp_{\sigma}(f)\Exp_{\sigma}({g-f}).\]
    Thus, taking $a=g-f$, it suffices to verify that if $a \in [z]\Lambda_{W(\mbb{C})^\bdd}^\wedge$ has constant term zero then $\Exp_{\sigma}(a) \in 1+ [z]\Lambda_{W(\mbb{C})^\bdd}^\wedge$.  This follows from the exponential case of (2), which we verify now: if we write $a=[z]b$, then using \cref{lemma.exponential-power-sum} we find
    \begin{align*} \Exp_{\sigma}(a)&=\Exp_{\sigma}([z]b)\\
    &=\prod_{i \geq 1} \exp\left(\frac{p_i \circ([z]b)}{i}\right) \\
    &=(1 + a + [z]^2 \frac{b^2}{2} + \ldots)\prod_{i \geq 2}\left(1 + [z]^i\left(\frac{p_i \circ b}{i}\right) + [z]^{2i}\frac{\left(\frac{p_i \circ b}{i}\right)^2}{2} + \ldots\right)\\
    &\equiv (1+a) \mod [z]^2\Lambda_{W(\mbb{C})^\bdd}^\wedge. 
    \end{align*}
    
    To verify the logarithmic case of (1), note that 
    \[ \Log_{\sigma}(1+f) - \Log_{\sigma}(1+g)=\Log_{\sigma}((1+f)/(1+g)).\]
    Since $(1+f)/(1+g)\in 1+ \Lambda_{W(\mbb{C})^\bdd}^\wedge$, by taking $a=(1+f)/(1+g)-1$, it suffices to verify that if $a \in  [z]\Lambda_{W(\mbb{C})^\bdd}^\wedge$ then $\log_{\sigma}(1+a) \in [z]\Lambda_{W(\mbb{C})^\bdd}$. This follows from the logarithmic case of (2), which we verify now: if we write $a=[z]b$, then using \cref{lemma.lambda-log}:
    \begin{align*} \Log_{\sigma}(1+a)=\Log_{\sigma}([z]b) &= \sum_{i \geq 1} \ell_i \circ ([z]b) \\
    &= a + \sum_{i \geq 2} \frac{-1}{i}\sum_{d|i}\mu(d) [z]^i (-p_d\circ b)^{i/d} \\
    &\equiv a \mod [z]^2\Lambda_{W(\mbb{C})^\bdd}^\wedge. 
    \end{align*}
   
\end{proof}
    
\end{proposition}

\subsection{Computations}

\begin{proposition}\label{prop.statistics-comparison-characters}
For $\ell$ a prime, let
\[ c_\ell=\frac{[q^{\ell-1}]}{[q]^{\ell-1}+[q]^{\ell-2} + \ldots + 1} \in W(\mbb{C})^\bdd.\]
Then, 
\begin{itemize}
    \item Taking $\ell=2$,
\[ \left(1 + c_2 \sum_{k\geq 1} [q^{-k}] e_{2 k}\right)^{[q]} \equiv \Exp_{\sigma}(e_2) \mod [q^{-1}]\Lambda_{W(\mbb{C})^\bdd}^\wedge\]
where we note the left-hand side is as in the quadratic limit of $\sigma$-moment generating functions in \cref{main.dirichlet-characters} / \cref{theorem.motivic-characters} and the right-hand side is as in the symplectic case (part (2)) of \cref{main.random-matrix} / \cref{theorem.body-random-matrices}. 
\item For $\ell >2$, 
\begin{multline*}\left(1 + c_\ell \sum_{\substack{k_1\geq 1 \textrm{ or } k_2 \geq 1 \\ k_1 \equiv k_2 \mod \ell}} (-1)^{k_1+k_2}[q^{-\frac{k_1+k_2}{2}}]e_{k_1}\overline{e}_{k_2}\right)^{[q]}\\ \equiv \Exp_{\sigma}(h_1\overline{h}_1) \mod [q^{-1}]\Lambda_{W(\mbb{C})^\bdd}^\wedge \end{multline*}
where we note the left-hand side is as in the non-quadratic limit of joint $\sigma$-moment generating functions in \cref{main.dirichlet-characters} / \cref{theorem.motivic-characters} and the right-hand side is as in the unitary case (part (4)) of \cref{main.random-matrix} / \cref{theorem.body-random-matrices}. 
\end{itemize}
\end{proposition}
\begin{proof}
We first treat the $\ell=2$ case. Then, by the definition of powers (\cref{ss.powers}),
\[ \left(1 + c_2 \sum_{k\geq 1} [q^{-k}] e_{2 k}\right)^{[q]}= \Exp_{\sigma}\left([q]\Log_{\sigma}\left(1 + c_2 \sum_{k\geq 1} [q^{-k}] e_{2 k}\right)\right). \]
Note that $c_2\equiv 1 \mod [q^{-1}]$. Thus $c_2 \sum_{k \geq 1}[q^{-k}] e_{2k} \equiv [q^{-1}]e_2 \mod [q^{-2}].$ 
Using both parts of \cref{prop.exp-log-cong}, we find 
\[ \Log_{\sigma}\left(1 + c_2 \sum_{k\geq 1} [q^{-k}] e_{2 k}\right) \equiv
\Log_{\sigma}(1+[q^{-1}]e_2) \equiv [q^{-1}]e_2 \mod [q^{-2}]. \]
Thus $[q]\Log_{\sigma}\left(1 + c_2 \sum_{k\geq 1} [q^{-k}] e_{2 k}\right)$ is in $W(\mbb{C})^\bdd$ and is $\equiv e_2 \mod [q^{-1}]$. Exponentiating, and applying \cref{prop.exp-log-cong}-(1), we find
\[ \left(1 + c_2 \sum_{k\geq 1} [q^{-k}] e_{2 k}\right)^{[q]} \equiv \Exp_{\sigma}(e_2) \mod [q^{-1}]. \]

We now treat the $\ell > 2$ case. By the definition of powers (\cref{ss.powers}), 
\begin{multline*} \left(1 + c_\ell \sum_{\substack{k_1\geq 1 \textrm{ or } k_2 \geq 1 \\ k_1 \equiv k_2 \mod \ell}} (-1)^{k_1+k_2}[q^{-\frac{k_1+k_2}{2}}]e_{k_1}\overline{e}_{k_2}\right)^{[q]} =\\ \Exp_{\sigma}\left([q]\Log_{\sigma} \left(1 + c_\ell \sum_{\substack{k_1\geq 1 \textrm{ or } k_2 \geq 1 \\ k_1 \equiv k_2 \mod \ell}} (-1)^{k_1+k_2}[q^{-\frac{k_1+k_2}{2}}]e_{k_1}\overline{e}_{k_2}\right)\right). \end{multline*}
Note that $c_\ell \equiv 1 \mod [q^{-1}]$. Thus, 
\[ c_\ell \sum_{\substack{k_1\geq 1 \textrm{ or } k_2 \geq 1 \\ k_1 \equiv k_2 \mod \ell}} (-1)^{k_1+k_2}[q^{-\frac{k_1+k_2}{2}}]e_{k_1}\overline{e}_{k_2} \equiv [q^{-1}]e_1 \overline{e}_1\equiv [q^{-1}]h_1 \overline{h}_1 \mod [q^{-2}]. \]
Using both parts of \cref{prop.exp-log-cong}, we find 
\begin{align*} \Log_{\sigma}\left(1 + c_\ell \sum_{\substack{k_1\geq 1 \textrm{ or } k_2 \geq 1 \\ k_1 \equiv k_2 \mod \ell}} (-1)^{k_1+k_2}[q^{-\frac{k_1+k_2}{2}}]e_{k_1}\overline{e}_{k_2}\right) & \equiv \Log_{\sigma}(1+[q^{-1}]h_1\overline{h}_1) \mod [q^{-2}] \\ &\equiv [q^{-1}]h_1\overline{h}_1 \mod [q^{-2}].\end{align*}
Multiplying by $[q]$ we get a congruence mod $[q^{-1}]$ of elements in $\Lambda_{W(\mbb{C})^\bdd}^\wedge$, and exponentiating and applying \cref{prop.exp-log-cong}-(1), we find
\[ \left(1 + c_\ell \sum_{\substack{k_1,k_2 \geq 1 \\ k_1 \equiv k_2 \mod \ell}} (-1)^{k_1+k_2}[q^{-\frac{k_1+k_2}{2}}]e_{k_1}\overline{e}_{k_2}\right)^{[q]} \equiv \Exp_{\sigma}(h_1 \overline{h}_1) \mod [q^{-1}]. \]
\end{proof}    

\begin{proposition}\label{prop.statistics-comparison-hypersurface}
    Let $Y/\mathbb{F}_q$ be a smooth projective irreducible variety of dimension $n+1$. Consider the quantity appearing as the limit in \cref{main.hypersurface-L-functions} / \cref{theorem.vanishing-cohom-L-functions},
    \[ Q:= (1 + p\sum_{i \geq 1} [q^{-in/2}] \epsilon^i f_i)^{[Y/\mbb{F}_q]} \cdot \Exp_\sigma(\mu h_1) \in \Lambda_{W(\mbb{C})}^\wedge \]
where $f_i=e_i$ and $\epsilon=-1$ if $n$ is odd and $f_i=h_i$ and $\epsilon=1$ if $n$ is even, 
\[ p = \frac{[q]^{n+1}-1}{[q]^{n+2}-1} \textrm{, and } \mu=-\epsilon \left(\sum_{i=0}^{n-1} (-1)^i\left([q^{\frac{-n}{2}}]+[q^{\frac{n-2i}{2}}]\right)[H^i(Y)]\right) - [q^{-n/2}][H^{n}(Y)]. \]
Then $Q \in \Lambda_{W(\mbb{C})^\bdd}^\wedge$, and, 
    \begin{enumerate}
        \item For $n>0$ even, $Q \equiv \Exp_{\sigma}(h_2) \mod [q^{-\frac{1}{2}}]\Lambda_{W(\mbb{C})^\bdd}^\wedge$, where we note the right-hand side is as in the orthogonal case (part (1)) of \cref{main.random-matrix} / \cref{theorem.body-random-matrices}. 
        \item For $n$ odd, $Q \equiv \Exp_{\sigma}(e_2) \mod [q^{-\frac{1}{2}}]\Lambda_{W(\mbb{C})^\bdd}^\wedge$, where we note the right-hand side is as in the symplectic case (part (2)) of \cref{main.random-matrix} / \cref{theorem.body-random-matrices}. 
        \item For $n=0$, $Q \equiv \Exp_{\sigma}(h_2+h_3+\ldots) \mod [q^{-\frac{1}{2}}]\Lambda_{W(\mbb{C})^\bdd}^\wedge$, where we note the right-hand side is as in the action of the symmetric group on the its standard irreducible representation as in \cref{example.standard-rep}. 
    \end{enumerate}
\end{proposition}
\begin{proof}
We have 
\[ \Log_{\sigma}(Q)=  [Y]\Log_{\sigma}(1 + p\sum_{i \geq 1} [q^{-in/2}] \epsilon^i f_i) + \mu. \]
We will analyze this quantity to apply \cref{prop.exp-log-cong}. First, let us note 
\[  p= \frac{[q]^{n+1}-1}{[q]^{n+2}-1} = \frac{[q^{-1}]-[q^{-(n+2)}]}{1-[q^{-(n+2)}]} \equiv [q^{-1}] \mod [q^{-(n+2)}].\]

We first treat the case $n\geq 1$. Then,
\[ p\sum_{i \geq 1} [q^{-in/2}]\epsilon^i f_i \equiv [q^{-1-n/2}]\epsilon f_1 + [q^{-1-n}]f_2 \mod [q^{-(n+3/2)}]. \]
Applying both parts of \cref{prop.exp-log-cong}, we thus find 
\begin{align*} \Log_{\sigma}(1+ p\sum_{i \geq 1} [q^{-in/2}]\epsilon^i f_i) &\equiv \Log_{\sigma}(1+ [q^{-1-n/2}]\epsilon f_1 + [q^{-1-n}]f_2) \\
&\equiv [q^{-1-n/2}]\epsilon f_1 + [q^{-1-n}]f_2 \mod [q^{-(n+3/2)}]. \end{align*}

Note that, by the Weil conjectures, $[q^{-i/2}][H^i(Y)] \in W(\mbb{C})^\bdd$. Furthermore, we have $[H^{2(n+1)}(Y)]=[q^{n+1}]$.  Writing 
\[ [Y]=\sum_i (-1)^i [H^i(Y)] \]
and combining with the above computations we find that 
\[ [Y]\Log_{\sigma}(1+ p\sum_{i \geq 1} [q^{-in/2}]\epsilon^i f_i) \equiv f_2+\left(\epsilon [q^{-1-n/2}] \sum_{i=n+2}^{2n+2}[H^i(Y)]\right)f_1 \mod [q^{-1/2}] \Lambda_{(W(\mbb{C})^\bdd}^\wedge \]
Rewriting $[H^{n+1+k}(Y)]=[q^k][H^{n+1-k}(Y)]$, we can rewrite this as
\[ [Y]\Log_{\sigma}(1+ p\sum_{i \geq 1} [q^{-in/2}]\epsilon^i f_i) \equiv f_2+ \left(\epsilon \sum_{i=0}^{n}(-1)^i[q^{n/2-i}][H^i(Y)]\right)f_1 \mod [q^{-1/2}] \Lambda_{W(\mbb{C})^\bdd}^\wedge. \]
On the other hand, by the definition of $\mu$, we have
\[ \mu \equiv -\epsilon\left(\sum_{i=0}^{n-1}(-1)^i [q^{n/2-i}][H^i(Y)]\right) -[q^{-n/2}][H^n(Y)] \mod [q^{-1/2}] \Lambda_{W(\mbb{C})^\bdd}. \]
Combining, we find 
\[ \Log_{\sigma}(Q)=[Y]\Log_{\sigma}(1+ p\sum_{i \geq 1} [q^{-in/2}]\epsilon^i f_i) + \mu h_1 \equiv f_2  \mod [q^{-1/2}] \Lambda_{(W(\mbb{C})^\bdd}. \]
In particular, $\Log_{\sigma}(Q) \in W(\mbb{C})^\bdd$, and thus by \cref{prop.exp-log-cong}-(1) we can exponentiate to obtain
\[ Q \equiv \Exp_{\sigma}(f_2) \mod [q^{-1/2}] \Lambda_{(W(\mbb{C})^\bdd}.\]

It remains to treat the case $n=0$. Note that, in this case, $\mu=-[H^0(Y)]=-1$, $\epsilon=1$, and $f_i=h_i$. In particular
\[  \Log_{\sigma}(Q)=  [Y]\Log_{\sigma}(1 + p\sum_{i \geq 1} h_i) - h_1. \]
Since $p\equiv [q^{-1}] \mod [q^{-2}]$, applying both parts of \cref{prop.exp-log-cong} we have 
\[ \Log_{\sigma}(1+p\sum_{i \geq 1} h_i) \equiv \Log_{\sigma}(1+[q^{-1}]\sum_{i \geq 1} h_i) \equiv [q^{-1}]\sum_{i \geq 1} h_i \mod [q^{-2}]. \]
Then, expanding $[Y]=[H^0(Y)]-[H^1(Y)]+[H^2(Y)]=[1]-[H^1(Y)]+[q]$ and using the bounds on  $[H^1(Y)]$ as in the case above, we find 
\[ \Log_{\sigma}(Q)=[Y]\Log_{\sigma}(1+p\sum_{i \geq 1} h_i) - h_1\equiv \sum_{i \geq  2} h_i \mod [q^{-1/2}]\Lambda_{W(\mbb{C})^\bdd}^\wedge. \]
In particular, $\Log_{\sigma}(Q)$ is in $W(\mbb{C})^\bdd$, and exponentiating and applying \cref{prop.exp-log-cong}-(1), we obtain
\[ Q \equiv \Exp_\sigma\left(\sum_{i \geq 2} h_i\right) \mod [q^{-1/2}]\Lambda_{W(\mbb{C})^\bdd}^\wedge. \]  
\end{proof}

\bibliographystyle{plain}
\bibliography{references, preprints}

\end{document}